\documentclass[a4paper, 11pt]{amsart}
\pdfoutput=1
\DeclareRobustCommand{\subtitle}[1]{#1}

\title{Formal category theory in $\infty$-equipments II\subtitle{:\\\small Lax functors, monoidality and fibrations}}
\author{Jaco Ruit}
\address{Mathematisch Instituut, Universiteit Utrecht, The Netherlands}
\email{j.c.ruit@uu.nl} 

\usepackage{amssymb}
\usepackage{amsfonts}
\usepackage{amsthm}

\usepackage{kpfonts} 
\usepackage[mathscr]{eucal}

\usepackage[english]{babel}

\usepackage[left=3.2cm,right=3.2cm,top=3cm,bottom=3cm]{geometry} 

\usepackage{tikz-cd}{}  
\usepackage{enumitem}
\usepackage{hyperref}
\usepackage{aliascnt}
\usepackage{adjustbox}
\usepackage{varioref}
\usepackage{xstring}
\usepackage{xcolor}
\usepackage{bbm}

\newtheorem{thmintro}{Theorem}

\newcommand{\inewtheorem}[2]{
	\newaliascnt{#1}{thmintro}
	\newtheorem{#1}[#1]{#2}
	\aliascntresetthe{#1}
}
\inewtheorem{propintro}{Proposition}
\inewtheorem{corintro}{Corollary}

\newtheorem{theorem}{Theorem}
\newcommand{\jnewtheorem}[2]{
	\newaliascnt{#1}{theorem}
	\newtheorem{#1}[#1]{#2}
	\aliascntresetthe{#1}
}
\numberwithin{theorem}{section}
\jnewtheorem{lemma}{Lemma}
\jnewtheorem{proposition}{Proposition}
\jnewtheorem{corollary}{Corollary}
\jnewtheorem{conjecture}{Conjecture}

\theoremstyle{definition}
\jnewtheorem{definition}{Definition}
\jnewtheorem{example}{Example}
\jnewtheorem{remark}{Remark}
\jnewtheorem{question}{Question}
\jnewtheorem{fact}{Fact}
\jnewtheorem{construction}{Construction}
\jnewtheorem{notation}{Notation}
\jnewtheorem{convention}{Convention}

\labelformat{theorem}{Theorem #1}
\labelformat{lemma}{Lemma #1}
\labelformat{proposition}{Proposition #1}
\labelformat{corollary}{Corollary #1}
\labelformat{definition}{Definition #1}
\labelformat{example}{Example #1}
\labelformat{section}{Section #1}
\labelformat{subsection}{Subsection #1}
\labelformat{equation}{(#1)}
\labelformat{remark}{Remark #1}
\labelformat{exercise}{Exercise #1}
\labelformat{question}{Question #1}
\labelformat{chapter}{Chapter #1}
\labelformat{construction}{Construction #1}
\labelformat{conjecture}{Conjecture #1}
\labelformat{notation}{Notation #1}
\labelformat{convention}{Convention #1}

\labelformat{thmintro}{Theorem #1}
\labelformat{corintro}{Corollary #1}
\labelformat{propintro}{Corollary #1}

\numberwithin{table}{subsection}
\labelformat{table}{Table #1}

\hypersetup{
	pdftitle={Formal category theory in infinity equipments II},
	pdfauthor={Jaco Ruit},
	bookmarksnumbered=true,     
	bookmarksopen=true,         
	bookmarksopenlevel=1,       
	colorlinks=true,   
	linkcolor = {cyan!60!black},   
	urlcolor = {cyan!60!black},   
	citecolor = {cyan!60!black},   
	pdfencoding=unicode,      
	pdfstartview=Fit,           
	pdfpagemode=UseOutlines,    
	pdfpagelayout=OneColumn,
	breaklinks=true,
	linktocpage
}

\newcommand{\cat}[1]{\textup{#1}}
\newcommand{\op}{\mathrm{op}}
\newcommand{\map}{\mathrm{Map}}
\def\colim{\qopname\relax m{colim}}
\newcommand{\fun}{\mathrm{Fun}}
\newcommand{\FUN}{\mathrm{FUN}}
\newcommand{\Cat}{\cat{Cat}}
\newcommand{\CAT}{\cat{CAT}}

\newcommand{\lax}{\mathrm{lax}}

\renewcommand{\S}{\mathscr{S}}
\newcommand{\C}{\mathscr{C}}
\renewcommand{\D}{\mathscr{D}}
\renewcommand{\P}{\mathscr{P}}
\newcommand{\Q}{\mathscr{Q}}
\newcommand{\E}{\mathscr{E}}
\newcommand{\X}{\mathscr{X}}
\newcommand{\Y}{\mathscr{Y}}

\renewcommand{\Vert}{\mathrm{Vert}}
\newcommand{\Hor}{\mathrm{Hor}}
\newcommand{\F}{\mathscr{F}}
\newcommand{\Seg}{\mathrm{Seg}}

\newcommand{\tp}{\mathrm{t}}
\newcommand{\DblCat}{\cat{DblCat}}
\newcommand{\Equip}{\cat{Equip}}
\newcommand{\ev}{\mathrm{ev}}

\newcommand{\PSh}{\cat{PSh}}
\newcommand{\Sq}{\cat{Sq}}	
\newcommand{\id}{\mathrm{id}}
\newcommand{\cart}[3]{{#1}^\circledast{#2}{#3}_\circledast}
\newcommand{\cocart}[3]{{#1}_\circledast{#2}{#3}^\circledast}

\newcommand{\vop}{\mathrm{vop}}
\newcommand{\Con}{\mathrm{Con}}
\newcommand{\CCAT}{\mathbb{C}\mathrm{at}}
\newcommand{\SSPAN}{\mathbb{S}\mathrm{pan}}
\newcommand{\FFUN}{\mathbb{F}\mathrm{un}}
\newcommand{\DDBLCAT}{\mathbb{D}\mathrm{blCat}}
\newcommand{\hop}{\mathrm{hop}}

\newcommand{\CCON}{\mathbb{C}\mathrm{on}}

\begin{document}

\begin{abstract}
	We study the framework of $\infty$-equipments which is designed to produce well-behaved theories for different generalizations of
	$\infty$-categories in a synthetic and uniform fashion. We consider notions of (lax) functors between these equipments, closed monoidal structures on these equipments, and 
	fibrations internal to these equipments.  As a main application, we will demonstrate that the foundations of internal $\infty$-category theory can be readily obtained using this formalism. 
\end{abstract}

\maketitle
\renewcommand{\subtitle}[1]{}

\setcounter{tocdepth}{1}
\tableofcontents
	
\section{Introduction}
There are many useful generalizations of $\infty$-categories such as enriched, indexed, internal or fibered versions. These variants have tailored notions 
of (co)limits, adjunctions, point-wise Kan extensions and more, that (generally) differ from the corresponding notions for their underlying $\infty$-categories. 
Consequently, one would need to redevelop the theory for each of the variants. For example, Heine recently worked out a very complete theory of enriched $\infty$-categories \cite{Heine}. 
Weighted colimits of enriched $\infty$-categories were also studied by Hinich \cite{Hinich}.
Indexed $\infty$-category theory, also called \textit{parametrized} $\infty$-category theory, has been introduced by Barwick, Dotto, Glasman, Nardin and Shah \cite{ExposeI} and was further studied 
by Shah in \cite{Shah} and \cite{Shah2}. This theory has now been subsumed by the theory of  $\infty$-categories internal to a general $\infty$-topos 
that has been studied 
by Rasekh \cite{Rasekh}, and Martini and Wolf \cite{Martini}, \cite{MartiniWolf}. 

The goal of this work is to (further) develop the $\infty$-categorical framework, introduced in \cite{EquipI}, that is 
designed to produce these theories for the different variants of $\infty$-categories all at once. This is a generalization 
of a program set out in the context of (strict) categories \cite{StreetWalters}, \cite{Wood}, \cite{Verity}, \cite{ShulmanFramedBicats}.  The key role in the $\infty$-categorical formalism 
is played by the \textit{$\infty$-equipments} \cite[Subsection 3.2]{EquipI}. These are certain double $\infty$-categorical structures that support a good notion of 
an internal \textit{formal category theory}. This includes notions of weighted (co)limits, point-wise Kan extensions and more, see \cite[Section 6]{EquipI}. 
In this sense, every $\infty$-equipment yields \textit{a category theory} for its objects. 
A prototypical example of an $\infty$-equipment is given by the double $\infty$-category $$\CCAT_\infty$$ of $\infty$-categories, with 
functors and profunctors as vertical and 
horizontal arrows respectively \cite[Section 4]{EquipI}. The formal category theory associated with this equipment 
yields the well-known theory of $\infty$-categories. More generally, 
Haugseng \cite{HaugsengEnriched} has constructed a candidate $\infty$-equipment of enriched $\infty$-categories. We expect that its formal category theory captures the correct notions 
for enriched $\infty$-category theory.

We have also constructed an $\infty$-equipment $$\CCAT_\infty(\E)$$ 
of $\infty$-categories internal to an $\infty$-topos $\E$ in \cite[Section 5]{EquipI}. 
As evidence of the power of the equipmental approach, we will demonstrate in \ref{section.int-cat-theory} that one can obtain the foundations of internal category theory 
by interpreting the formal category theory associated with the equipment $\CCAT_\infty(\E)$. As mentioned earlier, internal category theory was previously studied by Martini and Wolf \cite{MartiniWolf} by abstracting Lurie's work \cite{HTT} so that 
it may be interpreted in any $\infty$-topos. For presheaf $\infty$-toposes (the \textit{indexed} case), this has been studied by Shah \cite{Shah} before. 
Both these approaches are fundamentally different from the one that we present in this article. The equipmental approach 
has also been used by Shulman in \cite{ShulmanEnInCats} where strict equipments appear to study indexed (more generally, enriched indexed) categories.

The main goal of this paper is to expand the toolbox of $\infty$-equipments so that we may express more aspects of category theory inside $\infty$-equipments. This will be necessary to 
give the complete foundations of internal category theory later. To this end, we will also need to develop 
more aspects of double $\infty$-categorical theory that is not yet present in the current literature.
This work crucially builds upon the material that has been developed in the prequel paper \cite{EquipI} 
and material from \cite{FunDblCats}. We will 
review some parts of this in \ref{section.prelims}.

\subsection{Lax functors} In \ref{section.lax-fun}, we extend the notion of (normal) lax functors between double categories of Grandis and Par\'e \cite{GrandisPareLimits} to double $\infty$-categories. 
Such an extension was recently also considered by Abell\'an \cite{Abellan}, albeit using a slightly different set-up than we present here. 
We will construct a double $\infty$-category $\DDBLCAT_\infty$ of double $\infty$-categories in the process. Its vertical arrows are functors and its horizontal arrows 
are lax functors. Much of the theory about lax functors between double categories goes through in the $\infty$-categorical setting. For instance, in \ref{ssection.lax-fun-cart}, we will 
generalize Shulman's result that asserts that lax functors between equipments preserve cartesian cells \cite[Proposition 6.8]{ShulmanFramedBicats}, with essentially the same proof.

We also study the notion of a lax adjunction between lax functors. The main instance of such an adjunction (covered in \ref{ex.dbl-adj-geom-morph}) occurs as follows: 
every geometric morphism $f : \E \rightarrow \F$ between $\infty$-toposes gives rise to a lax adjunction 
$$
f^* : \CCAT_\infty(\F) \rightleftarrows \CCAT_\infty(\E) : f_*.
$$
Whenever $f$ is a geometric embedding, i.e.\ when the direct image functor 
of $f$ is fully faithful, then we will see that 
$f_*$ is a double $\infty$-categorical analog of an inclusion of a reflective subcategory. We will 
develop this notion of a \textit{reflective lax inclusion} independently, as well as a \textit{local} variant, in \ref{ssection.refl-lax-functors}.

Moreover, we introduce a notion of \textit{proper} and \textit{smooth} lax functors between $\infty$-equipments in \ref{ssection.proper-smooth}. These are functors 
that preserve the left (resp.\ right) Kan extensions internal to the formal category theories (see \cite[Subsection 6.3]{EquipI}) of the $\infty$-equipments.
Our main examples of these occur as \textit{base change} functors associated to suitable $\infty$-equipments (see \ref{ssection.base-change}), and this 
gives rise to a notion of proper and smooth arrows in these $\infty$-equipments. This generalizes Lurie's notions of proper and smooth functors \cite[Subsection 4.1.2]{HTT}.

\subsection{Monoidal structures} Now that we have access to a notion of lax adjunctions, we can make sense of \textit{monoidal closed} double $\infty$-categories in \ref{section.mon-equipments}. In the (strict) double categorical literature, there 
are various definitions of monoidal double categories, see for instance the definition of Grandis-Par\'e \cite{GrandisPare} or the one of Hansen-Shulman \cite{HansenShulman}. Our definition will be a generalization of the last one; we define 
monoidal double $\infty$-categories to be monoids in the $\infty$-category of double $\infty$-categories. The \textit{closure} condition is then phrased as the existence 
 of a normal lax right adjoint to the functor that tensors with a fixed object (see \ref{def.mon-closed}). This right adjoint supplies the monoidal double $\infty$-category 
with a good notion of \textit{internal homs} that have double categorical functoriality.

We may then study the $\infty$-equipments that admit such a closed monoidal structure. There is a relatively easy way of 
obtaining such structures by assembling level-wise cartesian monoidal structures, see \ref{prop.mon-cartesian}. We will use this to demonstrate that $\CCAT_\infty(\E)$ admits a (cartesian) closed monoidal structure. 
Whenever one has an arrow $w : i \rightarrow j$ and an object $x$ in a closed monoidal $\infty$-equipment $\P$,
one can ask for the existence of a left adjoint to the restriction arrow
$$
w^* : [j,x] \rightarrow [i,x]
$$
in the underlying vertical $(\infty,2)$-category of $\P$. We will see in \ref{ssection.pke-arrows} that this problem can be studied 
using the point-wise left Kan extensions along $w$ in the formal category theory of $\P$ (see \cite[Definition 6.19]{EquipI}),  
similarly to ordinary ($\infty$-)category theory.

\subsection{Tabulations and fibrations} A large portion of the paper is devoted to developing an abstract notion of \textit{two-sided discrete fibrations} internal to double $\infty$-categories, 
and in particular, internal to $\infty$-equipments. This will be the goal of \ref{section.fibrations}. These fibrations are particular spans in the vertical fragment of a double $\infty$-category that generalize two-sided discrete fibrations between $\infty$-categories (see \ref{def.tsdfib}). 
Our notion is an adaptation of a 2-categorical development by Street \cite{StreetFib}
to a double $\infty$-categorical setting, and we will draw parallels between the two approaches in \ref{rem.tsdfib-street}. 
In our context, the definition of these fibrations is phrased using the double $\infty$-categorical variant of Grandis' and Pare\'s notion of \textit{tabulations} or \textit{tabulators} \cite{GrandisPareLimits} 
of horizontal arrows.
The double $\infty$-categories that admit all tabulations for horizontal arrows are called \textit{tabular}.
In such a double $\infty$-category, every horizontal arrow has an associated two-sided discrete fibration.
The main result of \ref{section.fibrations} is the verification 
that our two-sided discrete fibrations are so `representably' as well:

\begin{thmintro}\label{thmintro.tsdfib-rep}
	Suppose that $\P$ is a tabular double $\infty$-category.
	Let $(p,q) : e \rightarrow x \times y$ be a two-sided discrete fibration in the vertical $(\infty,2)$-category $\Vert(\P)$ of $\P$. Then $(p,q)$ induces a two-sided discrete 
	fibration of $\infty$-categories 
	$$
	(p,q)_* : \Vert(\P)(z,e) \rightarrow \Vert(\P)(z,x) \times \Vert(\P)(z,y)
	$$
	for every $z\in\P$.
\end{thmintro}

\noindent In the paper, this result will be obtained as a corollary of \ref{thm.rep-tsdfib}. The converse of \ref{thmintro.tsdfib-rep} is not always true. The two-sided discrete fibrations rely on the horizontal fragments of double $\infty$-categories as well. 

In \ref{section.fib-equipments}, we continue to study the tabular $\infty$-equipments for which the passage from horizontal arrows 
to associated two-sided discrete fibrations is an embedding, in a particular sense. 
We first show that one may always construct a \textit{span representation} for a tabular $\infty$-equipment, generalizing work of 
Grandis-Par\'e \cite{GrandisPareSpan} and Niefield \cite{Niefield} for strict double categories.

\begin{thmintro}\label{thmintro.span-rep}
	Suppose that $\P$ is a tabular double $\infty$-category with pullbacks. Then there exists a lax 
	functor 
	$$
	\rho : \P \rightarrow \SSPAN(\Vert(\P)^{(1)})
	$$
	from $\P$ to the double $\infty$-category of spans in the vertical $\infty$-category $\Vert(\P)^{(1)}$ of $\P$.
	This functors acts as the identity on the underlying vertical $\infty$-categories and carries 
	horizontal arrows to their associated two-sided discrete fibrations.
\end{thmintro}

\noindent This theorem appears as \ref{thm.span-rep}. Its proof makes use of a universal property of the double $\infty$-category of spans, exhibited by Haugseng in \cite{HaugsengSpans2}.

We may then characterize the tabular $\infty$-equipments $\P$ that have the property that the associated span representations 
$\rho$ are \textit{horizontally locally reflective}, in the sense that the induced functors 
$$
\rho_{x,y} : \Hor(\P)(x,y) \rightarrow \Vert(\P)^{(1)}/(x \times y) = \Hor(\SSPAN(\Vert(\P)^{(1)}))(x,y)
$$
on horizontal mapping $\infty$-categories are inclusions of reflective subcategories for $x,y \in \P$.
This characterization is also supplied by \ref{thm.span-rep}. The $\infty$-equipments that 
have this property will be called \textit{fibrational}. We will demonstrate that $\CCAT_\infty(\E)$ is fibrational for every $\infty$-topos $\E$ (see \ref{prop.int-ccat-fib}). 

Using the span representation, we develop the following theory for a fibrational $\infty$-equipments:
\begin{itemize}
	\item We show that two-sided discrete fibrations are closed under base change (\ref{prop.tsdfib-closed-pb}), and identify the \textit{comma fibrations} that classify certain horizontal arrows (\ref{def.commas}).
	\item We prove a fibrational variant of the Yoneda lemma (\ref{prop.fib-yoneda}). In a 2-categorical context, such a result 
	was proven by Street in \cite[Corollary 16]{StreetFib}. For $\infty$-cosmoses, there is a similar result proven by Riehl-Verity \cite[Theorem 5.7.1]{RiehlVerity}.
	\item For a \textit{pointed} (see \cite[Definition 6.30]{EquipI}) and fibrational $\infty$-equipment, we introduce \textit{left} and \textit{right fibrations} (\ref{def.left-right-fib}).
	\item For a pointed and cartesian closed fibrational $\infty$-equipment, we introduce \textit{slice and coslice objects} (\ref{def.slice-fibs}). We use this to show that certain weighted colimits can be computed 
	as conical limits  (\ref{cor.fib-weighted-vs-conical}). This generalizes a result of Haugseng \cite[Theorem 1.2]{HaugsengEnds} for $\infty$-categories. As a corollary, we explain that 
	left Kan extensions can be computed by conical colimits, similar to the case for ordinary $\infty$-categories (\ref{cor.fib-pke}).
	\item In \ref{ssection.compr-fact}, we characterize the pointed and fibrational $\infty$-equipments that admit \textit{comprehensive factorizations}. In such an equipment, every arrow 
	may be factored (essentially uniquely) as an initial arrow followed by a left fibration. We show that $\CCAT_\infty(\E)$ has these factorizations. In such equipments, coconical colimits may be computed 
	as initial objects in coslices (\ref{cor.colim-initial-slice}).
\end{itemize}

The span representations may also be used to obtain a characterization of the double $\infty$-categories of spans of \cite{HaugsengSpans}.
This may be viewed as double $\infty$-categorical version of a result established by Carboni, Kasangian, and Street \cite[Theorem 4]{CarboniKasangianStreet} about recognizing span bicategories.

\begin{corintro}
	Let $\P$ be a double $\infty$-category. Then the following are equivalent:
	\begin{enumerate}
		\item $\P$ is equivalent to a double $\infty$-category of spans in an $\infty$-category $\C$  with pullbacks,
		\item $\P$ is a fibrational $\infty$-equipment and  every span in $\Vert(\P)^{(1)}$ is a two-sided discrete fibration (with respect to $\P$),
		\item the span representation of \ref{thmintro.span-rep} is a strict functor and an equivalence.
	\end{enumerate}
\end{corintro}

\noindent This appears as \ref{cor.char-span-dbl-cats} in the article.

\subsection{Applications to internal category theory}\label{ssection.intro-int-ccat}

Finally, in \ref{section.int-cat-theory}, we will give the promised application of the formalism developed in this paper. Here, we 
interpret the formal category theory in the $\infty$-equipment $$\CCAT_\infty(\E)$$ of 
$\infty$-categories internal to a fixed $\infty$-topos  $\E$ to recover the fundamentals of internal category theory.
The content and strategy of \ref{section.int-cat-theory} may be briefly summarized as follows:
\begin{itemize}
\item As a first step, we will characterize the companions and conjoints in $\CCAT_\infty(\E)$ in \ref{thm.comp-conj-topos}. To prove this, we will make use of the characterization of \cite[Theorem 4.6]{FunDblCats}
of companionships and conjunctions in functor double $\infty$-categories. This will be crucial to understand the formal category theory.
\item We  show that the two-sided discrete fibrations in $\CCAT_\infty(\E)$ are precisely the ones that are representably so, providing 
a converse to \ref{thmintro.tsdfib-rep} in this particular case. 
\item We show that left and right fibrations are proper and smooth respectively; this was shown earlier by Martini \cite[Proposition 4.4.7]{Martini}. 
\item We provide an internal variant of Quillen's theorem A (\ref{prop.int-ccat-quilen-A}), see also \cite[Corollary 4.4.8]{Martini}.
\item We show that \textit{colimits weighted by right fibrations} may be computed as conical colimits. In turn, these may be viewed as initial objects in certain internal coslice categories (see \ref{cor.int-ccat-weighted-rfib}). 
This proves that the notions of colimits  in $\CCAT_\infty(\E)$ coincides with that of Martini-Wolf \cite{MartiniWolf} (\ref{cor.int-ccat-colim-mw}).
\item We give the existence conditions and computation of left Kan extensions internal to $\CCAT_\infty(\E)$ in terms of colimits, essentially as known from classical category theory (see \ref{prop.internal-pke}). This also appears in \cite{Shah} (for a special case).
\item We give existence conditions for left adjoints to internal restriction functors in terms of left Kan extensions (\ref{prop.internal-pke-functor}). This was shown by Shah \cite{Shah} (again, in a special case) and Martini-Wolf \cite{MartiniWolf} as well.
\end{itemize}

We stress that the relevant double $\infty$-category $\CCAT_\infty(\E)$ is all the input structure that we use to develop these aspects.

\subsection*{Conventions} We will use the language of $\infty$-categories as developed by Joyal and Lurie. Moreover, we  use the following notational conventions:
\begin{itemize}
	\item The  $\infty$-categories of spaces (i.e.\ $\infty$-groupoids) and $\infty$-categories are denoted by $\S$ and $\Cat_\infty$ respectively. 
	\item For an $\infty$-category $\C$, we write $\PSh(\C) := \fun(\C^\op, \S)$ for the $\infty$-category of presheaves on $\C$.
	\item We will use uppercase notation for $(\infty,2)$-categorical upgrades of particular $\infty$-categories. The double $\infty$-categorical variants will be decorated with blackboard bolds.
	For instance, in this article, we will use both the $(\infty,2)$-category $\CAT_\infty$ and the double $\infty$-category $\CCAT_\infty$ of $\infty$-categories.
	\item We will suppress the notation of an arrow, such as `$\Rightarrow$', when denoting 2-cells in double $\infty$-categories. For example, the square 
	\[
		\begin{tikzcd}
			a \arrow[d,"f"']\arrow[r,"F"] & b \arrow[d, "g"] \\ 
			c \arrow[r,"G"] & d
		\end{tikzcd}
	\] 
	denotes a 2-cell in a double $\infty$-category with $F, G$ and $f,g$ the corresponding source-target horizontal arrows and vertical arrows respectively.
\end{itemize}

\subsection*{Acknowledgements}
Firstly, I would like to thank Lennart Meier for the numerous helpful comments on and discussions about the material of this paper.
Moreover, I express my gratitude towards Jay Shah for the insightful conversations about parametrized category theory that we had in M\"unster. 
I thank Emily Riehl for the fruitful discussions about parts of the theory presented in this article.
During the writing of this paper, the author was funded by the Dutch Research Council (NWO) through the grant ``The interplay of orientations and symmetry'', grant no. OCENW.KLEIN.364.

\section{Preliminaries}\label{section.prelims}
Our first section is dedicated to reviewing the relevant material in \cite{EquipI} and \cite{FunDblCats} that will be used throughout this article. Moreover, 
we take this opportunity to set up notation and introduce conventions.

\subsection{Double \texorpdfstring{$\infty$}{∞}-categories}
In \cite{EquipI} and \cite{FunDblCats}, we have studied the following localizations of the $\infty$-category $\PSh(\Delta^{\times 2})$ of bisimplicial spaces:

\begin{table}[ht]
	\centering
	\caption{Relevant notions of double and 2-fold Segal spaces}\label{table.dss}
\renewcommand{\arraystretch}{1.2}
\begin{tabular}{| c c c | c |}
	\hline
	\multicolumn{3}{|c|}{\textbf{Reflective subcategory}} & \textbf{Local objects} \\
	\hline 
	$\Cat^2(\S)$ & $\subset$ & $\PSh(\Delta^{\times 2})$ & \textit{Double Segal spaces}\\
	\hline
	$\DblCat_\infty$ & $\subset$ & $\Cat^2(\S)$ & \textit{Double $\infty$-categories} \\
	\hline
	$\DblCat_\infty^{lc}$ & $\subset$ & $\DblCat_\infty$ & \textit{Locally complete double $\infty$-categories} \\
	\hline 
	$\Seg^2(\S)$ & $\subset$ & $\PSh(\Delta^{\times 2})$ & \textit{2-Fold Segal spaces} \\
	\hline 
	$\Seg^2(\S)^{lc}$ & $\subset$ & $\Seg^2(\S)$ & \textit{Locally complete 2-fold Segal spaces} \\
	\hline 
	$\Cat_{(\infty,2)}$ & $\subset$ & $\Seg^2(\S)^{lc}$ & \textit{$(\infty,2)$-Categories} \\
	\hline
\end{tabular}
\end{table}

\noindent For the definitions, we refer to the concise summary given in \cite[Section 2]{FunDblCats}. 

\begin{convention}
	The image of $([n], [m])$ under the Yoneda embedding $$\Delta^{\times 2} \rightarrow \PSh(\Delta^{\times 2})$$ will be denoted by 
	$[n,m]$. This is a (locally complete) double $\infty$-category. We refer to the objects $[0,0], [1,0], [0,1]$ and $[1,1]$ as the \textit{free living object (0-cell)}, the \textit{free living horizontal} and \textit{vertical arrow (1-cell)},  
	and the \textit{free living 2-cell}, respectively. 
\end{convention}

\begin{example}
	In \cite{HaugsengSpans}, Haugseng constructs the double $\infty$-category of spans $$\SSPAN(\C)$$ in an $\infty$-category $\C$ with pullbacks. 
	Its objects and vertical arrows correspond to object and morphisms in $\C$ respectively, and its horizontal arrows are given by spans of morphisms in $\C$.
\end{example}

\begin{notation}
	There are three relevant involutions on $\Cat^2(\S)$:
	\begin{enumerate}
		\item the \textit{transpose functor} $$(-)^\tp : \Cat^2(\S) \rightarrow \Cat^2(\S),$$ induced by the involution 
		$\Delta \times \Delta \rightarrow \Delta \times \Delta : ([n],[m]) \mapsto ([m], [n])$, 
		\item the \textit{horizontal opposite functor} $$(-)^\hop : \Cat^2(\S) \rightarrow \Cat^2(\S)$$ induced by the involution 
		$\Delta \times \Delta \xrightarrow{\op \times \id} \Delta \times \Delta$, 
		\item the \textit{vertical opposite functor} $$(-)^\vop : \Cat^2(\S) \rightarrow \Cat^2(\S)$$ induced by the involution 
		$\Delta \times \Delta \xrightarrow{\id \times \op} \Delta \times \Delta$.
	\end{enumerate}
	See also \cite[Definition 2.6]{EquipI}.
	Note that (2) and (3) restrict to involutions on $\DblCat_\infty$ and $\DblCat_\infty^{lc}$, but $(-)^\tp$ does not.
\end{notation}

\begin{convention}
	One may realize $\DblCat_\infty$ as a reflective subcategory of $\fun(\Delta^\op, \Cat_\infty)$ as well.
	Namely, a double $\infty$-category $\P$ can be viewed as a simplicial $\infty$-category $\P'$ so that 
	$$
	\map_{\Cat_\infty}([m], \P'_n) \simeq \map_{\DblCat_\infty}([n,m], \P),
	$$
	naturally for $[n],[m] \in \Delta$. This has been discussed in \cite[Section 2.1]{EquipI}. We will use 
	this perspective on many 
	occasions throughout this article, and the described inclusion $$\DblCat_\infty \rightarrow \fun(\Delta^\op, \Cat_\infty)$$ will be left implicit.
\end{convention}

\begin{notation}
	Recall that a (strict) 2-category is called \textit{gaunt} if every invertible 1-cell and 2-cell is an identity. 
	We will write $\mathrm{Gaunt}_2$ for the full subcategory of the category of 2-categories spanned by the gaunt ones. 
	When we restrict to this subcategory, there 
	is an easy-to-construct fully faithful functor $$\mathrm{Gaunt}_2 \rightarrow \Seg^2(\S),$$ see \cite[Construction 3.4]{FunDblCats}. We will 
	leave this inclusion implicit. Moreover, we write 
	$$[n; m_1, \dotsc, m_n] \in \mathrm{Gaunt}_2$$
	for the \textit{Joyal disk} associated to integers $n,m_1, \dotsc, m_n \geq 0$.
\end{notation}

\begin{notation}
	There are reflective inclusions 
	$$
	(-)_h, (-)_v : \Seg^2(\S) \rightarrow \Cat^2(\S)
	$$
	which both admit right adjoints 
	$$
	\Hor(-), \Vert(-) : \Cat^2(\S) \rightarrow \Seg^2(\S)
	$$
	respectively. This has been discussed in \cite[Subsection 2.2]{EquipI} and \cite[Notation 3.9]{FunDblCats}. Consequently, there are two ways to extract a 2-fold Segal space from a double Segal space $\P$: by 
	taking its \textit{horizontal fragment} $\Hor(\P)$ or \textit{vertical fragment} $\Vert(\P)$. There are relations between 
	the above functors, namely,
	$$
	X_v = X_h^{\tp, \hop} \quad \text{ and } \quad \Vert(\P) = \Hor(\P^{\hop, \tp})
	$$
	for a 2-fold Segal space $X$ and a double Segal space $\P$.
\end{notation}

\begin{notation}
	Let $x, y$ be objects of a 2-fold Segal space $X$. Then we will write $$X(x,y)$$ 
	for the Segal space of maps from $x$ to $y$. This is level-wise defined by the pullback square 
	\[
		\begin{tikzcd}[column sep = small]
			X(x,y)_n \arrow[rr]\arrow[d] && X_{1,n} \arrow[d] \\
			\{(x,y)\} \arrow[r] & X_{0,0}^{\times 2} \arrow[r,"\simeq"] & X_{0,n}^{\times 2}
		\end{tikzcd}
	\]
	for $[n] \in \Delta$. A 2-fold Segal space is locally complete if and only if all mapping Segal spaces are complete. This was shown by Haugseng \cite{HaugsengSpans}, see also \cite[Proposition 2.16]{EquipI}.
\end{notation}

\begin{remark}
	If $\P$ is a double $\infty$-category, then $\Hor(\P)$ is always locally complete (see \cite[Remark 2.19]{EquipI} and \cite[Remark 3.10]{FunDblCats}). By definition (see \cite[Table 3.1.1]{FunDblCats}), 
	a double $\infty$-category $\P$ is locally complete 
	if and only if $\Vert(\P)$ is locally complete as well.
\end{remark}

\begin{notation}\label{not.rezk}
	Let $$\Seg(\S) \subset \PSh(\Delta)$$ 
	denote the full subcategory spanned by the Segal spaces of Rezk \cite{RezkSeg}. I.e.\ it spanned by those presheaves $X : \Delta^\op \rightarrow \S$ 
	which are local with regard to the spine inclusions 
	$$
	[1] \cup_{[0]} \dotsb \cup_{[0]} [1] \rightarrow [n],
	$$
	with $n\geq 1$.
	It was shown by Joyal and Tierney \cite{JoyalTierney} that the restriction of the Yoneda embedding $\Cat_\infty \rightarrow \PSh(\Delta)$ 
	is the inclusion of the reflective subcategory of complete Segal spaces, i.e.\ those Segal spaces that are additionally local with respect to the map 
	$$
	J \rightarrow [0]
	$$
	where $J$ is the Segal set that presents the free living invertible arrow. This Segal set may be defined via the pushout square as in \cite[Definition 2.3]{EquipI}.
	There is a canonical adjunction 
	$
	\PSh(\Delta) \rightleftarrows \PSh(\Delta^{\times 2})
	$
	induced by the adjunction 
	$
	\Delta^{\times 2} \rightleftarrows \Delta,
	$
	where the left adjoint is the projection onto the first coordinate and the right adjoint is the inclusion of 
	$\Delta \times \{[0]\}$.
	The adjunction on presheaf $\infty$-categories restricts to adjunctions
	$$
	\Seg(\S) \rightleftarrows \Seg^2(\S) : (-)^{(1)}, \quad \Cat_\infty \rightleftarrows \Cat_{(\infty,2)} : (-)^{(1)},
	$$
	as one readily verifies.
	We leave the (fully faithful) left adjoint inclusions implicit in this paper.
\end{notation}

\begin{remark}
	A 2-fold Segal space $X$ is an $(\infty,2)$-category if and only if $X$ is locally complete and $X^{(1)}$ is a complete Segal space. This can 
	easily be deduced from \cite[Remark 3.12]{FunDblCats}.
\end{remark}

\begin{notation}
As shown and discussed in \cite{FunDblCats}, the $\infty$-categories $\DblCat_\infty$ and $\Seg^2(\S)$ are cartesian closed. We will use the same notation 
for the internal homs. 
For double $\infty$-categories $\P$ and $\Q$, we write
$$
\FFUN(\P,\Q)
$$
for the double $\infty$-category of functors from $\P$ to $\Q$. We write 
$$
\FUN(X,Y)
$$
for the 2-fold Segal space of functors from $X$ to $Y$.
\end{notation}

\begin{remark}
	It is shown in \cite[Proposition 3.16]{FunDblCats} that $\FFUN(\P, \Q)$ is locally complete if $\Q$ is locally complete. 
	Moreover, if $Y$ is a locally complete 2-fold Segal space (resp.\ an $(\infty,2)$-category), then $\FUN(X,Y)$ is locally complete (resp.\ an $(\infty,2)$-category).
\end{remark}

\begin{notation}
	Let $X$ be a 2-fold Segal space and $\P$ a double $\infty$-category.
	Throughout this article, we will make use of both the \textit{vertical cotensor product}
	$$
	[X, \P] := \FFUN(X_v, \P)
	$$ and \textit{horizontal cotensor product} $$\{X, \P\} := \FFUN(X_h, \P)$$
	of $X$ and $\P$. These cotensor products were defined in \cite[Definition 3.19]{FunDblCats}. They interact well with 
	taking fragments; there are  canonical equivalences
	$$
	\Vert([X,\P]) \simeq \FUN(X, \Vert(\P)) \quad \text {and} \quad \Hor(\{X,P\}) \simeq \FUN(X, \Hor(\P)),
	$$
	see \cite[Proposition 3.23]{FunDblCats}.
\end{notation}

\begin{remark}\label{rem.cotensor-infty-cat}
	If $X$ is an $\infty$-category and $\P$ a double $\infty$-category, then the vertical cotensor product $[X,\P]$ is computed by the composite 
	$$
	\Delta^\op \xrightarrow{\P} \Cat_\infty \xrightarrow{\fun(X,-)} \Cat_\infty.
	$$
	We refer to \cite[Remark 3.24]{FunDblCats} for an explanation.
\end{remark}

We also recall the following definition. Firstly, note that for horizontal arrows $F : x' \rightarrow x$ and $G: y \rightarrow y'$ of a double $\infty$-category $\P$, one may consider the two-sided 
composition functor
$$
G \circ (-) \circ F : \Hor(\P)(x,y) \rightarrow \Hor(\P)(x',y').
$$
Then, we studied the following notions in \cite[Definition 2.27]{EquipI} and \cite[Definition 6.5]{FunDblCats}:

\begin{definition}
	Let $\P$ be a double $\infty$-category. Then $\P$ is called \textit{horizontally closed} if $G \circ (-) \circ F$ is 
	a left adjoint for every two horizontal arrows $F$ and $G$ of $\P$. If, additionally, all horizontal mapping $\infty$-categories $\Hor(\P)(x,y)$, $x,y\in \P$, are presentable, then $\P$ is called \textit{horizontally locally presentable}.
\end{definition}

\begin{example}
	Let $\C$ be an $\infty$-category with finite limits. 
	Then $\C$ is locally cartesian closed if and only if $\SSPAN(\C)$ is horizontally closed. This observation appears 
	in \cite[Corollary 7.4]{EquipI}.
\end{example}

\subsection{\texorpdfstring{$\infty$}{∞}-Equipments} We may expand \ref{table.dss} with another entry that consists of 
those double $\infty$-categories that support a good notion of \textit{formal category theory}:

\begin{table}[ht]
	\centering 
	%\caption{Relevant notions of double and 2-fold Segal spaces}
\renewcommand{\arraystretch}{1.2}
\begin{tabular}{| c c c | c |}
	\hline
	\multicolumn{3}{|c|}{\textbf{Reflective subcategory}} & \textbf{Local objects} \\
	\hline 
	$\Equip_\infty$ & $\subset$ & $\DblCat_\infty^{lc}$ & \textit{$\infty$-Equipments} \\
	\hline
\end{tabular}
\end{table}

\noindent The basics of the formal category theory internal to $\infty$-equipments has been studied and developed in \cite[Section 6]{EquipI}.
We will refer to this when necessary.

Recall that the definition of an $\infty$-equipment makes use of the notion of \textit{companionships} and \textit{conjunctions} in locally complete 
double $\infty$-categories. This is a double categorical analog of adjunctions. We will not review these two notions here, but instead, refer the reader 
to \cite[Definition 4.1]{FunDblCats} and \cite[Definition 4.2]{FunDblCats} for the definitions with the used terminology.

\begin{notation}\label{not.comp-conj}
	Suppose that $f : x \rightarrow y$ is a vertical arrow of a locally complete double $\infty$-category $\P$. Then we will denote 
	its companion and conjoint horizontal arrow (if they exist) by 
	$$
	f_\circledast : x \rightarrow y \quad \text{ and } \quad f^\circledast :y \rightarrow x, 
	$$
	respectively.
\end{notation}

\begin{definition}
	A locally complete double $\infty$-category is called an \textit{$\infty$-equipment} if it admits all companions and conjoints.
\end{definition}

\begin{example}
	The prototypical example of an $\infty$-equipment is the double $\infty$-category 
	$$\CCAT_\infty$$ 
	of $\infty$-categories. Its vertical arrows are given by functors and the horizontal arrows are given by 
	\textit{correspondences}, which may be viewed as \textit{profunctors} or \textit{two-sided discrete fibrations} as well, see \cite[Section 4]{EquipI}. 
	We will write 
	$$
	\CAT_\infty := \Vert(\CCAT_\infty)
	$$
	for the vertical fragment of $\CCAT_\infty$. This is the model for $(\infty,2)$-category of $\infty$-categories that we will use throughout this 
	paper. It has been compared to Lurie's model in \cite[Subsection 4.3]{EquipI}.
	
	More generally, we have constructed the $\infty$-equipment 
	$$
	\CCAT_\infty(\E)
	$$
	of $\infty$-categories internal to an arbitrary $\infty$-topos $\E$ in \cite[Section 5]{EquipI}. This $\infty$-equipment is one of the key players in throughout this article. We will 
	review its definition in \ref{ssection.con-modules}. We write 
	$$
	\CAT_\infty(\E) := \Vert(\CCAT_\infty(\E))
	$$
	for its vertical fragment.
\end{example}

\begin{convention}
	We will refer to the vertical and horizontal arrows $\P$ in an $\infty$-equipment as the \textit{arrows} and \textit{proarrows} in $\P$ respectively.
\end{convention}

There is also another characterization of $\infty$-equipments in terms of \textit{(co)cartesian 2-cells}.

\begin{definition}[{\cite[Definition 3.19]{EquipI}}]
	A 2-cell in a double $\infty$-category $\P$ is called \textit{(co)cartesian} if it classifies a (co)cartesian arrow 
	of the source-target projection $\P_1 \rightarrow \P_0^{\times 2}$.
\end{definition}

\begin{proposition}[{\cite[Corollary 3.28]{EquipI}}]
	Let $\P$ be a locally complete double $\infty$-category. Then the following are equivalent: 
	\begin{enumerate}
		\item $\P$ is an $\infty$-equipment,
		\item the source-target projection $\P_1 \rightarrow \P_0^{\times 2}$ is a cartesian fibration,
		\item the source-target projection $\P_1 \rightarrow \P_0^{\times 2}$ is a cocartesian fibration.
	\end{enumerate}
\end{proposition}

This relied on the following verification:

\begin{proposition}[{\cite[Proposition 3.24]{EquipI}}]\label{prop.cart-comp-conj}
	Let $f : a \rightarrow x$ and $g : b \rightarrow y$ be arrows of a locally complete double $\infty$-category $\P$. Suppose that 
	$f$ admits a companion and $g$ admits a conjoint. If $F : x\rightarrow y$ is a proarrow 
	of $\P$, then the pasting of 2-cells
	\[
		\begin{tikzcd}
			a \arrow[r, "{f_\circledast}"]\arrow[d,"f"'] & x \arrow[r,	"F"]\arrow[d,equal] & y \arrow[r, "{g^\circledast}"] \arrow[d, equal] & b\arrow[d, "g"] \\
			x \arrow[r, equal] & x \arrow[r, "F"] & y \arrow[r,equal] & y,
		\end{tikzcd}
	\]
	is a cartesian 2-cell. Here, the left and right 2-cells are the companionship and conjunction counits respectively.
	Similarly, if $G : b\rightarrow a$ is a horizontal arrow of $\P$, then the pasting
	\[
		\begin{tikzcd}
			b \arrow[r, equal]\arrow[d,"g"'] & b\arrow[r,	"G"]\arrow[d,equal] & a \arrow[r, equal] \arrow[d, equal] & a\arrow[d, "f"] \\
			y \arrow[r, "g^\circledast"] & b \arrow[r, "G"] & a \arrow[r,"f_\circledast"] & x,
		\end{tikzcd}
	\]
	is a cocartesian 2-cell. Here, the left and right 2-cells are now the conjunction and companionship units.
\end{proposition}

The class of $\infty$-equipments interacts well with taking vertical cotensor products: 

\begin{proposition}[{\cite[Corollary 4.9]{FunDblCats}}]
	The vertical cotensor products of $\infty$-equipments are again $\infty$-equipments.
\end{proposition}

In fact, we gave a complete description of companions and conjoints in functor double $\infty$-categories in \cite[Theorem 4.6]{FunDblCats} using the notion of 
\textit{companionable} and \textit{conjointable} 2-cells (see \cite[Definition 4.4]{FunDblCats}).

\begin{example}\label{ex.cotensors-ccat}
	As explained in \cite[Example 3.26]{FunDblCats}, there is an equivalence 
	\[
		[T^\op, \CCAT_\infty] \simeq \CCAT_\infty(\PSh(T))
	\]
	for every small $\infty$-category $T$.
\end{example}

We will 
make use of the following characterization of the (co)cartesian 2-cells in vertical cotensors:

\begin{proposition}\label{prop.cart-cotensors}
	Let $\P$ be an $\infty$-equipment and suppose that $X$ is a 2-fold Segal space. Then a 2-cell 
	\[
		\begin{tikzcd}
			h \arrow[r, "\gamma"]\arrow[d, "\alpha"'] & k \arrow[d, "\beta"] \\ 
			l \arrow[r, "\delta"] & m
		\end{tikzcd}
	\]
	in $[X,\P]$ is (co)cartesian if and only if it is (co)cartesian point-wise, i.e.\
	\[
		\begin{tikzcd}
			h(x) \arrow[r, "\gamma_x"]\arrow[d, "\alpha_x"'] & k(x) \arrow[d, "\beta_x"] \\ 
			l(x) \arrow[r, "\delta_x"] & m(x)
		\end{tikzcd}
	\]
	is (co)cartesian in $\P$ for every $x \in X$.
\end{proposition}
\begin{proof}
	We will only consider the cartesian case since the other case is similar. Note that the 2-cell of $[X, \P]$ in question is cartesian 
	if and only if the comparison vertical 2-cell $\gamma \rightarrow \cart{\beta}{\delta}{\alpha}$ is invertible in $\Vert([X,\P]) = \FUN(X, \P)$. 
	But this can be detected point-wise, so we have to check that the induced arrow $\gamma_x \rightarrow \beta_x^\circledast{\delta_x}{\alpha_{x,\circledast}}$ is invertible in $\Vert(\P)$ for every $x \in X$. 
	Here, we used that $\ev_x : [X, \P] \rightarrow \P$ is a functor and hence preserves cartesian 2-cells. 
\end{proof}

\subsection{Gray tensor product}\label{ssection.gray}
Throughout this article, we will make use of the model for the Gray tensor product for $(\infty,2)$-categories that was 
considered in \cite{FunDblCats}. This definition relied on the \textit{squares construction}: the functor
$$
\Sq : \Cat_{(\infty,2)} \rightarrow \DblCat_\infty 
$$
that carries an $(\infty,2)$-category to the double $\infty$-category of \textit{squares} in $\X$, given by 
$$
\Sq(\X)_{n,m} = \map_{\Cat_{(\infty,2)}}([n]\otimes_s [m], \X).
$$
Here $(-) \otimes_s (-)$ denotes the Gray tensor product for strict 2-categories. The joint work \cite{Squares} is dedicated 
to the study of the squares functor and its results have been used throughout \cite{EquipI} and \cite{FunDblCats}.

\begin{definition}[{\cite[Definition 5.2]{FunDblCats}}]
	The \textit{Gray tensor product} $\X \otimes \Y$ of $(\infty,2)$-categories $\X$ and $\Y$ is the unique $(\infty,2)$-category with the property that
	$$
	\map_{\Cat_{(\infty,2)}}(\X \otimes \Y, \mathscr{Z}) \simeq \map_{\DblCat_\infty}(\X_h \times \Y_v, \Sq(\mathscr{Z})).
	$$
	naturally for every $(\infty,2)$-category $\mathscr{Z}$.
\end{definition}

\begin{notation}[{\cite[Definition 5.6]{FunDblCats}}]
	Let $\X$ and $\Y$ be $(\infty,2)$-categories. Then we will consider the vertical cotensor product
	$$
	\FFUN^\lax(\X, \Y) := [\X, \Sq(\Y)],
	$$
	which has the following description:
	\begin{itemize}
		\item its objects are functors $f : \X \rightarrow \Y$,
		\item its vertical arrows $f \rightarrow g$ correspond to natural transformations from $f$ to $g$, i.e.\ maps $\alpha : [1] \times \X \rightarrow \Y$ 
		with identifications $\alpha|\{0\} \times \X \simeq f$ and $\alpha|\{1\} \times \X \simeq g$,
		\item its horizontal arrows $f \rightarrow g$ correspond to \textit{lax natural transformations} from $f$ to $g$, i.e.\ maps $\alpha : [1] \otimes \X \rightarrow \Y$ 
		with identifications $\alpha|\{0\} \otimes \X \simeq f$ and $\alpha|\{1\} \otimes \X \simeq g$.
	\end{itemize}
\end{notation}

\begin{remark}\label{rem.companions-ffun}
	The companions in $\FFUN^\lax(\X,\Y)$ are precisely those lax natural transformations $\alpha : h \rightarrow k$ that are \textit{strict}, i.e.\ 
	those for which the 2-cells inhabiting the lax naturality squares 
	\[ 
		\begin{tikzcd}
			h(x) \arrow[r, "\alpha_x"] \arrow[d, "h(f)"']& |[alias=f]| k(x) \arrow[d, "k(f)"] \\
			|[alias=t]| h(y) \arrow[r, "\alpha_y"'] & k(y)
			\arrow[from=f,to=t, Rightarrow]
		\end{tikzcd}
	\]
	associated to each arrow $f : x \rightarrow y$ in $\X$, are invertible (see \cite[Proposition 5.7]{FunDblCats}). There is a locally fully faithful companion embedding
	$$
	\Vert(\FFUN^\lax(\X,\Y)) = \FUN(\X,\Y) \rightarrow \FUN^\lax(\X,\Y) = \Hor(\FFUN^\lax(\X,\Y))
	$$
	that is the identity on objects and carries natural transformations to their associated strict lax natural transformation. This is discussed in \cite[Section 5.1]{FunDblCats} and 
	relies on the results of \cite{Squares} (cf.\ \cite[Corollary 3.13]{EquipI}).
\end{remark}

\begin{remark}\label{rem.sq-in-ffun}
	The \textit{universal property of squares} that will be proven in \cite{Squares} dictates the existence of a unique functor 
	$$
	\Sq(\FUN(\X,\Y)) \rightarrow \FFUN^\lax(\X,\Y)
	$$
	that extends the counit $\FUN(\X,\Y)_v = \Vert(\FFUN^\lax(\X,\Y))_v \rightarrow \FFUN^\lax(\X,\Y)$. The resulting extension 
	has the property that it is an equivalence onto the subobject of $\FFUN^\lax(\X,\Y)$ whose horizontal arrows are companions: the strict lax natural transformations.
\end{remark}

\subsection{Conduch\'e modules}\label{ssection.con-modules} We will give a quick recollection of the material in \cite[Section 5]{EquipI} in slightly more generality. In particular, 
we will recover the definition of the $\infty$-equipment $\CCAT_\infty(\E)$ of categories internal to $\E$ in the process. 

\begin{notation}
	For an $\infty$-category $K$, we will consider the $\infty$-category $$\Delta/K := \Cat_\infty/K \times_{\Cat_\infty} \Delta.$$ 
	It has the feature that the projection $\Delta/K \rightarrow \Delta$ is precisely the right fibration that classifies $K$, viewed as a Segal space.
\end{notation}

\begin{definition}\label{def.con-K-modules}
Let $\E$ be an $\infty$-topos. A \textit{Conduch\'e $K$-module internal to $\E$} is a functor 
$$
M : (\Delta/K)^\op \rightarrow \E 
$$
so that:
\begin{enumerate}
	\item for every $k \in K$, the restriction $(\{k\}_!)^*M$ is an $\infty$-category internal to $\E$ (see \cite[Section 7]{HaugsengSpans} or \cite[Definition 5.14]{EquipI}),
	\item for every arrow $f : [2] \rightarrow K$, the canonical map $$
	\quad\colim_{[n] \in \Delta^\op} M([n+1] \xrightarrow{0 \leq 1 \leq \dotsb \leq 1 \leq 2} [2] \xrightarrow{f} K) \rightarrow M([1] \xrightarrow{0\leq 2} [2] \xrightarrow{f} K)
	$$ 
	is an equivalence, 
\end{enumerate}
cf.\ \cite[Proposition 5.3]{EquipI}.
We will write 
$$
\Con(K; \E) \subset \fun((\Delta/K)^\op, \E)
$$ for the full subcategory spanned by these Conduch\'e modules.
\end{definition}

\begin{remark}
Note that the subcategories of Conduch\'e modules are functorial. We can consider the composite functor
$$
\Cat_\infty^\op \xrightarrow{\Delta/(-)} \widehat{\Cat}_\infty^\op \xrightarrow{\fun(-, \E)} \widehat{\Cat}_\infty
$$
and this restricts to the Conduch\'e modules level-wise, yielding a functor $\Con(-; \E) : \Cat_\infty^\op \rightarrow \widehat{\Cat}_\infty$.
\end{remark}

\begin{remark}
	Note that the Conduch\'e $[n]$-modules are precisely the \textit{complete Conduch\'e $(n+1)$-modules} of \cite[Definition 5.14]{EquipI}.
	Condition (2) of \ref{def.con-K-modules} is vacuous for $n = 0,1$ by characterization (1) of \cite[Proposition 5.3]{EquipI}.
\end{remark}

\begin{construction}
	The simplicial $\infty$-category $\CCON(K;\E)$ is obtained by restricting the functor $\Con(-; \E)$ along the composite
	$$
	\Delta \xrightarrow{(-)^\op} \Delta \xrightarrow{- \times K} \Cat_\infty.
	$$
\end{construction}

\begin{remark}
	Note that $\CCON([0]; \E)$ concides  with the double $\infty$-category $\CCAT_\infty(\E)$ of \cite[Construction 5.16]{EquipI}.
\end{remark}

We may now prove the following generalization of \cite[Corollary 5.19]{FunDblCats}:

\begin{proposition}\label{prop.con-straightening}
	The simplicial $\infty$-category $\CCON(K;\E)$ is a locally complete double $\infty$-category and there is a natural equivalence 
	$$
	\{K^\op, \CCAT_\infty(\E)\} \simeq \CCON(K; \E)
	$$
	for every $\infty$-category $K$.
\end{proposition}
\begin{proof}
	The demonstration provided in \cite[Corollary 5.19]{FunDblCats} generalizes to this situation as well. Following this strategy, it suffices to show that 
	$$\map_{\DblCat_\infty}(X_v \times K_h, \CCAT_\infty(\E)) \simeq \map_{\Cat_\infty}(X, \Con(K^\op;\E))$$
	naturally for $\infty$-categories $X$ and $K$. One can rerun the proof of \cite[Theorem 1.26]{AyalaFrancis} to show that 
	$\Con(-;\E)$ is right Kan extended from its restriction to $\Delta$. Consequently, it suffices to check the above for $K = [n]$. In these cases, 
	the equivalence holds by construction. Note that $\CCON(K;\E)$ is locally complete since $\CCAT_\infty(\E)$ is locally complete, see \cite[Proposition 3.16]{FunDblCats}.
\end{proof}

\section{Lax functors and adjunctions}\label{section.lax-fun}

In this section, we introduce and study the notion of (normal) lax functors between double $\infty$-categories, as well as a suitable 
notion of adjunctions between lax functors.

\subsection{Lax functors} Following Grandis and Par\'e \cite{GrandisPareLimits}, we may extend the notions of lax functors and adjunctions 
of (strict) double categories using the Gray tensor product (see \ref{ssection.gray}).

\begin{definition}
	Let $\P$ and $\Q$ be double $\infty$-categories. A \textit{lax functor} $h : \P \rightarrow \Q$ between 
	double $\infty$-categories is a lax natural transformation 
	$$h : [1] \otimes \Delta^\op \rightarrow \CAT_\infty$$ from $\P$ to $\Q$ with the property 
	that for any inert map $\alpha : [n] \rightarrow [m]$, 
	the associated lax naturality square 
	\[
	\begin{tikzcd}
		\P_m \arrow[r, "h_m"]\arrow[d, "\alpha^*"'] & |[alias=f1]|\Q_m\arrow[d, "\alpha^*"] \\
		|[alias=t1]|\P_n \arrow[r, "h_n"'] & \Q_n
		\arrow[from=f1, to=t1, Rightarrow]
	\end{tikzcd}
	\] commutes. Additionally, we call $U$ \textit{normal} if this condition holds for surjective structure maps $\alpha$ as well. 
	It is called \textit{strict} if all lax naturality squares commute.
\end{definition}

\begin{remark}
	In the double categorical literature, normal lax functors between (strict) double categories are also called \textit{unitary} or \textit{normalized}. We will stick to the terminology of Shulman \cite{ShulmanFramedBicats}. 
\end{remark}

\begin{remark}
	One may also define a dual notion of \textit{oplax functors} using \textit{oplax} natural transformations. However, we will not need this notion 
	in this paper.
\end{remark}

\begin{remark}\label{rem.lax-fun-fib}
	In \cite[Theorem 5.3.1]{HHLN}, Haugseng, Hebestreit, Linskens and Nuiten prove that lax natural transformations between $\Cat_\infty$-valued functors may be viewed 
	as fibered maps between their corresponding cocartesian fibrations, not necessarily preserving cocartesian arrows. 
	For the definition of lax natural transformations that we use in this paper (see \ref{ssection.gray}), this has been proven 
	use double categorical techniques in \cite[Proposition 5.20]{EquipI}.

	This result specializes in the following fibrational perspective on lax functors between double $\infty$-categories.
	If
	$\bar{\P} \rightarrow \Delta^\op$ and $\bar{\Q} \rightarrow \Delta^\op$ are cocartesian fibrations that classify 
	double $\infty$-categories $\P$ and $\Q$, then a lax functor between $\P$ and $\Q$ is precisely a functor $\bar{\P} \rightarrow \bar{\Q}$ fibered over $\Delta^\op$ that preserves 
	cocartesian lifts of inert morphisms.
\end{remark}

\begin{notation}
	Let $h : \P \rightarrow \Q$ be a lax functor between double $\infty$-categories. Then we may take fibers in the commutative naturality square
	\[
		\begin{tikzcd}
			\P_1 \arrow[r]\arrow[d] & \Q_1 \arrow[d] \\
			\P_0^{\times 2} \arrow[r] & \Q_0^{\times 2}
		\end{tikzcd}
	\]
	above objects $x,y \in \P$. We will denote this functor by 
	$$
	h_{x,y} : \Hor(\P)(x,y) \rightarrow \Hor(\Q)(h(x), h(y)).
	$$
\end{notation}

We may construct an ambient double $\infty$-category whose objects are double $\infty$-categories, 
vertical arrows are functors and horizontal arrows are lax functors:

\begin{definition}
	The \textit{double $\infty$-category of double $\infty$-categories}
	$$\DDBLCAT_\infty$$
	is defined to be the sub double $\infty$-category 
	of the pullback 
	$$
	\Sq(\FUN(\Delta^\op_{\mathrm{in}}, \CAT_\infty)) \times_{\FFUN^\lax(\Delta^\op_{\mathrm{in}}, \CAT_\infty)} \FFUN^\lax(\Delta^\op, \CAT_\infty)
	$$
	formed using the map of \ref{rem.sq-in-ffun},
	spanned by the double $\infty$-categories. Here $\Delta_\mathrm{in}$ denotes the wide subcategory of $\Delta$ consisting 
	of the inert maps.
\end{definition}

\begin{definition}
We will write 
$$
\mathrm{DBLCAT}_\infty \quad \text{ and } \quad \mathrm{DBLCAT}_\infty^\lax 
$$
for the vertical and horizontal $(\infty,2)$-category of $\DDBLCAT_\infty$ respectively.
\end{definition} 

Note that the companions of $\DDBLCAT^\lax$ are precisely the strict lax functors, see \ref{rem.companions-ffun}. Moreover, this remark specializes in the following:

\begin{proposition}
	There is a locally fully faithful functor 
	$$
	\mathrm{DBLCAT}_\infty \rightarrow \mathrm{DBLCAT}_\infty^\lax
	$$
	that is the identity on objects and carries a functor to its associated strict lax functor.
\end{proposition}

\subsection{Lax adjunctions} Now that we have access to ambient $(\infty,2)$-categories of double $\infty$-categories, we can make 
sense of adjunctions between double functors:

\begin{definition}
	Let $u : \P \rightarrow \Q$ and $v : \Q \rightarrow \P$ be lax functors between double $\infty$-categories. Then 
	$(u,v)$ is called a \textit{lax adjunction} if the pair forms an adjunction in $\mathrm{DBLCAT}_\infty^\lax$. Moreover, we use 
	the following terminology:
	\begin{itemize}
		\item the adjunction is called \textit{normal} if both $u$ and $v$ are normal,
		\item the adjective `lax' is dropped if both $u$ and $v$ are strict, in which case $(u,v)$ is an adjunction in $\mathrm{DBLCAT}_\infty$.
	\end{itemize}
\end{definition}

We have the following useful level-wise characterization of lax adjunctions:

\begin{proposition}\label{prop.lax-double-adjunctions-char}
	Suppose that $u : \P \rightarrow \Q$ is a lax functor between double $\infty$-categories. Then the following assertions are equivalent: 
	\begin{enumerate}
		\item $u$ admits a lax right adjoint,
		\item $u$ is strict, each functor $u_n : \P_n \rightarrow \Q_n$ admits a right adjoint $v_n : \Q_n \rightarrow \P$, 
		and the mate transformation
		$$
		\alpha^*v_m \rightarrow  v_nu_n\alpha^*v_m \simeq v_n\alpha^*u_mv_m \rightarrow v_n\alpha^*
		$$
		is invertible for every inert map $\alpha : [n] \rightarrow [m]$,
		\item $u$ is strict and (2) holds for $n,m \in \{0,1\}$.
	\end{enumerate}
	In this case, $u$ is normal if and only if assertion (2) or (3) holds for the surjective maps $\alpha$ as well.
\end{proposition}
\begin{proof}
	The equivalence of assertions (1) and (2) follows from Haugeng's recognition theorem for adjunctions between lax natural transformations \cite[Theorem 4.6]{HaugsengLax}; this has been shown in \cite[Corollary 5.11]{FunDblCats} for the current 
	model of lax natural transformations. One readily deduces 
	that (2) and (3) are equivalent by using the fact that $\P$ and $\Q$ satisfy the Segal conditions. 
\end{proof}

\begin{example}\label{ex.dbl-adj-geom-morph}
	Let $f = (f^*, f_*) : \E \rightarrow \F$ be a geometric morphism between $\infty$-toposes. 
	According to \cite[Proposition 5.17]{EquipI}, the inverse image functor 
	induces a functor 
	$$
	f^* := \CCAT_\infty(f^*) : \CCAT_\infty(\F) \rightarrow \CCAT_\infty(\E)
	$$
	between $\infty$-equipments of internal categories, induced by post-composing with $f^*$. Here we crucially used 
	that the inverse image preserves colimits, so in particular, geometric realizations. The direct image $f_*$ does not 
	generally preserve these realizations, so it does not give rise to a functor per se. However, 
	post-composition with $f_*$ preserves 
	Segal and completeness conditions so that characterization (3) of \ref{prop.lax-double-adjunctions-char} implies that 
	there exists a \textit{normal lax} functor 
	$$
	f_* : \CCAT_\infty(\E) \rightarrow \CCAT_\infty(\F)
	$$
	right adjoint to $f^*$.

	In the case that $f$ is given by an \'etale geometric morphism $\E/e \rightarrow \E$ 
	for some object $e \in \E$, then $f^*$ admits a further left adjoint 
	$f_! : \E/e \rightarrow \E$ given by projection. This projection does not preserve the terminal object, but Proposition 5.17 of \cite{EquipI}
	still yields a (strict) left adjoint for $f_*$, denoted by
	$$
	f_! := \CCAT_\infty(f_!) : \CCAT_\infty(\E) \rightarrow \CCAT_\infty(\F).
	$$
\end{example}

\begin{definition}\label{def.dbl-limits}
	Suppose that $I$ and $\P$ are double $\infty$-categories. Then we say that $\P$ \textit{admits all $I$-shaped limits} if the diagonal functor 
	$$
	\Delta_I : \P \rightarrow \FFUN(I, \P)
	$$
	admits a lax right adjoint $\lim_I$. We say that these limits are \textit{normal} (resp.\ \textit{strict}) if $\lim_I$ is normal (resp.\ strict). 
\end{definition}

\begin{proposition}\label{prop.dbl-limits-v-shaped}
	Let $I$ be an $\infty$-category. Then the following assertions are equivalent:
	\begin{enumerate}
		\item $\P$ has all $I_v$-shaped limits,
		\item for each $n$, the $\infty$-category $\P_n$ has $I$-shaped limits, and they 
		are preserved by the structure maps $\P_m \rightarrow \P_n$ associated to inert map $\alpha : [n] \rightarrow [m]$,
		\item (2) holds for $n,m \in \{0,1\}$.
	\end{enumerate}
	Moreover, if (2) or (3) holds for the surjective maps (resp.\ all maps) as well, then all $I_v$-shaped limits are normal (resp.\ strict).
\end{proposition}
\begin{proof}
	This follows from the characterizations of \ref{prop.lax-double-adjunctions-char}.
\end{proof}

\begin{example}\label{ex.int-ccat-dbl-limits}
	Let $\E$ be an $\infty$-topos. Then one may use the above characterization to verify that $\CCAT_\infty(\E)$ has all $I_v$-shaped limits for every 
	small $\infty$-category $I$. It also admits $I_h$-shaped limits. To this end, we may replace $I$ by $I^\op$ 
	and use the identification of \ref{prop.con-straightening}, so that we must demonstrate that the functor 
	$$
	\CCON([0]; \E) \rightarrow \CCON(I; \E)
	$$ 
	admits a right adjoint. In light of \ref{prop.lax-double-adjunctions-char}, it suffices to demonstrate that for $i = 0,1$ the commutative square 
	\[
		\begin{tikzcd}
			\Con([1]; \E) \arrow[r]\arrow[d, "{d_i^*}"'] & \Con([1] \times I; \E)\arrow[d, "d_i^*"] \\
			\Con([0]; \E) \arrow[r] & \Con(I; \E)
		\end{tikzcd}
	\]
	is right adjointable: i.e.\ the horizontal arrows are left adjoints and the associated mate is an equivalence.
	But the horizontal arrows are restricted from the functors
	$$
		p_j^* : \PSh((\Delta/[j])^\op, \E) \rightarrow \fun((\Delta/([j]\times I))^\op, \E),
	$$
	for appropriate choice of $j \in \{0,1\}$,
	where $p_j : \Delta/([j] \times I) \rightarrow \Delta/[j]$ denotes the pushforward. The functor $p_j^*$ admits a right adjoint 
	$p_{j,*}$ that is computed by right Kan extensions along $p_j$. Slightly informally, it may be computed by
	the formula 
	$$p_{j,*}X(f) = \lim_{g \in \map_{\Cat_\infty}([n], I)} X([n] \xrightarrow{(g,f)} [j] \times I)$$
	for $X : (\Delta/([j]\times I))^\op \rightarrow \E$ and $f : [n] \rightarrow [j]$.
	The functor $p_{j,*}$ restricts to a right adjoint 
	$$
	\Con([j] \times I; \E) \rightarrow \Con([j]; \E).
	$$
	It is not necessary to check the Conduch\'e conditions as these are vacuous for the codomains when $j=0,1$. One can readily verify that the mate is an equivalence.
\end{example}

\subsection{Lax adjunctions locally} We will briefly study the effect of a lax adjunction on the mapping $\infty$-categories 
of the vertical and horizontal fragments. 

\begin{proposition}\label{prop.lax-adj-vert}
	Suppose that $u : \P \rightleftarrows \Q : v$ is a normal lax adjunction between locally complete double $\infty$-categories with 
	counit $\epsilon$. Then the composite 
	functor 
	$$
	\Vert(\P)(x, v(y)) \rightarrow \Vert(\Q)(u(x), uv(y)) \xrightarrow{\epsilon_y \circ (-)} \Vert(\Q)(u(x), y)
	$$
	is an equivalence.
\end{proposition}
\begin{proof}
It will suffice to show that for each $n \geq 0$, the map of spaces 
$$
\Vert(\P)(x, v(y))_n \rightarrow \Vert(\Q)(u(x), uv(y))_n \xrightarrow{\epsilon_y \circ (-)} \Vert(\Q)(u(x), y)_n
$$
is an equivalence. Let $s : [n] \rightarrow [0]$ denote the unique degeneracy map. On account of \cite[Remark 2.19]{EquipI}, this composite 
may be identified with the composite 
$$
\map_{\P_n}(s^*x, s^*v_0(y)) \rightarrow \map_{\Q_n}(s^*u_0(x), s^*u_0v_0(y)) \xrightarrow{s^*\epsilon_y \circ (-)} \map_{\Q_n}(s^*u_0(x), s^*y).
$$
Since $s$ is surjective and $v$ is normal, there is a commutative diagram
\[
	\begin{tikzcd}[row sep = tiny, column sep = large]
		\Q_0\arrow[dd,"s^*"'] \arrow[dr, bend right = 10pt, "v_0"'name=f1]\arrow[rr,equal, bend left = 15pt, ""'name=t1] & & \Q_0\arrow[dd,"s^*"]\\ 
		& |[alias=m1]|\P_0\arrow[ur, bend right = 10pt, "u_0"'] & \\
		\Q_n \arrow[dr, bend right = 10pt, "v_n"'name=f2]\arrow[rr,equal, bend left = 15pt, ""'name=t2] & & \Q_n \\
		& |[alias=m2]|\P_n \arrow[ur, "u_n"', bend right = 10pt] &
		\arrow[from=f1,to=t1, Rightarrow, shorten <= 5pt, shorten >= 2pt]
		\arrow[from=f2,to=t2, Rightarrow, shorten <= 5pt, shorten >= 2pt]
		\arrow[from=m1, to=m2, "s^*", crossing over] 
	\end{tikzcd}
\]
in $\CAT_\infty$ where the top 2-cell is given by $\epsilon|[1;1] \times \{0\}$ and the bottom 2-cell is given by $\epsilon|[1;1] \times \{n\}$.
Thus the whiskered counit $s^*\epsilon_y :  s^*u_0v_0(y) \rightarrow s^*y$ in $\Q_n$ coincides with $\epsilon_{s^*y} : u_nv_n(s^*y) \rightarrow s^*y$,
and the above composite is given by
$$
\map_{\P_n}(s^*x, v_n(s^*y)) \rightarrow \map_{\Q_n}(u_n(s^*x), u_nv_n(s^*y)) \xrightarrow{\epsilon_{s^*y} \circ (-)} \map_{\Q_n}(u_n(s^*x), s^*y).
$$
This is an equivalence since $\epsilon_{s^*y}$ is the counit for the adjuction $(u_n, v_n)$ at $s^*y$.
\end{proof}

The situation on horizontal fragments turns out to be more subtle, see also Shulman's discussion in \cite[Appendix B]{ShulmanFramedBicats}.
However, when we restrict our attention to equipments, we may still recover induced local adjunctions:

\begin{proposition}\label{prop.lax-adj-hor}
	Suppose that we have a lax adjunction 
	$$
	u : \P \rightleftarrows \Q : v
	$$
	between $\infty$-equipments.
	Let $f : u(x) \rightarrow y$ and $g : u(x') \rightarrow y'$ be arrows in $\Q$ with 
	adjunct arrows $f^\sharp : x \rightarrow v(y)$ and 
	$g^\sharp : x' \rightarrow v(y')$ in $\P$ respectively. Then there is an induced adjunction
	$$
		g_\circledast{u(-)}f^\circledast: \Hor(\P)(x,x') \rightleftarrows \Hor(\Q)(y,y')	: g^{\sharp,\circledast}{v(-)}f^\sharp_\circledast.
	$$
\end{proposition}
\begin{proof}
	Let $F : x \rightarrow x'$ be a proarrow of $\P$.
	In light of \ref{prop.cart-comp-conj}, we obtain a (natural) pullback square
	\[
		\begin{tikzcd}
			\map_{\Hor(\Q)(y, y')}(\cocart{g}{u(F)}{f}, G) \arrow[r]\arrow[d] & \map_{\Q_1}(u(F), G) \arrow[d] \\
			\ast \arrow[r, "{\{(f,g)\}}"] & \map_{\Q_0}(u(x),y) \times \map_{\Q_0}(u(x'), y').
		\end{tikzcd}
	\]
	Since the maps 
	$d_i : [0] \rightarrow [1]$ are inert, there is a commutative diagram 
	\[
	\begin{tikzcd}[row sep = tiny, column sep = large]
		\P_1\arrow[dd,"d_i^*"'] \arrow[dr, bend right = 10pt, "u_1"'name=t1]\arrow[rr,equal, bend left = 15pt, ""'name=f1] & & \P_1\arrow[dd,"d_i^*"]\\ 
		& |[alias=m1]|\Q_1\arrow[ur, bend right = 10pt, "v_1"'] & \\
		\P_0 \arrow[dr, bend right = 10pt, "u_0"'name=t2]\arrow[rr,equal, bend left = 15pt, ""'name=f2] & & \P_0 \\
		& |[alias=m2]|\Q_0 \arrow[ur, "v_0"', bend right = 10pt] &
		\arrow[from=f1,to=t1, Rightarrow, shorten <= 2pt, shorten >= 5pt]
		\arrow[from=f2,to=t2, Rightarrow, shorten <= 2pt, shorten >= 5pt]
		\arrow[from=m1, to=m2, "d_i^*", crossing over] 
	\end{tikzcd}
\]
	in $\CAT_\infty$, where the top and bottom 2-cell are units of the level-wise adjunctions of $(u,v)$.
	Consequently, we obtain a commutative diagram
	\[
		\begin{tikzcd}
			\map_{\Q_1}(u(F), G) \arrow[r, "\simeq"]\arrow[d] & \map_{\P_1}(F,v(G)) \arrow[d] \\ 
			\map_{\Q_0}(u(x),y) \times \map_{\Q_0}(u(x'), y') \arrow[r, "\simeq"] & \map_{\P_0}(x, v(y)) \times \map_{\P_0}(x', v(y')),
		\end{tikzcd}
	\]
	where the horizontal maps are given by applying $v$ and restricting along units. Combining this with the above, 
	we obtain a pullback square
	\[
		\begin{tikzcd}
			\map_{\Hor(\Q)(q, q')}(\cocart{g}{u(F)}{f}, G) \arrow[r]\arrow[d] & \map_{\P_1}(F, v(G)) \arrow[d] \\
			\ast \arrow[r, "{(f^\sharp, g^\sharp)}"] & \map_{\Q_0}(x,v(y)) \times \map_{\P_0}(x', v(y')).
		\end{tikzcd}
	\]
	But this pullback square also computes $\map_{\Hor(\P)(x,x')}(F, g^{\sharp,\circledast}{v(G)}f^\sharp_\circledast)$.
\end{proof}

\begin{remark}\label{rem.lax-adj-hor}
	In particular, we may apply \ref{prop.lax-adj-hor} to the identity adjunction 
	$$
	\id : \P \rightleftarrows \P : \id,
	$$
	so that we obtain an adjunction 
	$$
		\cocart{g}{(-)}{f} : \Hor(\P)(x,x') \rightleftarrows \Hor(\P)(y,y')	: \cart{g}{(-)}{f}
	$$
	for every two arrows $f : x \rightarrow y$ and $g : x' \rightarrow y'$. This was also observed in \cite[Remark 3.34]{EquipI}.
\end{remark}

\subsection{Preservation of cartesian cells}\label{ssection.lax-fun-cart}

One may already deduce that a normal lax right adjoint between $\infty$-equipments preserves all companions and conjoints using \ref{prop.lax-adj-hor}. 
But a stronger statement is true, namely, we have the following generalization of a result by Shulman in the context of (strict) double categories \cite[Proposition 6.8]{ShulmanFramedBicats}: 

\begin{proposition}\label{prop.lax-pres-cart-cells}
	Let $h : \P \rightarrow \Q$ be a lax functor between locally complete double $\infty$-categories. Suppose 
	that $f : a \rightarrow x$ and $g : b \rightarrow y$ are vertical arrows of $\P$ so that: 
	\begin{itemize}
		\item $f$ and $h(f)$ admit a companion in $\P$ and $\Q$ respectively,
		\item $g$ and $h(g)$ admit a conjoint in $\P$ and $\Q$ respectively.
	\end{itemize} Then any cartesian cell $c$ in $\P$ of the form
	\[
		\begin{tikzcd}
		a \arrow[d, "f"'] \arrow[r, "G"] & b \arrow[d, "g"] \\ 
		x \arrow[r, "F"] & y
		\end{tikzcd}
	\]
	is carried to a cartesian cell in $\Q$ by $h$.
\end{proposition}

Let us first highlight two corollaries of this result. The proof follows below.

\begin{corollary}\label{cor.lax-fun-eq-pres-cart-cells}
	Let $h : \P \rightarrow \Q$ be a lax functor between $\infty$-equipments. Then $h$ preserves cartesian cells. 
	Consequently, if $f : a\rightarrow x$, $g : b \rightarrow y$ are arrows in $\P$ then there is a natural equivalence 
	$$
	h(\cart{g}{F}{f}) \simeq \cart{h(g)}{h(F)}{h(f)}
	$$
	for every proarrow $F : x \rightarrow y$ in $\P$.
\end{corollary}
\begin{proof}
	This follows directly from \ref{prop.lax-pres-cart-cells} and \ref{prop.cart-comp-conj}.
\end{proof}

\begin{corollary}\label{cor.normal-lax-fun-eq-pres-compconj}
	Let $h : \P \rightarrow \Q$ be a normal lax functor between $\infty$-equipments. Then $h$ preserves companions and conjoints, i.e.\ 
	$h(f_\circledast) \simeq h(f)_\circledast$ and $h(f^\circledast) \simeq h(f)^\circledast$ for every arrow $f : x\rightarrow y$ in $\P$.
\end{corollary}

The proof of Shulman of \ref{prop.lax-pres-cart-cells} in the case of double categories is via chasing through various pastings of 2-cells. 
We will see that the same proof applies in the $\infty$-categorical context as well.
First, we introduce auxiliary notation that will ease the bookkeeping involved in the proof of \ref{prop.lax-pres-cart-cells}.

\begin{notation}\label{not.lax-naturality-grids}
	Let $h : \P \rightarrow \Q$ be a lax functor between $\infty$-equipments. Suppose 
	that $\sigma : [m,0] \rightarrow \P$ is a functor. Then we may view $\sigma$ as an object of the $\infty$-category $\P_m$.
	Now if $\alpha : [n] \rightarrow [m]$ is a structure map, then $h$ provides a lax naturality square 
	\[
	\begin{tikzcd}
		\P_m \arrow[r, "h_m"]\arrow[d, "\alpha^*"'] & |[alias=f1]|\Q_m\arrow[d, "\alpha^*"] \\
		|[alias=t1]|\P_n \arrow[r, "h_n"'] & \Q_n.
		\arrow[from=f1, to=t1, Rightarrow]
	\end{tikzcd}
	\] 
	The component of the natural transformation at $\sigma$ corresponds to a map $[1] \rightarrow \Q_n$. In turn, 
	this may be viewed as a functor from $[n,1]$ to $\Q$. We will denote this by
	$$
	h(\sigma, \alpha) : [n, 1] \rightarrow \Q.
	$$
	Note that this functor restricts to
	to $\alpha^*h_m(\sigma)$ on $[n, \{0\}]$ and 
	to $h_n(\alpha^*\sigma)$ on $[n, \{1\}]$. 
\end{notation}
	
\begin{remark}\label{rem.notation-lax}
	There is compatibility in the functors of \ref{not.lax-naturality-grids}. Suppose that there is a further structure map 
	$\beta : [i] \rightarrow [n]$. Then the lax naturality square associated with the composite $\alpha\beta$ 
	is given by the outer square in the pasting 
	\[
	\begin{tikzcd}
		\P_m \arrow[r, "h_m"]\arrow[d, "\alpha^*"'] & |[alias=f1]|\Q_m\arrow[d, "\alpha^*"] \\
		|[alias=t1]|\P_n \arrow[r, "h_n"] \arrow[d, "\beta^*"'] & |[alias=f2]|\Q_n\arrow[d, "\beta^*"] \\ 
		|[alias=t2]| \P_i \arrow[r, "h_i"'] & \Q_i.
		\arrow[from=f1, to=t1, Rightarrow]
		\arrow[from=f2, to=t2, Rightarrow]
	\end{tikzcd}
	\] 
	Unraveling the definitions, this means that $h(\sigma, \alpha\beta)$ can be obtained by vertically pasting 
	the grids $h(\sigma, \alpha) \circ [\beta, 1]$ and $h(\alpha^*\sigma, \beta)$. 
	More precisely,
	one may consider the commutative square
	\[ 
		\begin{tikzcd}[column sep = large]
			{[i,0]} \arrow[rr, "{[i,\{0\}]}"]\arrow[d, "{[i,\{1\}]}"'] && {[i,1] } \arrow[d, "{h(\alpha^*\sigma, \beta)}"] \\ 
			{[i,1]} \arrow[r, "{[\beta, 1]}"] & {[n,1]} \arrow[r, "{h(\sigma, \alpha)}"] &  \P,
		\end{tikzcd}
	\]
	which induces a map $[i,2] \rightarrow \P$ by the universal property 
	of the pushout. If we restrict this map to 
	$[i, \{0\leq 1\}]$, then this is equivalent to $h(\sigma, \alpha\beta)$.
\end{remark}

\begin{proof}[Proof of \ref{prop.lax-pres-cart-cells}]
    We closely follow the proof by Shulman \cite[Proposition 6.8]{ShulmanFramedBicats} for the case of (strict) equipments.
    In view of \ref{prop.cart-comp-conj}, the cartesian cell $c$ of $\P$ decomposes as
        \[
            c = \begin{tikzcd}
                a \arrow[r, "{f}_\circledast"name=f1]\arrow[d,"f"'] & x \arrow[r,	"F"]\arrow[d,equal] & y \arrow[r, "g^\circledast"name=f2] \arrow[d, equal] & b\arrow[d, "g"] \\
                x \arrow[r, equal, ""name=t1] & x \arrow[r, "F"] & y \arrow[r,equal, ""name=t2] & y,
            \end{tikzcd}
        \]
    where the left and right cells are given by the companion and conjoint counit respectively.
    There is a similarly formed cartesian 2-cell
    \[
            c' = \begin{tikzcd}
                h(a) \arrow[r, "h(f)_\circledast"name=f1]\arrow[d,"{h(f)}"'] & h(x) \arrow[r,	"hF"]\arrow[d,equal] & h(y) \arrow[r, "h(g)^\circledast"name=f2] \arrow[d, equal] & h(b)\arrow[d, "{h(g)}"] \\
                h(x) \arrow[r, equal, ""name=t1] & h(x) \arrow[r, "hF"] & h(y) \arrow[r,equal, ""name=t2] & h(y)
            \end{tikzcd}
        \]
    in $\Q$. We will write $G'$ for the source horizontal arrow of $c'$. 
    The universal property of cartesian cells (see \cite[Corollary 3.21]{EquipI}) asserts that there exists a (essentially) unique 
    comparison 2-cell 
    \[
        \alpha = \begin{tikzcd}
            h(a) \arrow[r, "hG"]\arrow[d,equal] & h(b) \arrow[d,equal] \\
            h(a) \arrow[r, "G'"] & h(b)
        \end{tikzcd}
    \]
    so that 
    \[
    \begin{tikzcd}
        h(a) \arrow[r, "hG"name=h1]\arrow[d,equal] & h(b) \arrow[d,equal] \\
        h(a) \arrow[r, "G'"name=h2]\arrow[d, "{h(f)}"'] & h(b) \arrow[d, "{h(g)}"] \\
        h(x) \arrow[r, "hF"'name=h3] & h(y). 
		\arrow[from=h1,to=h2,phantom, "\scriptstyle{\alpha}"]
		\arrow[from=h2,to=h3,phantom, "\scriptstyle{c'}"]
    \end{tikzcd}
    \simeq 
    \begin{tikzcd}
        h(a) \arrow[r, "hG"name=h1]\arrow[d, "{h(f)}"'] & h(b) \arrow[d, "{h(g)}"] \\
        h(x) \arrow[r, "hF"'name=h2] & h(y). 
		\arrow[from=h1,to=h2,phantom, "\scriptstyle{h(c)}"]
    \end{tikzcd}
    \]
    We will show that $\alpha$ is an equivalence in $\Hor(\Q)(h(a), h(b))$ by exhibiting an 
    inverse. This then proves the lemma.
    
   Recall that companionship and conjunction units are \textit{cocartesian}, see \cite[Proposition 3.24]{EquipI}. 
   Consequently, we may use the universal property of cocartesian cells so that we obtain 2-cells
    \[
        \gamma = \begin{tikzcd}
            h(a) \arrow[r, "h(f)_\circledast"]\arrow[d,equal] & h(x)\arrow[d,equal] \\
            h(a) \arrow[r, "h(f_\circledast)"] & h(x),
        \end{tikzcd}
        \quad 
        \delta = \begin{tikzcd}
            h(y) \arrow[r, "h(g)^\circledast"]\arrow[d,equal] & h(b)\arrow[d,equal] \\
            h(y) \arrow[r, "h(g^\circledast)"] & h(b),
        \end{tikzcd}
    \]
    that satisfy the relations
    \[
        \begin{tikzcd}
            h(a) \arrow[r, equal]\arrow[d, equal] & h(a)\arrow[d, "{h(f)}"] \\
            h(a) \arrow[r, "h(f)_\circledast"name=h1]\arrow[d,equal] & h(x)\arrow[d,equal] \\
            h(a) \arrow[r, "h(f_\circledast)"'name=h2] & h(x)
			\arrow[from=h1, to=h2, phantom, "\scriptstyle{\gamma}"]
        \end{tikzcd}	
        \simeq 
        \begin{tikzcd}
            h(a) \arrow[r, equal, ""name=f1]\arrow[d, equal] & h(a)\arrow[d, equal] \\
            h(a) \arrow[r, "h(\id_a)"'name=t1]\arrow[d,equal] & h(a)\arrow[d,"{h(f)}"] \\
            h(a) \arrow[r, "h(f_\circledast)"'] & h(x),
            \arrow[from=f1,to=t1,phantom, "\scriptstyle{h(\id_a, s_0)}"]
        \end{tikzcd}
        \quad
        \begin{tikzcd}
            h(b) \arrow[r, equal]\arrow[d, "{h(g)}"'] & h(b)\arrow[d, equal] \\
            h(y) \arrow[r, "h(g)^\circledast"name=h1]\arrow[d,equal] & h(b)\arrow[d,equal] \\
            h(y) \arrow[r, "h(g^\circledast)"'name=h2] & h(b),
			\arrow[from=h1, to=h2, phantom, "\scriptstyle{\delta}"]
        \end{tikzcd}	
        \simeq 
        \begin{tikzcd}
            h(b) \arrow[r, equal,""name=f1]\arrow[d, equal] & h(b)\arrow[d, equal] \\
            h(b) \arrow[r, "h(\id_b)"'name=t1]\arrow[d,"{h(g)}"'] & h(b)\arrow[d,equal] \\
            h(y) \arrow[r, "h(g^\circledast)"] & h(b).
            \arrow[from=f1,to=t1,phantom, "\scriptstyle {h(\id_b, s_0)}"]
        \end{tikzcd}	
    \]
	Here, the top unmarked 2-cells are given by the companionship/conjunction units. The bottom unmarked 2-cells are given 
	by images of companionship/conjunction units under $h$.
    Now we may form the candidate for the inverse to $\alpha$ by the pasting
    \[
        \beta := \begin{tikzcd}
        h(a) 	\arrow[r, "h(f)_\circledast"name=h1]\arrow[d,equal]& h(x) \arrow[d, equal]\arrow[r,"hF"] & h(y) \arrow[d,equal]\arrow[r, "h(g)^\circledast"name=h3] & h(b)\arrow[d,equal]\\
        h(a) \arrow[r, "h(f_\circledast)"name=h2]\arrow[d,equal] & h(x) \arrow[r, "hF"name=f1] & h(y) \arrow[r, "h(g^\circledast)"name=h4] & h(b)\arrow[d,equal] \\
        h(a) \arrow[rrr, "hG"'name=t1] &&& h(b).
        \arrow[from=f1,to=t1, phantom, "\scriptstyle{h((f_\circledast, F, g^\circledast), d_2d_1)}"]
		\arrow[from=h1,to=h2, phantom, "\scriptstyle{\gamma}"]
		\arrow[from=h3,to=h4, phantom, "\scriptstyle{\delta}"]
        \end{tikzcd}
    \]
	We will verify that $\beta\alpha \simeq \id_{hG}$ and $\alpha\beta \simeq \id_{G'}$ in two steps.
    
    Let us first look at the vertical composition $\beta\alpha$. We use the defining feature of $\alpha$ and the conjunction and companionship triangle identities 
	to rewrite this as
    \[
        \beta\alpha \simeq 
            \begin{tikzcd}
                h(a) \arrow[r, equal]\arrow[d,equal] & h(a) \arrow[d, "{h(f)}"']\arrow[r, "hG"name=h1] & h(b)\arrow[d, "{h(g)}"]\arrow[r,equal] & h(b) \arrow[d,equal] \\
                h(a) 	\arrow[r, "h(f)_\circledast"name=h3]\arrow[d,equal]& h(x) \arrow[d, equal]\arrow[r,"hF"name=h2] & h(y) \arrow[d,equal]\arrow[r, "h(g)^\circledast"name=h5] & h(b)\arrow[d,equal]\\
                h(a) \arrow[r, "h(f_\circledast)"name=h4]\arrow[d,equal] & h(x) \arrow[r, "hF"name=f1] & h(y) \arrow[r, "h(g^\circledast)"name=h6] & h(b)\arrow[d,equal] \\
                h(a) \arrow[rrr, "hG"'name=t1] &&& h(b).
				\arrow[from=f1,to=t1, phantom, "\scriptstyle{h((f_\circledast, F, g^\circledast), d_2d_1)}"]
				\arrow[from=h1, to=h2, phantom, "\scriptstyle{h(c)}"]
				\arrow[from=h3,to=h4, phantom, "\scriptstyle{\gamma}"]
				\arrow[from=h5,to=h6, phantom, "\scriptstyle{\delta}"]
            \end{tikzcd}
	\] 
	Next, we use the definition of $\gamma$ and $\delta$ to rewrite this as
	\[ 
        \beta\alpha \simeq
            \begin{tikzcd}
                h(a) \arrow[r, equal, ""name=h1]\arrow[d,equal] & h(a) \arrow[d, equal]\arrow[r, "hG"] & h(b)\arrow[d, equal]\arrow[r,equal, ""name=h3] & h(b) \arrow[d,equal] \\
                h(a) 	\arrow[r, "h(\id_a)"name=h2]\arrow[d,equal]& h(a) \arrow[d, "{h(f)}"']\arrow[r,"hG"name=h5] & h(b) \arrow[d,"{h(g)}"]\arrow[r, "h(\id_b)"name=h4] & h(b)\arrow[d,equal]\\
                h(a) \arrow[r, "h(f_\circledast)"]\arrow[d,equal] & h(x) \arrow[r, "hF"name=f1] & h(y) \arrow[r, "h(g^\circledast)"] & h(b)\arrow[d,equal] \\
                h(a) \arrow[rrr, "hG"'name=t1] &&& h(b).
				\arrow[from=f1,to=t1, phantom, "\scriptstyle{h((f_\circledast, F, g^\circledast), d_2d_1)}"]
				\arrow[from=h1,to=h2,phantom, "\scriptstyle{h(\id_a, s_0)}"]
				\arrow[from=h3,to=h4,phantom, "\scriptstyle{h(\id_b, s_0)}"]
				\arrow[from=h5, to=f1, phantom, "\scriptstyle{h(c)}"]
            \end{tikzcd}
    \]
    This is precisely given by
    \[
        \beta\alpha\simeq \begin{tikzcd}
                h(a) \arrow[r, equal, ""name=h1]\arrow[d,equal] & h(a) \arrow[d, equal]\arrow[r, "hG"] & h(b)\arrow[d, equal]\arrow[r,equal, ""name=h3] & h(b) \arrow[d,equal] \\
                h(a) 	\arrow[r, "h(\id_a)"name=h2]\arrow[d,equal]& h(a) \arrow[r,"hG"name=f1] & h(b) \arrow[r, "h(\id_b)"name=h4] & h(b)\arrow[d,equal]\\
                h(a) \arrow[rrr, "hG"'name=t1] &&& h(b).
                \arrow[from=f1,to=t1, phantom, "\scriptstyle{h((\id_a, G, \id_b), d_2d_1)}"]
				\arrow[from=h1,to=h2,phantom, "\scriptstyle{h(\id_a, s_0)}"]
				\arrow[from=h3,to=h4,phantom, "\scriptstyle{h(\id_b, s_0)}"]
            \end{tikzcd}
    \]
    Note that the pasting of the top row of 2-cells is equivalent to $h(G, s_0s_3) \circ [d_2d_1, 1]$.
	Thus \ref{rem.notation-lax} implies that the total pasting of the top and bottom rows  is given by $h(G, d_2d_1s_0s_3) \simeq \id_{hG}$.
    
    Next, we will consider the vertical composite $\alpha \beta$ and show that it is equivalent to $\id_{G'}$. 
    To this end, the universal property of the cartesian cell $c'$
    implies that it suffices to verify that the pasting 
    \[
    c'\alpha\beta \simeq \begin{tikzcd}
        h(a) 	\arrow[r, "h(f)_\circledast"name=h1]\arrow[d,equal]& h(x) \arrow[d, equal]\arrow[r,"hF"] & h(y) \arrow[d,equal]\arrow[r, "h(g)^\circledast"name=h3] & h(b)\arrow[d,equal]\\
        h(a) \arrow[r, "h(f_\circledast)"name=h2]\arrow[d,equal] & h(x) \arrow[r, "hF"name=h5] & h(y) \arrow[r, "h(g^\circledast)"name=h4] & h(b)\arrow[d,equal] \\
        h(a) \arrow[rrr, "hG"name=f1]\arrow[d, "{h(f)}"'] &&& h(b)\arrow[d, "{h(g)}"] \\
        h(x) \arrow[rrr, "hF"'name=t1] &&& h(y)
		\arrow[from=f1,to=t1, phantom, "\scriptstyle{h(c)}"]
		\arrow[from=h1,to=h2, phantom, "\scriptstyle{\gamma}"]
		\arrow[from=h3,to=h4, phantom, "\scriptstyle{\delta}"]
		\arrow[from=h5,to=f1, phantom, "\scriptstyle{h((f_\circledast, F, g^\circledast), d_2d_1)}"]
        \end{tikzcd}
    \]
    recovers $c'$. This can be rewritten as
    \[
    c'\alpha\beta \simeq \begin{tikzcd}
        h(a) 	\arrow[r, "h(f)_\circledast"name=h1]\arrow[d,equal]& h(x) \arrow[d, equal]\arrow[r,"hF"] & h(y) \arrow[d,equal]\arrow[r, "h(g)^\circledast"name=h3] & h(b)\arrow[d,equal]\\
        h(a) \arrow[r, "h(f_\circledast)"'name=h2]\arrow[d,"{h(f)}"'] & h(x) \arrow[d,equal]\arrow[r, "hF"name=f1] & h(y)\arrow[d,equal] \arrow[r, "h(g^\circledast)"'name=h4] & h(b)\arrow[d, "{h(g)}"] \\
        h(x) \arrow[r, "h(\id_x)"]\arrow[d,equal] & h(x) \arrow[r, "hF"name=f1] & h(y) \arrow[r, "h(\id_y)"] & h(y)\arrow[d,equal] \\
        h(x) \arrow[rrr, "hF"'name=t1] &&& h(y).
        \arrow[from=f1,to=t1, phantom, "\scriptstyle{h((\id_x, F, \id_y), d_2d_1)}"]
		\arrow[from=h1,to=h2, phantom, "\scriptstyle{\gamma}"]
			\arrow[from=h3,to=h4, phantom, "\scriptstyle{\delta}"]
        \end{tikzcd}
    \]
	In light of the definition of $\gamma$ and the companionship triangle identities, we obtain
	\[
		\gamma \simeq \begin{tikzcd}
            h(a) \arrow[r, equal]\arrow[d, equal] & h(a)\arrow[d, "{h(f)}"] \arrow[r,"h(f)_\circledast"] & h(x)\arrow[d,equal]\\
            h(a) \arrow[r, "h(f)_\circledast"name=h1]\arrow[d,equal] & h(x)\arrow[d,equal] \arrow[r,equal] & h(x)\arrow[d,equal] \\
            h(a) \arrow[r, "h(f_\circledast)"'name=h2] & h(x) \arrow[r,equal] & h(x)
			\arrow[from=h1, to=h2, phantom, "\scriptstyle{\gamma}"]
        \end{tikzcd}		
		\simeq 
		\begin{tikzcd}
            h(a) \arrow[r, equal, ""name=f1]\arrow[d, equal] & h(a)\arrow[d, equal] \arrow[r, "h(f)_\circledast"] & h(x) \arrow[d,equal] \\
            h(a) \arrow[r, "h(\id_a)"'name=t1]\arrow[d,equal] & h(a)\arrow[d,"{h(f)}"] \arrow[r, "h(f)_\circledast"] & h(x)\arrow[d,equal]\\
            h(a) \arrow[r, "h(f_\circledast)"'] & h(x) \arrow[r,equal] & h(x).
            \arrow[from=f1,to=t1,phantom, "\scriptstyle{h(\id_a, s_0)}"]
        \end{tikzcd}
	\]
    Consequently, we compute that 
    \[
        \begin{tikzcd}
        h(a) 	\arrow[r, "h(f)_\circledast"name=h1]\arrow[d,equal]& h(x) \arrow[d, equal] \\
        h(a) \arrow[r, "h(f_\circledast)"'name=h2]\arrow[d,"{h(f)}"'] & h(x) \arrow[d,equal] \\
        h(x) \arrow[r, "h(\id_x)"'] & h(x)
		\arrow[from=h1,to=h2, phantom, "\scriptstyle{\gamma}"]
        \end{tikzcd}
        \simeq 
		\begin{tikzcd}
            h(a) \arrow[r, equal, ""name=f1]\arrow[d, equal] & h(a)\arrow[d, equal] \arrow[r, "h(f)_\circledast"] & h(x) \arrow[d,equal] \\
            h(a) \arrow[r, "h(\id_a)"'name=t1]\arrow[d,equal] & h(a)\arrow[d,"{h(f)}"] \arrow[r, "h(f)_\circledast"] & h(x)\arrow[d,equal]\\
            h(a) \arrow[r, "h(f_\circledast)"']\arrow[d,"{h(f)}"'] & h(x) \arrow[r,equal]\arrow[d,equal] & h(x) \arrow[d,equal] \\
			h(x) \arrow[r, "h(\id_x)"'] & h(x) \arrow[r,equal] & h(x)
            \arrow[from=f1,to=t1,phantom, "\scriptstyle{h(\id_a, s_0)}"]
        \end{tikzcd}
		\simeq 
		\begin{tikzcd}
            h(a) \arrow[r, equal, ""name=f1]\arrow[d, equal] & h(a)\arrow[d, equal] \arrow[r, "h(f)_\circledast"] & h(x) \arrow[d,equal] \\
            h(a) \arrow[r, "h(\id_a)"'name=t1]\arrow[d,"{h(f)}"'] & h(a)\arrow[d,"{h(f)}"] \arrow[r, "h(f)_\circledast"] & h(x)\arrow[d,equal]\\
            h(a) \arrow[r, "h(\id_x)"']& h(x) \arrow[r,equal]& h(x) 
			\arrow[from=f1,to=t1,phantom, "\scriptstyle{h(\id_a, s_0)}"].
        \end{tikzcd}
    \]
   Hence, it follows that
    \[
        \begin{tikzcd}
			h(a) 	\arrow[r, "h(f)_\circledast"name=h1]\arrow[d,equal]& h(x) \arrow[d, equal] \\
			h(a) \arrow[r, "h(f_\circledast)"'name=h2]\arrow[d,"{h(f)}"'] & h(x) \arrow[d,equal] \\
			h(x) \arrow[r, "h(\id_x)"'] & h(x)
			\arrow[from=h1,to=h2, phantom, "\scriptstyle{\gamma}"]
			\end{tikzcd}
			\simeq 
        \begin{tikzcd}
            h(a) \arrow[d, "{h(f)}"'] \arrow[r, "h(f)_\circledast"] & h(x) \arrow[d,equal] \\
            h(x)\arrow[d,equal] \arrow[r, equal, ""name=f1] & h(x) \arrow[d,equal]\\
            h(x)\arrow[r,"h(\id_x)"'name=t1] & h(x)
			\arrow[from=f1,to=t1,phantom, "\scriptstyle{h(\id_x, s_0)}"].
        \end{tikzcd}
    \]
    A similar formula may be obtained for the vertical composite 
    of $\delta$ with the image of the conjunction unit associated with $g$. With these computations, 
	we now deduce that 
    \[
    c'\beta\alpha \simeq \begin{tikzcd}
        h(a) 	\arrow[r, "h(f)_\circledast"]\arrow[d,"{h(f)}"']& h(x) \arrow[d, equal]\arrow[r,"hF"] & h(y) \arrow[d,equal]\arrow[r, "h(g)^\circledast"] & h(b)\arrow[d,"{h(g)}"]\\
        h(x) \arrow[r, equal, ""name=h1]\arrow[d,equal] & h(x) \arrow[d,equal]\arrow[r, "hF"] & h(y)\arrow[d,equal] \arrow[r, equal, ""name=h3] & h(y) \arrow[d,equal]\\
        h(x) \arrow[r, "h(\id_x)"name=h2]\arrow[d,equal] & h(x) \arrow[r, "hF"name=f1] & h(y) \arrow[r, "h(\id_y)"name=h4] & h(y)\arrow[d,equal] \\
        h(x) \arrow[rrr, "hF"'name=t1] &&& h(y).
		\arrow[from=f1,to=t1, phantom, "\scriptstyle{h((\id_x, F, \id_y), d_2d_1)}"]
		\arrow[from=h1,to=h2,phantom, "\scriptstyle{h(\id_x, s_0)}"]
		\arrow[from=h3,to=h4,phantom, "\scriptstyle{h(\id_y, s_0)}"]
        \end{tikzcd}
    \]
    Note that the composite of the middle row of 2-cells is equivalent to $h(F, s_0s_3) \circ [d_2d_1, 1]$. 
	It follows again from \ref{rem.notation-lax} that the composite of the middle and bottom rows is given by 
	$h((\id_x, F, \id_y), d_2d_1s_0s_3) = \id_{hF}$. This shows that $c'\beta\alpha$ is equivalent to $c'$, as desired.
\end{proof}

\section{Reflective, proper and smooth lax functors} 

In this section, we will highlight four different kinds of lax functors:
\begin{enumerate}
	\item the \textit{reflective} ones: these are generalizations of reflective inclusions of $\infty$-categories,
	\item the \textit{horizontally locally reflective} ones: these induce reflective inclusions on horizontal mapping $\infty$-categories,
	\item the \textit{proper} ones: these are functors that interact well with formal category theory, in the sense that they preserve all internal left Kan extensions,
	\item the \textit{smooth} ones: these are dual to (3).
\end{enumerate}
We will see that (2), (3) and (4) all contain (1). We will discuss some examples as well.

\subsection{Reflective lax functors}\label{ssection.refl-lax-functors} We commence by discussing variants (1) and (2) in this subsection.

\begin{definition}
	A lax functor $i : \Q \rightarrow \P$ between double $\infty$-categories is called a \textit{reflective inclusion} if $i$ 
	is normal, fully faithful level-wise, and admits a left adjoint.
\end{definition}

\begin{example}\label{ex.geom-emb-cotensor}
	On account of \ref{ex.dbl-adj-geom-morph}, every geometric embedding between $\infty$-toposes gives rise to reflective inclusion between the respective $\infty$-equipments 
	of internal categories. 
	
	Since every $\infty$-topos $\E$ admits a geometric embedding $i : \E \rightarrow \PSh(T)$ into a presheaf topos for some small $\infty$-category $T$, 
	this particularly implies that there is a reflective inclusion
	$$
	i_* : \CCAT_\infty(\E) \rightarrow [T^\op, \CCAT_\infty]
	$$
	into a vertical cotensor product of $\CCAT_\infty$.
	Here, we used the identifications of \ref{ex.cotensors-ccat}.
	This is an important example, because this allows us to reduce questions 
	about the equipment $\CCAT_\infty(\E)$ to the equipment $\CCAT_\infty$ using the theory developed in \cite{FunDblCats}. We will see this principle in action on a few occasions.
\end{example}

Reflective inclusions between equipments enjoy the following properties: 

\begin{proposition}\label{prop.refl-incl-formal-cat-theory}
	Let $i : \Q \rightarrow \P$ be a reflective inclusion between $\infty$-equipments. Then $i$ has the following properties:
	\begin{enumerate} 
		\item it preserves and reflects cartesian cells,
		\item it reflects cocartesian cells,
		\item it preserves and reflects fully faithful arrows  (see \cite[Definition 6.1]{EquipI}),
		\item it reflects exact squares (see \cite[Definition 6.24]{EquipI}).
	\end{enumerate}
\end{proposition}
\begin{proof}
	Let $L$ be the left adjoint to $i$. Then we have that $L_1i_1 \simeq \id_{\Q_1}$ and $L_0i_0 \simeq \id_{\Q_0}$. Since $L$ is strict, it preserves companions, conjoints and (co)cartesian cells. We have shown in \ref{ssection.lax-fun-cart} 
	that $i$ preserves companions, conjoints and cartesian cells. Now assertions (1) and (2) follow immediately from these observations.
	Note that (3) follows directly from the definition of fully faithful arrows, (1) and the 
	fact that normal lax functors preserve the identity 2-cells associated to \textit{vertical} arrows. 
	
	Let us show assertion (4). Suppose that we have a lax commutative square 
	\[
		\alpha = \begin{tikzcd}
			a \arrow[r, "f"]\arrow[d,"v"'] &|[alias=f]| b \arrow[d, "w"] \\ 
			|[alias=t]|c \arrow[r, "g"'] & d 
			\arrow[from=f,to=t, Rightarrow]
		\end{tikzcd}
	\]
	in $\Vert(\Q)$. Then this factors uniquely as a vertical composition
	\[
		\begin{tikzcd}
			a \arrow[d, "v"'] \arrow[r, equal] & a \arrow[d, "f"] \\
			c \arrow[r, "\cart{w}{}{g}"]\arrow[d, "g"'] & b\arrow[d,"w"] \\
			d \arrow[r,equal] & d
		\end{tikzcd}
	\] 
	in $\Q$ where the bottom 2-cell is cartesian. By definition, the lax square $\alpha$ is exact if and only if the top 
	2-cell in this factorization is cocartesian. On account of (2), it suffices to check this for its image under $i$ in $\P$. Since the (vertical) factorization of
	above is preserved by $i$ and the image of the bottom 2-cell is still cartesian on account of (1), this is precisely the condition for the image $i(\alpha)$ to be exact. 
\end{proof}

As announced, there is also a local variant on reflective functors:

\begin{definition}\label{def.hor-loc-refl}
	A lax functor $i : \Q \rightarrow \P$ between double $\infty$-categories is called \textit{horizontally locally reflective} if the following two 
	conditions are met:
	\begin{enumerate}
		\item for every two 
		objects $x,y \in \Q$, the induced functor on horizontal mapping $\infty$-categories 
			$$
			i_{x,y} : \Hor(\Q)(x,y) \rightarrow \Hor(\P)(i(x), i(y))
			$$
		is an inclusion of a reflective subcategory,
		\item for every three objects $x,y,z \in \Q$, the natural transformation in the lax naturality square
		\[
		\begin{tikzcd}
			\Hor(\Q)(y,z) \times \Hor(\Q)(x,y) \arrow[r]\arrow[d] & |[alias=t]| \Hor(\Q)(x,z) \arrow[d] \\
			|[alias=f]| \Hor(\P)(i(y), i(z)) \times \Hor(\P)(i(x), i(y)) \arrow[r]  & \Hor(\P)(i(x),i(z))
			\arrow[from=f,to=t,Rightarrow, shorten <= 10pt, shorten >= 10pt]
		\end{tikzcd}
		\] 
		is a local equivalence with respect to $i_{x,z}$ component-wise.
	\end{enumerate} 
	In this case and context, we will refer to the local objects and local equivalences of the reflective inclusions $i_{x,y}$
		as the \textit{local horizontal arrows} and \textit{2-cells} respectively. 
\end{definition}

\begin{remark}
	In the context of \ref{def.hor-loc-refl}, the 2-cell in the lax square can be recovered as follows. The lax naturality square 
	associated to the face map $d_1$ gives a natural transformation 
	$$
	[1] \times \P_2 \rightarrow \Q_1
	$$ 
	from $d_1^*i_2$ to $i_1d_1^*$. This fits in a commutative square
	\[
		\begin{tikzcd}
			{[1]} \times \P_2 \arrow[r]\arrow[d]\arrow[d] & \Q_1 \arrow[d] \\
			\P_0^{\times 2} \arrow[r, "i_0^{\times 2}"] & \Q_0^{\times 2},
		\end{tikzcd}
	\]
	where the right arrow is the source-target projection, and the left arrow is given by the projection $[1] \times \P_2 \rightarrow \P_2$ followed 
	by $(\{0\}^*, \{2\}^*) : \P_2 \rightarrow \P_0^{\times 2}$. 
	Taking fibers above $(x,z)$, we obtain a natural transformation 
	$$
	[1] \times (\P_2 \times_{\P_0^{\times 2}} \{(x,z)\}) \rightarrow \Hor(\Q)(Ix, Iz).
	$$
	The desired natural transformation is given by the restriction 
	$$
	[1] \times (\Hor(\P)(y,z) \times \Hor(\P)(x,y)) \rightarrow [1] \times (\P_2 \times_{\P_0^{\times 2}} \{(x,z)\})  \rightarrow \Hor(\Q)(Ix, Iz).
	$$
	Here we used that $\P_2 \simeq \P_1 \times_{\P_0} \P_1$.
\end{remark}

\begin{example}
	Let $\E$ be an $\infty$-topos. 
	In \ref{section.fib-equipments}, we will construct a horizontally locally reflective lax functor 
	$$
	\CCAT_\infty(\E) \rightarrow \SSPAN(\Cat_\infty(\E))
	$$
	that acts as the identity on objects and arrows, and carries proarrows to their classifying \textit{internal two-sided discrete fibrations}.
\end{example}

We have the following global characterization of horizontal locally reflective inclusions:

\begin{proposition}\label{prop.hor-refl-char-equipments}
	Suppose that $i : \Q \rightarrow \P$ is a lax functor between double $\infty$-categories. Let 
	us write 
	$$j : \Q_1 \rightarrow \P_1 \times_{\P_0^{\times 2}} \Q_0^{\times 2}$$
	for the fibered functor above $\Q_0^{\times 2}$ induced by $i_1$. Then $i$ is horizontally locally reflective if the following two conditions are met: 
	\begin{enumerate}[label=(\arabic*')]
		\item $j$ is fully faithful and admits a left adjoint fibered above $\Q_0^{\times 2}$,
		\item the natural transformation $[1] \times \Q_2 \rightarrow \P_1$ that sits in the lax naturality square 
		\[
		\begin{tikzcd}
			\Q_2 \arrow[r, "i_2"]\arrow[d,"d_1^*"'] &|[alias=f]| \P_2\arrow[d,"d_1^*"] \\
			|[alias=t]|\Q_1 \arrow[r,"i_1"'] & \P_1, 
			\arrow[from=f,to=t,Rightarrow]
		\end{tikzcd}
		\]
		and necessarily factors through $j$, is local with respect to $j$ component-wise.
	\end{enumerate}
	The converse statement is true if $\P$ and $\Q$ are $\infty$-equipments.
\end{proposition}
\begin{proof}
If (1') holds, then it is readily verified that condition (2') is equivalent to condition (2) of \ref{def.hor-loc-refl}. 
Condition (1) of this definition is precisely the assertion that 
$i$ is a fully faithful right adjoint on the fibers above $\Q_0^{\times 2}$. Hence, if (1') is met then (1) follows.

In the case that $\P$ and $\Q$ are $\infty$-equipments, then (1) implies (1'). This follows from the observation 
that $i$ preserves cartesian arrows on account of \ref{cor.lax-fun-eq-pres-cart-cells} 
and the fiberwise characterization of fibered adjoints between cartesian fibrations \cite[Proposition 7.3.2.6]{HA}.
\end{proof}

\begin{corollary}
	Every reflective inclusion of double $\infty$-categories is horizontally locally reflective.
\end{corollary}
\begin{proof}
	Let $i : \Q \rightarrow \P$ be a reflective inclusion of double $\infty$-categories. Then this admits a left adjoint $L : \P \rightarrow \Q$. In this case, 
	the composite functor $$K : \P_1 \times_{\P_0^{\times 2}} \Q_0^{\times 2} \rightarrow \P_1 \xrightarrow{L_1} \Q_1$$ 
	is a left adjoint to the functor $j$ of \ref{prop.hor-refl-char-equipments} so that $Kj \simeq \id$, and this is fibered above $\Q_0^{\times 2}$. 
	Thus condition (1') of \ref{prop.hor-refl-char-equipments} is fulfilled. Now, the 2-cell filling the lax naturality square \[
		\begin{tikzcd}
			\Q_2 \arrow[r, "i_2"]\arrow[d,"d_1^*"'] &|[alias=f]| \P_2\arrow[d,"d_1^*"] \\
			|[alias=t]|\Q_1 \arrow[r,"i_1"'] & \P_1, 
			\arrow[from=f,to=t,Rightarrow]
		\end{tikzcd}
		\] 
	is precisely the mate 
	$$
	d_1^*i_2 \rightarrow i_1L_1d_1^*i_2 \xrightarrow{\simeq} i_1d_1^*L_2i_2 \xrightarrow{\simeq} i_1d_1^*, 
	$$
	see \cite[Corollary 5.11]{FunDblCats}. Thus if we apply $L_1$ to this mate, we obtain an equivalence. This implies that condition (2') of 
	\ref{prop.hor-refl-char-equipments} is met.
\end{proof}

Our main interest in horizontally locally reflective lax functors stems from the following observation:

\begin{proposition}\label{prop.hor-refl-strong}
	Suppose that $i : \Q \rightarrow \P$ is a horizontally locally reflective lax functor. Let 
	$W : x \rightarrow y$ be a horizontal arrow and $z$ be an object of $\Q$. Suppose that the post-composition functor 
	$$
	iW \circ (-) : \Hor(\P)(i(z), i(x)) \rightarrow \Hor(\P)(i(z), i(y))
	$$ 
	has a right adjoint ${}^{iW}{(-)}$.
	Then the following assertions are equivalent:
	\begin{enumerate}
		\item the functor $iW \circ (-)$ preserves local horizontal 2-cells,
		\item the functor ${}^{iW}{(-)}{}$ preserves local horizontal arrows, 
		\item the functor
		$$
		W \circ (-) : \Hor(\Q)(z,x) \rightarrow \Hor(\Q)(z,y)
		$$ 
		admits a right adjoint ${}^{W}{(-)}{}$ as well.
	\end{enumerate}
	In this case, there is a commutative diagram 
	\[
		\begin{tikzcd}
			\Hor(\Q)(z,y) \arrow[r, "{}^{W}{(-)}{}"]\arrow[d] &  \Hor(\Q)(z,x)\arrow[d] \\
			\Hor(\P)(i(z), i(y)) \arrow[r, "{}^{iW}{(-)}{}"] & \Hor(\P)(i(z),i(x)).
		\end{tikzcd}
	\]
	Moreover, there is a similar statement for functors of pre-composition with $W$.
\end{proposition}
\begin{proof}
It is clear that (1) and (2) are equivalent. Recall that 
the natural transformation in the lax square
\[
	\begin{tikzcd}
		\Hor(\Q)(z,x)\arrow[r] \arrow[d, "W \circ (-)"'] & |[alias=f]|\Hor(\P)(i(z),i(x)) \arrow[d, "iW \circ (-)"] \\
		|[alias=t]|\Hor(\Q)(z, y) \arrow[r] & \Hor(\P)(i(z), i(y))
		\arrow[from=f,to=t,Rightarrow]
	\end{tikzcd}
\]
is a local 2-cell component-wise. Consequently, 
we obtain  a chain of natural equivalences
\begin{align*}
\map_{\Hor(\Q)(z,y)}(W \circ F, G) &\simeq \map_{\Hor(\P)(i(z), i(y))}(i(W \circ F), iG)  \\
&\xleftarrow{\simeq} \map_{\Hor(\P)(i(z), i(y))}(iW \circ iF, iG) \\ &\simeq \map_{\Hor(\P)(i(z),i(x))}(iF, {}^{iW}{iG}{}).
\end{align*}
for $F \in \Hor(\Q)(z,x)$ and $G \in \Hor(\Q)(z,y)$.
From this, it follows that (2) and (3) are equivalent.
\end{proof}

\begin{remark}\label{rem.hor-refl-strong}
	Note that if $i : \Q \rightarrow \P$ is a reflective inclusion of double $\infty$-categories then  
	condition (1) of \ref{prop.hor-refl-strong} is automatically met.
\end{remark}

\begin{corollary}\label{cor.refl-hor-closed}
	The following two classes of double $\infty$-categories are closed under reflective inclusions:
	\begin{itemize}
		\item the horizontally closed ones,
		\item the horizontally locally presentable ones.
	\end{itemize}
\end{corollary}

\begin{corollary}\label{cor.int-ccat-hor-closed}
	If $\E$ is an $\infty$-topos, then $\CCAT_\infty(\E)$ is horizontally locally presentable.
\end{corollary}
\begin{proof}
	To see this, we choose a geometric embedding $i : \E \rightarrow \PSh(T)$, where $T$ is some 
	small $\infty$-category. Then we obtain an induced reflective inclusion
	$i_* : \CCAT_\infty(\E) \rightarrow [T^\op, \CCAT_\infty]$, so that the result follows from \ref{cor.refl-hor-closed} and \cite[Theorem 6.9]{FunDblCats}.
\end{proof}

\subsection{Proper and smooth lax functors}\label{ssection.proper-smooth} We will now investigate two other classes 
of lax functors that contain the reflective inclusions. The definition 
makes use of the
notion of exact squares as developed in \cite[Definition 6.24]{EquipI}.

\begin{definition}\label{def.proper-pres-cones}
	Let $v : \Q \rightarrow \P$ be a lax functor between $\infty$-equipments. Then $v$ is called \textit{proper} 
	if it is normal and admits a left adjoint $u : \P \rightarrow \Q$ so that for any arrow 
	$f : x\rightarrow y$ in $\Q$, the commutative naturality square 
	\[
		\begin{tikzcd}
			uv(x) \arrow[d,"uv(f)"']\arrow[r] & |[alias=f]|x\arrow[d,"f"] \\
			|[alias=t]| uv(y) \arrow[r] & y
			%\arrow[from=f,to=t, Rightarrow, "\sim"{sloped,swap,pos=0.35}]
		\end{tikzcd}
	\]
	provided by the counit of this adjunction, is exact in $\Q$.
	Dually, $v$ is called \textit{smooth} if it is normal and admits a left adjoint $u : \P \rightarrow \Q$ so that for any arrow 
	the commutative square 
	\[
		\begin{tikzcd}
			uv(x) \arrow[r,"uv(f)"]\arrow[d] & |[alias=f]|uv(y)\arrow[d] \\
			|[alias=t]|x \arrow[r,"f"] & y,
			%\arrow[from=f,to=t, Rightarrow, "\sim"{sloped,swap,pos=0.35}]
		\end{tikzcd}	
	\]
	is exact for every arrow $f : x\rightarrow y$.
\end{definition}

\begin{remark}
	This terminology is chosen so that the lax functors between 
	the slice equipments of $\CCAT_\infty$ given by base change along a functor  will be proper (resp.\ smooth) if and only if 
	the functor is proper (resp.\ smooth) in the sense of Cisinski \cite{Cisinski}. We will explain this in the next subsection.
\end{remark}

\begin{example}\label{ex.flat-geom-emb}
	Every reflective inclusion of double $\infty$-categories is proper and smooth as the counits will be equivalences in this case.
\end{example}

The main feature of proper lax functors is that they preserve the \textit{proarrows of cones under diagrams} (see \cite[Definition 6.9]{EquipI}) in the associated formal category theories. 

\begin{proposition}\label{prop.proper-pres-cones}
	Let $v : \Q \rightarrow \P$ be a proper lax functor between $\infty$-equipments. Suppose that $f : i \rightarrow x$ and $w : i \rightarrow j$ are arrows 
	of $\Q$ so that there exists a proarrow $C$ of $w^\circledast$-cones under $f$.
	Then $vC$ is a proarrow of $v(w)^\circledast$-cones under $v(f)$.
\end{proposition}
\begin{proof}
This readily follows from \ref{prop.lax-adj-hor} and \ref{cor.normal-lax-fun-eq-pres-compconj}. 
Let $u : \P \rightarrow \Q$ be the
left adjoint to $v$ and write $\epsilon$ for the counit of the adjunction. 
Now for any proarrow 
$F : v(x) \rightarrow v(j)$ in $\Q$, we obtain a of chain natural equivalences
\begin{align*}
\map_{\Hor(\P)(v(x), v(j))}(F, vC)&\simeq \map_{\Hor(\Q)(x,j)}({\epsilon_{j,\circledast}}{u(F)}\epsilon_x^\circledast, C) \\
 &\simeq \map_{\Hor(\Q)(x,i)}(w^\circledast\epsilon_{j,\circledast}u(F)\epsilon_x^\circledast, f^\circledast) \\
 &\simeq \map_{\Hor(\Q)(x,i)}(\epsilon_{i,\circledast}u(v(w))^\circledast u(F)\epsilon_x^\circledast, f^\circledast) \\
 &\simeq \map_{\Hor(\Q)(x,i)}(\epsilon_{i,\circledast}u(v(w)^\circledast F)\epsilon_x^\circledast, f^\circledast) \\
 &\simeq \map_{\Hor(\P)(v(x),v(i))}({v(w)}^\circledast F, v(f^\circledast)) \\ 
 &\simeq \map_{\Hor(\P)(v(x),v(i))}(v(w)^\circledast\circ F, v(f)^\circledast).
\end{align*}
This entails that $vC$ is a proarrow of $v(w)^\circledast$-cones under $v(f)$. 
\end{proof}

The above proposition entails that proper lax functors preserve \textit{left Kan extensions} (see \cite[Definition 6.19]{EquipI}) 
and their \textit{units} (see \cite[Construction 6.20]{EquipI}).

\begin{corollary}
	Let $v : \Q \rightarrow \P$ be a 
	proper lax functor between $\infty$-equipments. Suppose that the 2-cell
	\[
		\alpha = \begin{tikzcd}
			|[alias=f]|i \arrow[r,"f"] \arrow[d, "w"'] & x \\
			j \arrow[ur, "g"'name=t]
			\arrow[from=f,to=t, Rightarrow, shorten >= 6pt]
		\end{tikzcd}
	\]
	is the unit in $\Vert(\Q)$ of a left Kan extension of $f$ along $w$. Then the image
	\[
		v(\alpha) = \begin{tikzcd}
			|[alias=f]|v(i) \arrow[r,"v(f)"] \arrow[d, "v(w)"'] & v(x) \\
			v(j) \arrow[ur, "v(g)"'name=t]
			\arrow[from=f,to=t, Rightarrow, shorten >= 6pt]
		\end{tikzcd}
	\]
	is the unit of a left Kan extension of $v(f)$ along $v(w)$ in $\Vert(\P)$.
\end{corollary}
\begin{proof}
	By definition, the conjoint of $g$ is the proarrow of cones under $f$ weighted by the conjoint $w^\circledast$. Consequently, 
	\ref{cor.normal-lax-fun-eq-pres-compconj} and \ref{prop.proper-pres-cones} imply that the conjoint of $v(g)$ is the proarrow 
	of cones under $v(f)$ weighted by the conjoint $v(w)^\circledast$. Thus $v(g)$ is the left Kan extension of $v(f)$ along $v(w)$. 
	The claim that the associated unit is precisely given by $v(\alpha)$ may be verified by inspection using the equivalences in the proof of \ref{prop.proper-pres-cones}.
\end{proof}

\begin{remark}
	There are dual statements for smooth lax functors: these preserve proarrows of cones \textit{over} diagrams and \textit{right} Kan extensions.
\end{remark}

\subsection{Example: change of base}\label{ssection.base-change} The main instances 
of proper and smooth lax functors are given by change of base functors. We will discuss this here and give some examples.
We first define slice double categories:

\begin{definition}\label{def.slices}
	Let $I$ be an $(\infty,2)$-category. Suppose that $f : I_v \rightarrow \P$ is a functor between double $\infty$-categories. Then we define the (vertical) 
	\textit{slice double $\infty$-category} $\P/f$  by the pullback square
	\[
		\begin{tikzcd}
			\P/f\arrow[d] \arrow[r] & {[[1] \times I, \P]}\arrow[d,"{(\ev_1, \ev_0)}"] \\
			\P \arrow[r] & {[I, \P]^{\times 2}}
		\end{tikzcd}
	\]
	where the bottom functor is given by the functor that selects $f$ and the diagonal functor.
\end{definition}

\begin{remark}
	A special case of this construction was discussed in \cite[Section 3.3]{FunDblCats}. Since $\infty$-equipments 
	are closed under vertical cotensor products and limits, it follows that slices of $\infty$-equipments are again 
	$\infty$-equipments.
\end{remark}

\begin{remark}\label{rem.level-wise-slice}
	In the context of \ref{def.slices}, if $I$ is an $\infty$-category then the slice double $\infty$-category $\P/f$ is computed 
	as a level-wise slice. In that case, we have $(\P/f)_n = \P_n/s^*f$ where $s$ is the degeneracy map $[n] \rightarrow [0]$.
\end{remark}

\begin{example}
	If $E$ is an object of an $\infty$-topos $\E$, then we can view $E$ as a category internal to $\E$: as the simplicial object in $\E$ that is constantly equal to $E$. See also \ref{ssection.int-cat-gen-obj}. 
	One may verify that there is a canonical equivalence 
	$$
	\CCAT_\infty(\E)/E \simeq \CCAT_\infty(\E/E).
	$$
\end{example}

\begin{definition}
	A double $\infty$-category $\P$ is said to \textit{admit (resp.\ normal, strict) pullbacks} if $\P$ admits (resp.\ normal, strict) limits 
	shaped by the free span category $\Sigma = \{s \leftarrow \top \rightarrow t\}$.
\end{definition}

\begin{proposition}
	Suppose that $\P$ is a double $\infty$-category with (normal, strict) pullbacks. Then for every arrow $f : x \rightarrow y$ in $\P$, there is a (normal, strict) lax adjunction 
	$$
	f_! : \P/x \rightleftarrows \P/y : f^*,
	$$
	where the left adjoint $f_!$ is the pushforward functor $\P/x \simeq \P/f \rightarrow \P/y$.
\end{proposition}
\begin{proof}
	Note that the equivalence $\P/x \simeq \P/f$ follows from \ref{rem.level-wise-slice}. In particular, $f_!$ may be identified with $(s^*f)_! : \P_n/s^*x \rightarrow \P_n/s^*y$ level-wise.
	The result now readily follows from characterization (3) in \ref{prop.lax-double-adjunctions-char}.
\end{proof}

\begin{example}
	We have the following examples of double $\infty$-categories that admit strict pullbacks:
	\begin{enumerate}
		\item the $\infty$-equipment $\CCAT_\infty(\E)$ of categories internal to an $\infty$-topos $\E$ (see \ref{ex.int-ccat-dbl-limits}),
		\item the double $\infty$-category $\SSPAN(\C)$ of spans in an $\infty$-category $\C$ with finite limits.
	\end{enumerate}
\end{example}

\begin{proposition}\label{prop.proper-smooth-arrows}
	Suppose that $f$ is an arrow in an $\infty$-equipment $\P$ that admits normal pullbacks. Then the following assertions are equivalent:
	\begin{enumerate}
		\item the pullback functor $f^* : \P/y \rightarrow \P/x$ is proper,
		\item for any commutative diagram 
		\[
		\begin{tikzcd}
			i' \arrow[r]\arrow[d] & j' \arrow[r]\arrow[d] & x\arrow[d] \\
			i \arrow[r] & j \arrow[r] & y
		\end{tikzcd}
		\]
		of pullback squares in $\Vert(\P)$, the left square is exact.
	\end{enumerate}
	Similarly, the following assertions are equivalent: 
	\begin{enumerate}[label=(\arabic*')]
		\item the pullback functor $f^* : \P/y \rightarrow \P/x$ is smooth,
		\item for any commutative diagram 
		\[
		\begin{tikzcd}
			i' \arrow[r]\arrow[d] & i\arrow[d] \\
			j' \arrow[r]\arrow[d] & j\arrow[d] \\
			x \arrow[r] & y
		\end{tikzcd}
		\]
		of pullback squares in $\Vert(\P)$, the top square is exact.
	\end{enumerate}
\end{proposition}
\begin{proof}
	Our assumptions assure that $f^*$ is always normal. Spelling out the definition, $f^*$ is proper if and only if for every arrow $j' \rightarrow j$ in $\P$, the square 
	\[
		\begin{tikzcd}
			f_!f^*j' \arrow[r]\arrow[d] & j'\arrow[d] \\
			f_!f^*j \arrow[r] & j
		\end{tikzcd}
	\]
	is exact. This is precisely assumption (2). The demonstration for the smooth case is analogous.
\end{proof}

\begin{definition}\label{def.proper-smooth-arrows}
	Let $f$ be an arrow in an $\infty$-equipment $\P$ that admits pullbacks. Then $f$ is called \textit{proper} if
	the equivalent conditions (1) and (2) of \ref{prop.proper-smooth-arrows} are met. Dually, $f$ is called \textit{smooth}  
	if conditions (1') and (2') of this proposition are satisfied.
\end{definition}

\begin{example}
	Suppose that $\P = \CCAT_\infty$. Then a functor between $\infty$-categories is proper (resp.\ smooth) if and only if 
	it is proper (resp.\ smooth) in the sense of Cisinski \cite[Definition 4.4.41]{Cisinski}. This was observed in \cite[Example 6.28]{EquipI}. 
\end{example}

\begin{example}\label{ex.proper-point-wise}
	Let $\P$ be an $\infty$-equipment that admits normal pullbacks. If $X$ is an $\infty$-category, then one can readily 
	check that the cotensor $\infty$-equipment $[X, \P]$ admits normal pullbacks as well using \ref{prop.lax-double-adjunctions-char} and \ref{rem.cotensor-infty-cat}. An arrow $\alpha : h \rightarrow k$ in $[X,\P]$ is proper (resp.\ smooth) if and only if 
	$\alpha_x : h(x) \rightarrow k(x)$ is proper (resp.\ smooth) for each $x \in X$. This can 
	be deduced from \ref{prop.cart-cotensors} as follows. Since the evaluations 
	$$
	[X,\P] \rightarrow \P, \quad x \in X,
	$$
	jointly reflect and preserve (co)cartesian cells, an argument similar to in \ref{prop.refl-incl-formal-cat-theory}
	shows that the evaluations jointly reflect and preserve exact squares. Then one uses that pullbacks are computed level-wise as well.
\end{example}

\section{Monoidal \texorpdfstring{$\infty$}{∞}-equipments}\label{section.mon-equipments}

Suppose that we have an arrow $w : i \rightarrow j$ and an object $x$ in an 
$\infty$-equipment $\P$. In \cite[Proposition 6.22]{EquipI}, 
we explained that a left adjoint to the restriction functor $w^* : \Vert(\P)(j,x) \rightarrow \Vert(\P)(i,x)$ 
exists when $x$ admits all left Kan extensions along $w$. In this case, 
the values of the left adjoint are computed by these Kan extensions.
The goal of this section
is to obtain a similar \textit{internal} result  where the vertical mapping $\infty$-categories of $\P$ are now replaced 
by suitable internal homs in $\P$. 

\subsection{Monoidal double \texorpdfstring{$\infty$}{∞}-categories}

To access these internal  
hom objects, we will need to introduce the notion 
of a \textit{closed monoidal structure} on a double $\infty$-category. In the strict context, this has been studied by (among others) Grandis-Par\'e \cite{GrandisPare} and Hansen-Shulman \cite{HansenShulman}.

\begin{definition}
	A \textit{monoidal double $\infty$-category} is a monoid in the $\infty$-category $\DblCat_\infty$. 
	Concretely, it is 
	a simplicial double $\infty$-category $$\P^{\otimes} : \Delta^\op \rightarrow \DblCat_\infty$$ so that 
	so that the structure maps of $\P^{\otimes}$  give rise to an equivalence
	$$
	\P^{\otimes}([n]) \xrightarrow{\simeq} \P^{\otimes}(\{0\leq 1\}) \times \dotsb \times \P^{\otimes}(\{n-1\leq n\})
	$$
	for all $n \geq 0$.
	The object $\P^{\otimes}([1])$ is called the \textit{underlying double $\infty$-category} of $\P^{\otimes}$.
\end{definition}

\begin{remark}
	Recall that the functors $\Vert(-)$ and $\Hor(-)$ preserve products, and hence lift to functors 
	between respective monoid $\infty$-categories. Consequently, if a double $\infty$-category $\P$ has a monoidal structure, 
	we obtain an induced monoidal structure on its vertical and horizontal fragments as well.
\end{remark}

\begin{example}
	If a double $\infty$-category $\P$ admits a monoidal structure, then there is an induced point-wise monoidal 
	structure on $\FFUN(X, \P)$ for every double $\infty$-category $X$. This follows again from the fact $\FFUN(X,-)$ preserves products.
\end{example}

\begin{notation}
	If $\C$ is an $\infty$-category with finite products, we will write 
	$$
	\mathrm{Mon}(\C) \subset \fun(\Delta^\op, \C) 
	$$
	for the $\infty$-category of monoids in $\C$.
\end{notation}

The following observation is useful:

\begin{lemma}\label{lem.mon-dbl-cat}
	There is a canonical fully faithful functor 
	$$
	\mathrm{Mon}(\DblCat_\infty) \rightarrow \fun(\Delta^\op, \mathrm{Mon}(\Cat_\infty))
	$$
	whose essential image is given by the functors $R : \Delta^\op \rightarrow \mathrm{Mon}(\Cat_\infty)$  
	with the property that the composite $$\Delta^\op \xrightarrow{R} \mathrm{Mon}(\Cat_\infty) \xrightarrow{\ev_{[1]}} \Cat_\infty$$ 
	is a double $\infty$-category.
\end{lemma}
\begin{proof}
	Note that we have a canonical equivalence 
	$$
	\mathrm{Mon}(\fun(\Delta^\op, \Cat_\infty)) \xrightarrow{\simeq} \fun(\Delta^\op, \mathrm{Mon}(\Cat_\infty)),
	$$
	that is restricted from the functor $\fun(\Delta^{\op, \times 2}, \Cat_\infty) \rightarrow \fun(\Delta^{\op,\times 2}, \Cat_\infty)$ 
	induced by swapping the coordinates of $\Delta^{\op, \times 2}$. 
	Now, we note that the inclusion $\DblCat_\infty \rightarrow \fun(\Delta^\op, \Cat_\infty)$ induces a fully faithful 
	functor $$\mathrm{Mon}(\DblCat_\infty) \rightarrow \mathrm{Mon}(\fun(\Delta^\op, \Cat_\infty))$$ whose essential 
	image consists of those monoids $M \in \mathrm{Mon}(\fun(\Delta^\op, \Cat_\infty))$ so that the evaluations $M([n])$ are double $\infty$-categories. 
	But this can be checked at $n=1$ since double $\infty$-categories are closed under products. 
	Thus the above constructed equivalence $
	\mathrm{Mon}(\fun(\Delta^\op, \Cat_\infty)) \simeq \fun(\Delta^\op, \mathrm{Mon}(\Cat_\infty))
	$
	restricts as described in the proposition.
\end{proof}

We will use this observation in the following result, which 
offers a convenient way of obtaining monoidal structures by assembling level-wise cartesian structures:

\begin{proposition}\label{prop.mon-cartesian}
	Suppose that $\P$ is a double $\infty$-category with all strict limits shaped by finite sets.
	Then there exists a monoidal structure $\P^\times$ on $\P$ so that the restriction 
	$$
	\Delta^\op \xrightarrow{\P^\times} \DblCat_\infty \subset \fun(\Delta^\op, \Cat_\infty) \xrightarrow{\ev_{[n]}} \Cat_\infty
	$$
	encodes the cartesian monoidal structure on $\P_n$.
\end{proposition}
\begin{proof}
	If we view $\P$ as a functor $\P : \Delta^\op \rightarrow \Cat_\infty$ then the assumptions and \ref{prop.dbl-limits-v-shaped} 
	imply that $\P$ factors through the $\infty$-category of $\infty$-categories with finite products and finite product preserving functors between them.
	Hence, the desired result follows from \cite[Corollary 2.4.1.9]{HA} and \ref{lem.mon-dbl-cat}.
\end{proof}

\begin{definition}
	In the context of \ref{prop.mon-cartesian}, the double $\infty$-category $\P$ is said to have all strict finite products and  
	$\P^\times$ is called the \textit{cartesian monoidal structure} on $\P$.
\end{definition}

\begin{example}\label{ex.ccat-monoidal}
	Let $\C$ be an $\infty$-category. Note that $\Cat_\infty/\C$ admits all finite products. If $E$ and $E'$ 
	are fibered over $\C$, then their product is given by the pullback square
	\[
		\begin{tikzcd}
			E \times_{\C} E' \arrow[r]\arrow[d] & E' \arrow[d] \\
			E \arrow[r] & \C.
		\end{tikzcd}
	\]
	Since Conduch\'e fibrations are closed under base change and the composition of two Conduch\'e fibrations is again a Conduch\'e fibration, 
	we therefore deduce that the full subcategory $\Con(\C) \subset \Cat_\infty/\C$ of Conduch\'e fibrations is closed under finite products. Moreover, 
	if $f : \C \rightarrow \D$ is functor, then the pullback functor 
	$\Con(\D) \rightarrow \Con(\C)$ preserves these finite products. Consequently, 
	\ref{prop.dbl-limits-v-shaped} implies that
	 $\CCAT_\infty$ has all strict products. Hence $\CCAT_\infty$ can be endowed with the cartesian monoidal structure.

	Note that, from the perspective 
	of profunctors, 
	the product of 
	two proarrows $F : \mathscr{A} \rightarrow \mathscr{B}$ and $G : \C \rightarrow \D$ in $\CCAT_\infty$ 
	is given by the profunctor 
	$$
(\mathscr{B} \times \D)^\op \times (\mathscr{A} \times \C) \rightarrow \S : (b,d,a,c) \mapsto F(b,a) \times G(d,c).
$$ 
\end{example}

\begin{definition}\label{def.mon-closed}
	A monoidal structure $\P^\otimes$ on a double $\infty$-category $\P$ is called  \textit{closed} if for all objects $x$ of $\P$, the functor 
	$
	x \otimes - : \P \rightarrow \P
	$
	admits a normal lax right adjoint 
	$$
	[x, -] : \P \rightarrow \P.
	$$
\end{definition} 

\begin{example}\label{ex.ccat-closed-monoidal}
	The cartesian monoidal structure of \ref{ex.ccat-monoidal} is closed. To see this, we 
	note that $\Cat_\infty/\C$ is both tensored and cotensored over $\Cat_\infty$. 
	For $X \in \Cat_\infty$, we we have an adjunction 
$$
 X \times - : \Cat_\infty/\C \rightleftarrows \Cat_\infty/\C : [X,-].
$$
The left adjoint is given by taking fibered products with the projection $X \times \C \rightarrow \C$.
The right adjoint carries an $\infty$-category $E$ fibered over $\C$ to the pullback 
\[
	\begin{tikzcd}
		{[X,E]} \arrow[r]\arrow[d] & \fun(X, E) \arrow[d] \\
		\C \arrow[r] & \fun(X,\C),
	\end{tikzcd}
\]
where the bottom map is the diagonal. This cotensor is compatible with pullbacks: if $f : \C \rightarrow \D$ 
is a functor, then the canonical map
$$
f^*[X, E] \rightarrow [X, f^*E]
$$
is an equivalence for every $\infty$-category $E$ fibered over $\D$. Recall 
that $(\CCAT_\infty)_{n} = \Cat_\infty/[n]$ for $n = 0, 1$. 
Hence, the above observations and \ref{prop.lax-double-adjunctions-char} imply that 
the cartesian structure on $\CCAT_\infty$ is closed.

From the perspective of profunctors, the cotensor 
is computed as follows. If $X \in \Cat_\infty$, then the cotensor 
of a proarrow $F : \C \rightarrow \D$ with $X$ carries two functors
$f : X \rightarrow \mathscr{C}$ and $g : X \rightarrow \mathscr{D}$ to the end 
$$
[X,F](g,f) = \int_{y,x \in X} F(g(y), f(x)).
$$
\end{example}

\begin{example}\label{ex.int-ccat-cart-closed}
	Taking the viewpoint of Conduch\'e modules (see \ref{ssection.con-modules}), we may generalize \ref{ex.ccat-monoidal} and \ref{ex.ccat-closed-monoidal} 
	to conclude that  that $\CCAT_\infty(\E)$ is a cartesian closed double $\infty$-category for every $\infty$-topos $\E$.
\end{example}

\subsection{Kan extension arrows in monoidal equipments}\label{ssection.pke-arrows}

Let us fix a closed monoidal $\infty$-equipment $\P$. If $w : i \rightarrow j$ is an arrow in $\P$, 
then we obtain a restriction arrow 
$$
w^* : [j,x] \rightarrow [i,x]
$$
for every object $x$ of $\P$. It has the defining property that it is adjunct to the composite arrow
$$i \otimes [j,x] \xrightarrow{w \otimes [j,x]} j \otimes [j,x] \xrightarrow{} x,$$
where the right arrow is the evaluation counit.
We will study left adjoints to $w^*$ in terms of the left Kan extensions inside the formal category 
theory of $\P$ \cite[Section 6.3]{EquipI}.

In the case that $w^*$ is a right adjoint, then we may always compute left Kan extensions along $w$ (if they exist) as 
values of the left adjoint:

\begin{proposition}\label{prop.internal-restr-vertical}
	Suppose that $w^*$ admits a left adjoint
	$$
	w_! : [i,x] \rightarrow [j,x] 
	$$
	in $\Vert(\P)$. If $f : i \otimes y \rightarrow x$ is an arrow that 
	admits a left Kan extension $g : j \otimes y \rightarrow x$ along $w \otimes y$, 
	then $g$ is adjunct to the composite arrow 
	$$
	y \xrightarrow{f^\sharp} [i,x] \xrightarrow{w_!} [j,x].
	$$
	Here $f^\sharp$ denotes the adjunct arrow of $f$.
	Moreover, the unit $f \rightarrow g(w \otimes y)$ associated with this left Kan extension is adjunct to the whiskering 
	$f^\sharp \rightarrow w^*w_!f^\sharp$ of the unit 
	of $(w_! , w^*)$ with $f^\sharp$.
\end{proposition}
\begin{proof}
	One may readily verify that there is a commutative diagram 
	\[
			\begin{tikzcd}
				\Vert(\P)(y, [j,x]) \arrow[r, "(w^*)_*"] \arrow[d, "\simeq"'] & \Vert(\P)(y, [i,x]) \arrow[d, "\simeq"] \\
				\Vert(\P)(j \otimes y, x) \arrow[r, "(w\otimes y)^*"] & \Vert(\P)(i \otimes y,x),
			\end{tikzcd}
	\]
	of vertical mapping $\infty$-categories. Thus the desired result follows from \cite[Proposition 6.22]{EquipI}.
\end{proof}

We would now like to obtain a converse to the above observation and 
give criteria for the existence of a left adjoint $w_!$ in terms of the existence of left Kan extensions. 
To that end, we may use \cite[Corollary 6.6]{EquipI}. This asserts that 
a left adjoint  to $w^*$ exists if and only if the companion 
$$
w^{\ast}_\circledast : [j,x] \rightarrow [i,x]
$$
is a conjoint as well. The main take-away of this subsection 
is the following computation of the latter proarrow:

\begin{proposition}\label{prop.computation-proarrow-lke}
	Let $f : i \otimes y \rightarrow x$ be an arrow of $\P$ 
	so that the proarrow 
	$$
	C : x \rightarrow j \otimes y
	$$
	of $(w \otimes y)^\circledast$-cones under $f$ exists. Then there is a cartesian 2-cell 
	\[
		\begin{tikzcd}
			{[j,x]} \arrow[r, "f^{\sharp,\circledast}w^*_\circledast"]\arrow[d,equal] & y \arrow[d]\\
			{[j,x]} \arrow[r, "{[j,C]}"] & {[j, j \otimes y]}
		\end{tikzcd}
	\] 
	in $\P$ where the right vertical arrow is the unit. In particular, if $f$ admits a left Kan extension $g : j \otimes y \rightarrow x$ along 
	$w \otimes y$ then $f^{\sharp,\circledast}w^*_\circledast$ is the conjoint of the arrow $g^\sharp: y\rightarrow [j,x]$ adjunct to $g$.
\end{proposition}
\begin{proof}
	For brevity, we will denote the proarrow $f^{\sharp,\circledast}w^*_\circledast$ by $G$. Let us write 
	$\eta$ and $\epsilon$ for the unit and counit of the adjunction $(j\otimes -, [j,-])$ respectively.
	Applying \ref{prop.lax-adj-hor} to the adjunction $(i\otimes - , [i, -])$, we obtain equivalences
	\begin{align*}
	\map_{\Hor(\P)([j,x], y)}(F, G) &\simeq \map_{\Hor(\P)(x, x)}(f_\circledast(i\otimes F){(w^*)^{\sharp,\circledast}}, \id_x) \\ 
	&\simeq \map_{\Hor(\P)(x, i\otimes y)}({(i\otimes F)}{(w^*)}^{\sharp,\circledast}, f^\circledast),
	\end{align*}
	natural in every proarrow $F :[j,x] \rightarrow y$. In the last step, we used the observation of \ref{rem.lax-adj-hor}.
	Since the adjunct map $(w^*)^\sharp : i \otimes [j,x] \rightarrow x$ is given by the composite
	$$
	i \otimes [j,x] \xrightarrow{w \otimes [j,x]} j \otimes [j,x] \xrightarrow{\epsilon_x} x,
	$$
	we deduce that
	$$
	{(i\otimes F)}{(w^*)^{\sharp, \circledast}} \simeq {{(i\otimes F)}(w \otimes [j,x])}^\circledast\epsilon_x^\circledast.
	$$
	Functors between double $\infty$-categories preserve companions and conjoints; in particular the functor 
	$(-) \otimes (-) : \P^{\times 2} \rightarrow \P$. It thus follows that
	$$
	{(i\otimes F)}{(w \otimes [j,x])^\circledast} \simeq {(i\otimes F)}{(w^\circledast \otimes [j,x])} \simeq  w^\circledast \otimes F  \simeq (w \otimes y)^{\circledast}{(j \otimes F)}.
	$$
	It follows that
	\begin{align*}
		\map_{\Hor(\P)([j,x], y)}(F, G) &\simeq \map_{\Hor(\P)(x, i\otimes y)}({{(w\otimes y)^\circledast}{(j \otimes F)}}{\epsilon_x^\circledast}, f^\circledast).
	\end{align*}
	By definition, $C$ is the proarrow of ${(w\otimes y)}^\circledast$-cones under $f$, hence we obtain 
	an equivalence
	\begin{align*}
		\map_{\Hor(\P)([j,x], y)}(F, G) &\simeq \map_{\Hor(\P)(x, j\otimes y)}({(j \otimes F)}{\epsilon_x^\circledast}, C) \\
		 &\simeq \map_{\Hor(\P)([j,x], y)}(F, \eta_y^\circledast{[j,C]}{}).
	\end{align*}
	natural in $F$. In the last step, we made use of \ref{prop.lax-adj-hor} once again. This proves the proposition.
\end{proof}

\begin{remark}
	In \ref{section.int-cat-theory}, we will see an application of \ref{prop.computation-proarrow-lke} to the $\infty$-equipment $\CCAT_\infty(\E)$, 
	where $\E$ is an arbitrary $\infty$-topos. This then yields the existence results of functors of Kan extensions internal to $\E$.
\end{remark}

We will now discuss an example where one can deduce a familiar existence result for the left adjoint to $w^*$ in terms of Kan extensions. To this end, 
we introduce an auxiliary assumption: 
\begin{itemize}
	\item[($\ast$)] A proarrow $F : x \rightarrow y$ in $\P$ is a conjoint if and only if $p^\circledast F$ is a conjoint for every $p : 1 \rightarrow y$.
\end{itemize}
Here $1$ denotes the unit of the monoidal structure on $\P$. 

\begin{example}
	The above assumption holds for \textit{strongly pointed} (see \cite[Definition 6.37]{EquipI}) cartesian equipments such as $\CCAT_\infty$.
\end{example}

\begin{corollary}
	Suppose that assumption ($\ast$) holds.
	If every arrow $f : i \rightarrow x$ in $\P$ admits 
	a left Kan extension along $w$, then $w^*$ admits a left adjoint $w_!$.
\end{corollary}
\begin{proof}
	By assumption, $w^*$ admits a left adjoint if and only if 
	for every arrow $f^\sharp : 1 \rightarrow [i,x]$ the proarrow $f^{\sharp,\circledast}w^*_\circledast : [j,x] \rightarrow 1$ is a conjoint. 
	Let $g : j \rightarrow x$ be the left Kan extension along $w$ of the adjunction arrow $f : i \rightarrow x$. Then 
	\ref{prop.computation-proarrow-lke} asserts that 
	$$
	f^{\sharp, \circledast}w^*_\circledast \simeq \eta_1^\circledast[j,g^\circledast] \simeq([j,g]\eta_1)^\circledast
	$$
	where $\eta_1 : 1 \rightarrow [j,j]$ is the arrow adjunct to the identity $j \rightarrow j$.
\end{proof}

\begin{corollary}
	Suppose that assumption ($\ast$) holds and that the unit $1$ is the terminal object of $\Vert(\P)$.
	Let $i$ and $x$ be objects of $\P$. Then the following assertions are equivalent:
	\begin{enumerate}
		\item $x$ admits all conical colimits shaped by $i$ (see \cite[Definition 6.30]{EquipI}),
		\item the diagonal arrow
		$$
		\Delta : x \rightarrow [i,x]
		$$
		admits a left adjoint.
	\end{enumerate}
	If these two equivalent conditions are met, then the left adjoint to the arrow (2) carries an arrow $f : i \rightarrow x$ to its conical colimit $\colim f$.
\end{corollary}
\begin{proof}
	The diagonal arrow $\Delta$ is given by restriction along the terminal arrow $i \rightarrow 1$. For every 
	 element $f^\sharp : 1 \rightarrow [i,x]$, $f^{\sharp,\circledast}{\Delta}_\circledast$ is the proarrow of coconical cones under the adjunct arrow $f : i \rightarrow x$ on account of \ref{prop.computation-proarrow-lke}.
	 The desired result follows from this computation.
\end{proof}

\section{Tabulations and fibrations in double \texorpdfstring{$\infty$}{∞}-categories}\label{section.fibrations}

The horizontal arrows of $\CCAT_\infty$, i.e.\ the correspondences between $\infty$-categories, may equivalently be viewed as
\textit{two-sided discrete fibrations} between $\infty$-categories  (we will review this notion shortly). 
This fibrational perspective follows from the work of Ayala and Francis \cite[Section 4]{AyalaFrancis} 
where they exhibit an explicit equivalence between the $\infty$-category $(\CCAT_\infty)_1 = \Cat_\infty/[1]$ of horizontal arrows 
and the $\infty$-category of two-sided discrete fibrations. One may find a further discussion of this equivalence in \cite[Subsection 4.1]{EquipI}.

The combined goal of the current section and \ref{section.fib-equipments} is to demonstrate that the situation 
for $\CCAT_\infty$ is an instance of a more general and purely double categorical phenomenon. 
Namely, we will formulate a notion of two-sided discrete fibrations 
internal to any suitable double $\infty$-category in this section.
Then we continue in \ref{section.fib-equipments} by 
characterizing those double $\infty$-categories for which the horizontal arrows may equivalently be viewed 
as these internal two-sided discrete fibrations.
For these double $\infty$-categories one may build a nice theory of two-sided discrete fibrations, including let and right fibrations as special cases.

As preparation for what follows, let us review the already established notion of two-sided fibrations between $\infty$-categories, together with an intermediate notion:

\begin{definition}\label{def.regular-span}
	A span  $(p,q) : E \rightarrow \C \times \D$ of $\infty$-categories is called  
	\textit{regular} if the following two conditions are met:
	\begin{enumerate}
		\item $p$ is a cocartesian fibration so that its 
		cocartesian arrows lie over equivalences in $\D$,
		\item $q$ is a cartesian fibration so that its cartesian arrows lie over equivalences in $\C$.
	\end{enumerate}
\end{definition}

\begin{remark}
	This terminology was introduced by Yoneda \cite{Yoneda} in the strict categorical setting. 
	The regular spans were called \textit{curved orthofibrations} in \cite{HHLN}.
\end{remark}

\begin{definition}\label{def.tsdfib}
	A  regular span $(p,q) : E \rightarrow \C \times \D$ between $\infty$-categories is called  
	a \textit{two-sided discrete fibration} if one of the equivalent conditions of \cite[Proposition 2.3.13]{HHLN} is met: 
	\begin{enumerate}
		\item the functor $(p,q)$ is conservative,
		\item for each $c \in \C$ and $d\in \D$, the fiber $E_{c,d}$ is a space,
		\item for each $d \in \D$, the functor $ E \times_\D \{d\} \rightarrow \C$ is a left fibration,
		\item for each $c \in \C$, the functor $\{c\} \times_\C E \rightarrow \D$ is a right fibration.
	\end{enumerate}
\end{definition}

\subsection{The horizontal arrow double \texorpdfstring{$\infty$}{∞}-category} Let $\P$ be a locally complete double $\infty$-category. We will commence by investigating the double $\infty$-category 
$$
\{[1], \P\} = \FFUN([1]_h, \P)
$$
of horizontal arrows in $\P$. This object plays an important role in the formalism 
of internal fibrations that we will introduce later in this section.  We note the 
horizontal arrow double $\infty$-category is again locally complete by \cite[Proposition 3.16]{FunDblCats}.

\begin{example}
	Let $\E$ be an $\infty$-topos.
	On account of \ref{prop.con-straightening}, the double $\infty$-category of horizontal arrows in $\CCAT_\infty(\E)$ can be identified with 
	the double $\infty$-category $$\CCON([1]^\op; \E)$$ 
	of Conduch\'e $[1]^\op$-modules internal to $\E$.
\end{example}	

\begin{example}
	As a simplicial $\infty$-category, the double $\infty$-category $\{[1], \CCAT_\infty\}$ 
	may level-wise be described as the full subcategory 
	$$
	\{[1], \CCAT_\infty\}_n \subset \Cat_\infty/([n] \times [1]^\op)
	$$
	spanned by the Conduch\'e fibrations. This follows from \cite[Corollary 5.19]{FunDblCats}.
\end{example}

The double $\infty$-category of horizontal arrows comes with a projection 
$$
(\ev_0, \ev_1): \{[1], \P\} \rightarrow \P \times \P
$$
which enjoys the following fibrational property:

\begin{theorem}\label{thm.rep-tsdfib}
	The projection $(\ev_0, \ev_1)$ induces a two-sided discrete fibration
	$$
	\Vert(\{[1], \P\})(F,G) \rightarrow \Vert(\P)(x,x') \times \Vert(\P)(y,y')
	$$
	between vertical mapping $\infty$-categories for every two horizontal arrows $F : x \rightarrow y$ and $G : x' \rightarrow y'$ of $\P$.
\end{theorem}

\begin{remark}
	In the context of \ref{thm.rep-tsdfib}, we will see that the cocartesian pushforward of an object $\alpha$ of $\Vert(\{[1], \P\})(F,G)$, 
	described by a 2-cell 
	\[
		\alpha = \begin{tikzcd}
			x \arrow[d,"f"'] \arrow[r, "F"] & y \arrow[d,"g"] \\
			x \arrow[r, "G"] & y'
		\end{tikzcd}
	\]
	along a vertical 2-cell $\beta : f \rightarrow f'$ in $\Vert(\P)(x,x')$ corresponds to the composite
	\[
		\begin{tikzcd}
			x \arrow[r, equal, ""name=f]\arrow[d, "f'"'] &  x \arrow[d,"f"] \arrow[r, "F"] & y \arrow[d,"g"] \\
			x' \arrow[r,equal, ""'name=t] & x' \arrow[r, "G"] & y'. 
			\arrow[from=f, to=t, phantom, "\scriptstyle\beta"]
		\end{tikzcd}
	\]
	Similarly, the cartesian pullback of $\alpha$ along a vertical 2-cell $\gamma : g' \rightarrow g$ is given by the pasting 
	\[
		\begin{tikzcd}
		  x \arrow[d,"f"'] \arrow[r,"F"] & y \arrow[d,"g"'] \arrow[r,equal, ""name=f] & y \arrow[d,"g'"] \\
			 x' \arrow[r,"G"] & y' \arrow[r,equal, ""'name=t] & y'. 
			\arrow[from=f, to=t, phantom, "\scriptstyle\gamma"]
		\end{tikzcd}
	\]
\end{remark}

This result essentially follows from double $\infty$-categorical combinatorics. Its demonstration will be in two steps: first, we will show that the above span is conservative and then show that it is regular.
Let us first start by introducing the following characterization of \textit{locally conservative functors}:

\begin{lemma}\label{lem.loc-cons}
	Let $f : \X \rightarrow \Y$ be a functor between $(\infty,2)$-categories. Then the following assertions are equivalent: 
	\begin{enumerate}
		\item $f$ is locally conservative, i.e.\ the induced functor $\X(x,y) \rightarrow \Y(fx, fy)$ between mapping $\infty$-categories is conservative for all $x, y \in \X$,
		\item the canonical map 
		$$
		\map([1], \X) \rightarrow \map([1;1], \X) \times_{\map([1;1], \Y)} \map([1], \Y)
		$$
		is an equivalence.
	\end{enumerate}
\end{lemma}
\begin{proof}
	This follows because the map in (2) fits in a commutative triangle
	\[
		\begin{tikzcd}[column sep= small]
		\map([1], \X) \arrow[rr]\arrow[dr] &&  \map([1;1], \X) \times_{\map([1;1], \Y)} \map([1], \Y) \arrow[dl] \\ 
		& \map([0], \X)^{\times 2},
		\end{tikzcd}
	\]
	hence it will be an equivalence if and only if it is an equivalence on fibers above a pair of objects $(x,y)$ of $\X$. But on fibers, this is precisely 
	the map 
	$$
	\X(x,y)_0 \rightarrow \X(x,y)_1 \times_{\Y(fx, fy)_1} \Y(fx,fy)_0,
	$$
	which is an equivalence if and only if $\X(x,y) \rightarrow \Y(fx, fy)$ is locally conservative.
\end{proof}

\begin{corollary}
	The span of \ref{thm.rep-tsdfib} is conservative.
\end{corollary}
\begin{proof}
	This follows directly from the fact that the commutative square
	\[
		\begin{tikzcd}
			\{0,1\}_h \times [1;1]_v \arrow[r]\arrow[d] & \{0,1\}_h \times [1]_v\arrow[d] \\
			{[1]_h} \times [1;1]_v \arrow[r] & {[1]_h} \times [1]_v
		\end{tikzcd}
	\]
	is a pushout square of double $\infty$-categories, see the proof of \cite[Proposition 3.16]{FunDblCats}.
\end{proof}

Our next goal is to verify that the span of \ref{thm.rep-tsdfib} is regular. To this end, we make some preparatory definitions.

\begin{construction}
	Let $n \geq 0$, then we define the double $\infty$-category $P_n$ by the pushout square
	\[
		\begin{tikzcd}
			{\{0\}_h} \times [1;n]_v \arrow[r]\arrow[d, "{[1;d_1]}"']  & {[1]_h} \times [1;n]_v \arrow[d] \\
			{\{0\}_h} \times [1;n+1]_v\arrow[r] & P_n.
		\end{tikzcd}
	\]
	The canonical cofibered comparison map
	\[
	\begin{tikzcd}
		& {[1]_h \times [1;1]_v}\arrow[dl]\arrow[dr, "{[1]_h \times [1;\{0\leq 1\}]_v}"] \\
	{[1]_h} \times [1;1]_v \cup_{P_0} P_n \arrow[rr, "\phi_n"] && {[1]_h} \times [1;n+1]_v 
	\end{tikzcd}
	\]
	will play a role in what follows. 
\end{construction}

 With this, we may formulate the following characterization of cocartesian arrows:

\begin{lemma}\label{lem.local-cocart-arrows-ev0}
	Let $F : x \rightarrow y$ and $G : x' \rightarrow y'$ be horizontal arrows of $\P$. Suppose that 
	$\alpha : [1]_h \times [1;1]_v \rightarrow \P$ classifies an arrow of $\Vert(\{[1], \P\})(F,G)$.
	Then the following are equivalent:
	\begin{enumerate}
		\item the arrow $\alpha$ is cocartesian with respect to the projection 
	$$
	\Vert(\{[1], \P\})(F,G) \rightarrow \Vert(\P)(x,x'), 
	$$
		\item the map 
		$$
		\map({[1]_h} \times [1;n+1]_v, \P) \rightarrow \map({[1]_h} \times [1;1]_v \cup_{P_0} P_n , \P) 
		$$
		induced by $\phi_n$ is an equivalence on the fiber above $\alpha$ for all $n$,
		\item assertion (2) holds for $n=1$.
	\end{enumerate}
\end{lemma}

To prove the above lemma, we make use of the following immediate observation:

\begin{lemma}\label{lem.slices-mapping-cats}
	Let $x,y$ be objects in an $(\infty,2)$-category $\X$. Suppose that  $\alpha : [1;m] \rightarrow \X$ classifies an arrow 
	$[m] \rightarrow \X(x,y)$. Then there is a natural pullback square
	\[
	\begin{tikzcd}
	\map_{\Cat_\infty}([n], \alpha/\X(x,y)) \arrow[r]\arrow[d] & \map_{\Cat_{(\infty,2)}}([1; m + n + 1], \X) \arrow[d] \\
	\ast \arrow[r, "\alpha"] & \map_{\Cat_{(\infty,2)}}([1;m], \X).
	\end{tikzcd}
	\]
\end{lemma}

\begin{proof}[Proof of \ref{lem.local-cocart-arrows-ev0}]
	Throughout this proof, let us write $\alpha_0 := \alpha|[1]_h \times [1;\{0\}]_v$ and write $\beta, \beta_0$ for the images 
	of $\alpha$ and $\alpha_0$ under $\ev_0$ respectively. 
	Note that (2) is vacuous for $n=0$. Now, (1) holds precisely if the commutative square
	\[
		\begin{tikzcd}
			\alpha/\Vert(\{[1], \P\})(F,G) \arrow[r]\arrow[d] & \alpha_0/\Vert(\{[1], \P\})(F,G)\arrow[d]\\
			\beta/\Vert(\P)(x,x') \arrow[r] & \beta_0/\Vert(\P)(x,x')
		\end{tikzcd}
	\]
	is a pullback square. On account of \ref{lem.slices-mapping-cats}, this is true if and only if for every $n \geq 0$, the commutative square 
	\[
		\begin{tikzcd}
		\map_{\DblCat_\infty}([1]_h \times [1; n+2]_v, \P) \arrow[r]\arrow[d] & \map_{\DblCat_\infty}([1]_h \times [1; n+1]_v, \P) \arrow[d] \\
		\map_{\DblCat_\infty}(\{0\}_h \times [1; n+2]_v, \P) \arrow[r] & \map_{\DblCat_\infty}(\{0\}_h \times [1; n+1]_v, \P)
		\end{tikzcd}
	\]	
	is a pullback square after taking the fibers in each corner above the appropriate restrictions of $\alpha$. 
	In turn, this assertion can be precisely translated into the condition that $\phi_{n+1}$ induces an equivalence above the fiber of $\alpha$. This is precisely (2). 

	To show that (2) and (3) are equivalent, we note that there is a canonical equivalence of double $\infty$-categories 
	$$
	[1]_h \times [1;2]_v \cup_{[1]_h \times [1]_v} [1]_h \times [1;n-2]_v \xrightarrow{\simeq} [1]_h \times [1;n]_v
	$$
	for $n \geq 2$. One can use this to prove that the canonical map
	$
	P_2 \cup_{[1]_h \times [1]_v} [1;n-2]_v \rightarrow P_n
	$
	is an equivalence as well. This in turn implies that the comparison map $\phi_n$ is a pushout along $\phi_2$ for $n \geq 2$. 
\end{proof}

\begin{lemma}\label{lem.retraction-phin}
	The inclusion $i : P_0 \rightarrow [1]_h \times [1;1]_v$ admits a retraction $r : [1]_h \times [1;1]_v \rightarrow P_0$ 
	so that the cobase change of $\phi_n$ along $r$ induces an equivalence 
	$$
	r_*\phi_n: P_n \rightarrow P_0 \cup_{[1]_h \times [1;1]_v} [1]_h \times [1;n+1]_v
	$$
	of double $\infty$-categories. Moreover, the restriction $r|\{1\}_h \times [1;1]_v$ factors through the 
	degeneracy map $[1;1]_v \rightarrow [1]_v$. 
\end{lemma}
\begin{proof}
	Throughout this proof, we will denote the pushout-product functor for arrows in $\PSh(\Delta^{\times 2})$ by $$(-) \boxtimes (-) : \fun([1], \PSh(\Delta^{\times 2}))^{\times 2} \rightarrow \PSh(\Delta^{\times 2}).$$ We may express $[1;n]_h$ 
	as the pushout-product
	$[1;n]_h = (\{0,1\} \rightarrow [1])_h \boxtimes ([n] \rightarrow [0])_v$ so that we obtain the formula
	$$[1;n]_v = [1;n]_h^{\tp, \hop} =  ([n]^\op \rightarrow [0])_h  \boxtimes (\{0,1\} \rightarrow [1])_v.$$ 
	Consequently, we may construct a commutative diagram
	\[
		\begin{tikzcd}
		(A_n \rightarrow [1])_h \boxtimes (\{0,1\}\rightarrow [1])_v \arrow[r]\arrow[d] & ([1] \times [n+1] \rightarrow [1])_h \boxtimes (\{0,1\}\rightarrow [1])_v \arrow[d] \\
		P_n \arrow[r] &{ [1]_h \times [1;n+1]_v}
		\end{tikzcd}
	\]
	in $\PSh(\Delta^{\times 2})$ where 	$A_n := [1] \times [n] \cup_{\{0\}\times[n]} \{0\} \times [n+1]$. The maps that form this pushout and the top map  in the above square
	are now induced by the face map $$d_{n} : [n] \rightarrow [n+1].$$ The right arrow in this diagram is an equivalence, and the left arrow lies in the saturation of 
	the maps (\textit{Seg}) and (\textit{hc}) of \cite[Table 3.1.1]{FunDblCats}.

	We may concretely identify $P_0$. Consider the simplicial subset $[2] \simeq A'_0 \subset [1] \times [1]$ spanned by the 2-simplex 
	$(0,0) \leq (0,1) \leq (1,1)$. Then we obtain a factorization
	$$
	A_0 \rightarrow A'_0 \xrightarrow{i'} [1] \times [1].
	$$
	and the left arrow is a pushout along the $\Lambda^1[2] \rightarrow [2]$.  Consequently, the inclusion 
	$$
	(A_0 \rightarrow [1])_h \boxtimes  (\{0,1\}\rightarrow [1])_v \rightarrow (A'_0\rightarrow [1])_h \boxtimes (\{0,1\}\rightarrow [1])_v
	$$
	lies in the saturation of (\textit{Seg}) of \cite[Table 3.1.1]{FunDblCats}. Since the right-hand side is a double $\infty$-category, we may 
	assume that this is precisely $P_0$. Now, we 
	the desired retraction $r : [1]_h \times [1;1]_v \rightarrow P_0$ is induced by the (unique) fibered retraction 
	\[ 
		\begin{tikzcd}[column sep=tiny]
			{[1]} \times [1] \arrow[rr ,"r'"]\arrow[dr]&& A'_0\arrow[dl] \\
			& {[1]}
		\end{tikzcd}
	\]
	of $i'$ that carries $(1,0)$ to $(1,1)$.

	The proposition follows if we manage to show that the map
	$$
	(A_1 \cup_{[1] \times [1]} A'_0  \rightarrow [1])_h \boxtimes (\{0,1\}\rightarrow [1])_v \rightarrow ([1] \times [2] \cup_{[1] \times [1]} A'_0 \rightarrow [1])_h \boxtimes (\{0,1\}\rightarrow [1])_v
	$$
	lies in the saturation of (\textit{Seg}) of \cite[Table 3.1.1]{FunDblCats}. To this end, let us write 
	$$
	q : [1] \times [2]  \rightarrow [1] \times [2] \cup_{[1] \times [1]} A'_0 =: B_1
	$$
	for the canonical (quotient) map of simplicial sets.
	Recall that $[1] \times [2]$ is generated by the 3-simplices 
	$\sigma_1, \sigma_2, \sigma_3$ that correspond to the shuffles $(0,0) \leq (0,1) \leq (0,2) \leq (1,2)$, 
	$(0,0)\leq (0,1) \leq (1,1) \leq (1,2)$ and $(0,0) \leq (1,0) \leq (1,1) \leq (1,2)$ respectively. The images 
	under $q$ are still non-degenerate and also generate $B_1$. Let $\tau_1$ and $\tau_2$ denote the 
	2-simplices $(0,0) \leq (1,1) \leq (1,2)$ and $(1,0) \leq (1,1) \leq (1,2)$. 
	Let $F_{a,b} \subset B_1$ denote the smallest simplicial subset containing $A_1$ 
	together with the simplices $q(\tau_1), \dotsc, q(\tau_b)$ 
	and $q(\sigma_1), \dotsc, q(\sigma_b)$.

	With these definitions, we obtain a filtration 
	$$
	A_1 \cup_{[1] \times [1]} A'_0 = F_{0,0} \rightarrow F_{1,0} \rightarrow F_{2,0} \rightarrow F_{2,1} \rightarrow F_{2,2} \rightarrow F_{2,3} = B_1
	$$
	such that:
	\begin{enumerate}
		\item the inclusions $F_{0,0} \rightarrow F_{1,0}$ and $F_{1,0} \rightarrow F_{2,0}$ are pushouts along $\Lambda^1[2] \rightarrow [2]$,
		\item the inclusion $F_{2,0} \rightarrow F_{2,1}$ is a pushout along $\Lambda^2[3] \rightarrow [3]$,
		\item the inclusions $F_{2,1} \rightarrow F_{2,2}$ and $F_{2,2} \rightarrow F_{2,3}$ are both pushouts along $\Lambda^3[3] \rightarrow [3]$
		so that the restrictions to $\{2 \leq 3\}$ select the same degenerate edge of $B_1$.
	\end{enumerate}
	All five maps are pushouts of maps that are local to the (1-dimensional) Segal spaces in $\PSh(\Delta)$ (see \ref{not.rezk}). For (1), this follows by definition. 
	For (2), this follows from \cite[Lemma 3.5]{JoyalTierney}. For (3), this can be deduced from the proof of \cite[Lemma 11.10]{RezkSeg}.
	It now follows that the total composite $A_1 \cup_{[1] \times [1]} A'_0 \rightarrow B_1$ is local with regard to the inclusion $\Seg(\S) \subset \PSh(\Delta)$. 
	
	Consequently, the horizontal arrows in the commutative square 
	\[
	\begin{tikzcd}
		(A_1 \cup_{[1] \times [1]} A'_0)_h \arrow[d]\arrow[r] & (B_1)_h \arrow[d] \\
		{[1]}_h \arrow[r,equal] & {[1]}_h
	\end{tikzcd}
	\] are contained in the saturation of (\textit{Seg}) on account of the discussion in \ref{not.rezk}, and the fact that $(-)_h$ preserves 
	the equivalences local to the inclusion $\Seg^2(\S) \subset \PSh(\Delta^{\times 2})$. Now, we will use that the product bifunctor of $\PSh(\Delta^{\times 2})$ preserves 
	the morphisms in the saturation of (\textit{Seg}) in both variables, see \cite[Proposition 3.16]{FunDblCats} and \cite[Lemma 2.12]{FunDblCats}. 
	This implies that the pushout-product bifunctor $(-) \boxtimes (-)$ carries morphisms level-wise contained 
	in the saturation of (\textit{Seg}), in both variables, to morphisms in the saturation of (\textit{Seg}). 
	In particular, 
	when  we view the above square as an arrow 
	in $\fun([1], \PSh(\Delta^{\times 2}))$ and apply the functor $(-) \boxtimes (\{0,1\} \rightarrow [1])_v$ to it, we will remain in the saturation of (\textit{Seg}).
\end{proof}

\begin{proof}[Proof of \ref{thm.rep-tsdfib}]
	It remains to show that the span of the theorem is regular. We will only show property (1) of \ref{def.regular-span} since property (2) is formally dual.
	Namely, it may 
	be obtained from (1) by looking at the horizontal opposite $\P^\hop$ of $\P$, for which 
	the horizontal arrow double $\infty$-category is computed as $\{[1], \P^\hop\} = \{[1]^\op, \P\}^\hop$. 
	
	Let 
	$i$ and $r$ be as in \ref{lem.retraction-phin}. Suppose that $\beta : \{0\}_h \times [1;1]_v \rightarrow \P$ classifies an arrow of $\Vert(\P)(x,x')$ and $\alpha_0 : [1]_h \times [1]_v \rightarrow \P$ 
	classifies object of $\Vert(\{[1], \P\})(F,G)$, so that $\beta_0 := \{0\}_h \times [1]_v$ coincides with the restriction of $\alpha_0$
	to $\{0\}_h \times [1;\{0\}]_v$. Consider the candidate $\alpha$ defined by the composite
	$$
	[1]_h \times [1;1]_v \xrightarrow{r} P_0 \xrightarrow{\alpha'} \P,
	$$
	where the second map $\alpha'$ is the unique map determined by $\alpha_0$ and $\beta$. Note that $\alpha|\{1\}_v \times [1;1]_v$ is an equivalence in $\Vert(\P)(y,y')$ on account of the property 
	of the retract $r$ exhibited in \ref{lem.retraction-phin}. 
	It remains to show that $\alpha$ is cocartesian. To this end, we will make use of the characterization of \ref{lem.local-cocart-arrows-ev0}: we have to show that the map 
	$$
		\phi_n^* : \map({[1]_h} \times [1;n+1]_v, \P) \rightarrow \map({[1]_h} \times [1;1]_v \cup_{P_0} P_n , \P) 
	$$
	induced by $\phi_n$ is an equivalence at the fiber above $\alpha$. But this is precisely the fiber of the map 
	$$
		(r_*\phi_n)^* : \map(P_0 \cup_{[1]_h \times [1;1]_v} [1]_h \times [1;n+1]_v, \P) \rightarrow \map(P_n , \P) 
	$$
	induced by the cobase change $r_*\phi_n$ of $\phi_n$ at the fiber above $\alpha'$. But $r_*\phi_n$ is an equivalence on account of \ref{lem.retraction-phin} so that the result follows.
\end{proof}

\begin{remark}\label{rem.arrh-vert}
	The vertical mapping $\infty$-categories of $\{[1], \P\}$ are usually hard to access. There are some cases one can say more. Namely, 
	one may compute that 
	$$
	\Vert(\{[1], \P\})(\id_x, \id_y) \simeq \fun([1], \Vert(\P)(x,y))
	$$
	for $x,y \in \P$. However, we will not give a demonstration of this here.
\end{remark}

\subsection{Tabulations}
To define the notion of two-sided discrete fibrations internal to a double $\infty$-category, we will make use of the 
previously discussed horizontal arrow double $\infty$-category and one final ingredient: Grandis'\ and Par\'e's notion of \textit{tabulations} \cite{GrandisPareLimits}, \cite{GrandisPareSpan}.

\begin{definition}\label{def.tabulation}
	We say that a locally complete double $\infty$-category $\P$ is \textit{tabular} if it admits 
	limits shaped by the free horizontal arrow $[1]_h$. That is, whenever the diagonal functor 
	$$
	\Delta : \P \rightarrow \{[1], \P\}
	$$
	admits a lax right adjoint $t : \{[1], \P\} \rightarrow \P$. A \textit{tabulation} for a horizontal arrow $F:x\rightarrow y$ is the 
	data of an object $e$ of $\P$ 
	together with a so-called \textit{tabulating 2-cell} 
	\[
		\begin{tikzcd}
			e \arrow[r,equal]\arrow[d] & e \arrow[d] \\ 
			x \arrow[r, "F"] & y
		\end{tikzcd}
	\]
	in $\P$ that defines a counit for the adjunction $(\Delta, t)$: the composite map
	$$
	\map_{\P_0}(z, e) \rightarrow \map_{\P_1}(\id_z, \id_e) \rightarrow \map_{\P_1}(\id_z, F)
	$$
	should be an equivalence for any object $z$ of $\P$. In particular, we then have that $e \simeq tF$.
\end{definition}

\begin{remark}\label{rem.cotabs}
	There is a dual notion of \textit{cotabulations} in $\P$, these are precisely tabulations in the vertical opposite 
	$\P^\vop$.
\end{remark}

\begin{remark}
	Originally, tabulations were termed \textit{tabulators} by Grandis and Par\'e. We will use the terminology of Koudenburg \cite{Koudenburg}
	that is also used in the 2-categorical literature \cite{CarboniKasangianStreet}.
\end{remark}

\begin{example}
	The $\infty$-equipment $\CCAT_\infty(\E)$ associated to an $\infty$-topos $\E$ is tabular in light of \ref{ex.int-ccat-dbl-limits}.
\end{example}

\begin{example}
	If $F : \C \rightarrow \D$ 
	is a proarrow between $\infty$-categories given by a correspondence $E \rightarrow [1]$, then its tabulating 2-cell is given by
	\[
		\begin{tikzcd}
			\fun_{/[1]}([1], E) \arrow[r,equal]\arrow[d, "\ev_1"'] & \fun_{/[1]}([1], E) \arrow[d, "\ev_0"] \\ 
			\C \arrow[r, "F"] & \D
		\end{tikzcd}
	\]
	that is classified by the canonical map of correspondences $$[1] \times \fun_{/[1]}([1], E) \rightarrow E.$$ The span
	$\C \leftarrow \fun_{/[1]}([1], E) \rightarrow \D$ is precisely the two-sided discrete fibration associated to the correspondence $E \rightarrow [1]$ under the identifications discussed in \cite[Section 4.1]{EquipI}, see also the work of Ayala-Francis \cite[Section 4.1]{AyalaFrancis}.
	The cotabulation of $F$ is given by the 2-cell
	\[
		\begin{tikzcd}
			\C \arrow[r, "F"]\arrow[d] & \D\arrow[d] \\
			E \arrow[r, equal] & E
		\end{tikzcd}
	\]
	where the vertical arrows are inclusions of fibers, corresponding to the map of correspondences $E \rightarrow [1] \times E$.
\end{example}

\begin{definition}
	Suppose that $\P$ is a tabular locally complete double $\infty$-category. Then 
	a span $$(p,q) : e \rightarrow x \times y$$ in $\Vert(\P)^{(1)}$ is called a \textit{two-sided discrete fibration} (internal to $\P$) 
	if there exists a 2-cell of the form
	\[
		\begin{tikzcd}
			e \arrow[r,equal]\arrow[d,"p"'] & e \arrow[d, "q"] \\ 
			x \arrow[r, "F"] & y
		\end{tikzcd}
	\]
	in $\P$ that is tabulating.
\end{definition}

\begin{remark}\label{rem.tsdfib-street}
There are more strategies for defining suitable generalizations of two-sided discrete fibrations of $(\infty-)$categories. In a stricter context, 
this was done by Street \cite{StreetFib} where the candidates internal to a suitable \textit{2-category} $\X$ are introduced:
\begin{enumerate}
	\item the so-called \textit{covering spans} in $\X$,
	\item the more general notion of \textit{discrete split bifibrations} in $\X$, this is a representable notion:\ a span is a discrete split bifibration if and only if 
	it induces a two-sided discrete fibration on mapping categories.
\end{enumerate}
There is no discrepancy between definitions (1) and (2) if $\X$ is given by the 2-category of categories (cf.\ \cite[Remark 3.44]{StreetFib2}, \cite[Proposition 4.10]{CarboniJohnsonStreetVerity}), 
but this is generally not the case.

One can interpret our definition as a double $\infty$-categorical variant on definition (1). In the 2-categorical literature, 
usually approach (2) is taken to define two-sided discrete fibrations. We depart from this strategy as this definition would only depend on the vertical fragment of the double $\infty$-categories that we consider. We stress
that our definition of two-sided discrete fibrations depends on the horizontal categorical direction as well.
\end{remark}

The following result shows that two-sided discrete fibrations internal to a suitable double $\infty$-category are also two-sided discrete fibrations representably:

\begin{proposition}\label{prop.tsdfib-rep}
	Suppose that $\P$ is a tabular locally complete double $\infty$-category. If $(p,q) : e \rightarrow x \times y$ 
	is a two-sided discrete fibration in $\P$, then 
	$$
	\Vert(\P)(z,e) \rightarrow \Vert(\P)(z,x)\times \Vert(\P)(z,y)
	$$ 
	is a two-sided discrete fibration for every $z \in \P$.
\end{proposition}
\begin{proof}
	By assumption, there exists a 2-cell \[
		\begin{tikzcd}
			e \arrow[r,equal]\arrow[d,"p"'] & e \arrow[d, "q"] \\ 
			x \arrow[r, "F"] & y,
		\end{tikzcd}
	\]
	that forms a counit for the adjunction $(\Delta, t)$ so that
	$$
	\Vert(\P)(z,e) \simeq \Vert(\{[1], \P\})(\id_z, F).
	$$
	The desired result can now be deduced from \ref{thm.rep-tsdfib}.
\end{proof}

Some double functors interact well with tabulations:

\begin{proposition}\label{prop.lax-radj-pres-tabs}
	Normal lax right adjoint functors between tabular locally complete double $\infty$-categories 
	 preserve tabulating 2-cells, and, in particular, preserve two-sided discrete fibrations.
\end{proposition}
\begin{proof}
	Let $v : \Q \rightarrow \P$ be such a right adjoint with left adjoint $u : \P \rightarrow \Q$. 
	Suppose that $\alpha$ is a tabulating 2-cell in $\Q$, i.e.\ it is 
	given by a morphism $\id_e \rightarrow F$ in $\Q_1$ as in \ref{def.tabulation}. Since $v$ is normal,  $v(\alpha)$ is a morphism $\id_{v(e)} \simeq v(\id_e) \rightarrow vF$ in $\P_1$. Moreover, by normality, 
	the composite
	$$
	\map_{\P_0}(z,v(e)) \rightarrow \map_{\P_1}(\id_{z}, \id_{v(e)}) \rightarrow \map_{\P_1}(\id_{z}, vF)
	$$
	is equivalent to the composite 
	$$
	\map_{\Q_0}(u(z),e) \rightarrow \map_{\Q_1}(\id_{u(z)}, \id_{e}) \rightarrow \map_{\Q_1}(\id_{u(z)}, F),
	$$
	cf.\ \ref{prop.lax-adj-vert}. The latter composite map is an equivalence by assumption.
\end{proof}

\begin{proposition}\label{prop.geom-pres-tabs}
	Let $f : \E \rightarrow \F$ be a geometric morphism between $\infty$-toposes. Then the functor $f^*$ and the lax functor $f_*$ (see \ref{ex.dbl-adj-geom-morph}) preserve tabulating 2-cells and 
	two-sided discrete fibrations.
\end{proposition}
\begin{proof}
	Since $f_*$ is a normal lax right adjoint, this case follows from \ref{prop.lax-radj-pres-tabs}.  For the case of 
	$f^*$, note that the relevant tabulations are computed level-wise by limits shaped by the finite sets $\map_{\Delta}([n], [1])$, $[n] \in \Delta$, in light of \ref{ex.int-ccat-dbl-limits}. Such limits 
	are preserved by inverse image functors.
\end{proof}

Moreover, we have the following closure property:

\begin{proposition}\label{prop.tabs-cotensor}
	Let $\P$ be a tabular locally complete double $\infty$-category. If $X$ is an $\infty$-category, then the vertical cotensor product 
	$[X, \P]$ is again tabular and locally complete. Moreover, a 2-cell 
	\[
		\alpha = \begin{tikzcd}
			e \arrow[r,equal]\arrow[d] & e \arrow[d] \\ 
			x \arrow[r, "F"] & y
		\end{tikzcd}	
	\] in $[X,\P]$ is tabulating if and only if the evaluation $\alpha_\xi$ is a tabulating 2-cell in $\P$ for all $\xi \in X$.
\end{proposition}
\begin{proof}
	The locally complete assertion follows from \cite[Proposition 3.16]{FunDblCats}. The tabular part and the characterization of tabulations may be deduced from \ref{prop.lax-double-adjunctions-char}, \ref{rem.cotensor-infty-cat} and the 
	fact that $\fun(X,-)$ preserves adjunctions. 
\end{proof}

\section{Fibrational \texorpdfstring{$\infty$}{∞}-equipments}\label{section.fib-equipments}

The goal of this section 
is to axiomatize and study a class of $\infty$-equipments that support a good theory of fibrations. The key feature of these equipments should be 
that there is no information loss when passing from proarrows to their associated two-sided discrete fibrations.

\subsection{Span representations}
We will demonstrate that one can pass from proarrows to two-sided discrete fibrations in a coherent fashion:

\begin{theorem}\label{thm.span-rep}
	Suppose that $\P$ is a tabular locally complete double $\infty$-category with all pullbacks. Then there exists a lax 
	functor 
	$$
	\rho : \P \rightarrow \SSPAN(\Vert(\P)^{(1)})
	$$
	that lifts the identity functor on $\Vert(\P)^{(1)}$ and carries each horizontal arrow $F:x \rightarrow y$ to the span $(p,q) : e\rightarrow x\times y$ associated to its tabulating 2-cell 
	\[
		\begin{tikzcd}
			e \arrow[r,equal]\arrow[d, "p"'] & e \arrow[d, "q"] \\ 
			x \arrow[r, "F"] & y.
		\end{tikzcd}
	\] 
	Moreover, if $\P$ is an $\infty$-equipment, then $\rho$ is locally horizontally reflective if and only if the following conditions are met: 
	\begin{enumerate}
		\item all tabulating 2-cells are cocartesian,
		\item for two composable proarrows $F : x \rightarrow y$ and $G : y \rightarrow z$ in $\P$ 
		with (cocartesian) tabulating 2-cells 
		\[
			\begin{tikzcd}
				e \arrow[r,equal]\arrow[d] & e \arrow[d] \\ 
				x \arrow[r, "F"] & y
			\end{tikzcd}
			\text{ and }
			\begin{tikzcd}
				e' \arrow[r,equal]\arrow[d] & e' \arrow[d] \\ 
				y \arrow[r, "G"] & z,
			\end{tikzcd}
		\]
		the composite of the two 2-cells 
		\[
			\begin{tikzcd}
				e \times_y e' \arrow[r,equal]\arrow[d] & e \times_y e' \arrow[d]\arrow[r,equal]& e \times_y e' \arrow[d] \\ 
				x \arrow[r, "F"] & y \arrow[r, "G"] & z
			\end{tikzcd}
		\]
		obtained by whiskering these tabulating 2-cells with the projections $e \times_y e' \rightarrow e$ and $e\times_y e' \rightarrow e'$, is again cocartesian.
	\end{enumerate}
	In this case, the right adjoint to the fully faithful functor
	$$
	\rho_{x,y} : \Hor(\P)(x,y) \rightarrow \Vert(\P)^{(1)}/(x\times y)
	$$
	carries a span $(p,q) : e \rightarrow x \times y$ to the cocartesian pushforward ${q}_\circledast{p}^\circledast$.
\end{theorem}

\begin{definition}
	The functor $\rho$ from \ref{thm.span-rep} will be called the \textit{span representation}.
\end{definition}

\begin{remark}
	The construction of span representations for strict double categories
	has been studied by Grandis and Par\'e \cite{GrandisPareSpan} and Niefield \cite{Niefield}. Condition (1) of \ref{thm.span-rep} has been 
	considered earlier by Koudenburg \cite{Koudenburg} for strict equipments.
\end{remark}

We defer the proof of the theorem to the next subsection. Let us first study the class of equipments with horizontally locally reflective span representations.

\begin{definition}\label{def.fib-equipment}
	An $\infty$-equipment $\P$ is called \textit{fibrational} if it is tabular and admits all pullbacks, and 
	conditions (1) and (2) of \ref{thm.span-rep} are met.
\end{definition}

The prototypical example of a fibrational $\infty$-equipment that we will study in this paper, is the following:

\begin{proposition}\label{prop.int-ccat-fib}
The $\infty$-equipment $\CCAT_\infty(\E)$ is fibrational for every $\infty$-topos $\E$, and hence 
it admits a horizontally locally reflective span representation
$$
\rho:\CCAT_\infty(\E) \rightarrow \SSPAN(\Cat_\infty(\E)).
$$
\end{proposition}

\begin{remark}
	For $\E = \S$, this gives a double categorical answer to a question posed by Ayala and Francis \cite[Problem 0.17]{AyalaFrancis} about the 
	existence of such a functor.
\end{remark}

To prove \ref{prop.int-ccat-fib} for all $\infty$-toposes, we can reduce to the prototypical $\infty$-topos of spaces. We will need the following observations:

\begin{lemma}\label{lem.cocomma}
	Suppose that $(p,q) : E \rightarrow \C \times \D$ is a span of $\infty$-categories.
	Let us write $P \rightarrow [1]$ for the cocomma correspondence that is defined by the pushout 
	$$
	P := \C \cup_{\{1\} \times E} [1] \times E \cup_{\{0\} \times E} \D.
	$$
	Then the following is true:
	\begin{enumerate}
		\item the canonical lax square \[
			\begin{tikzcd}
				E \arrow[r,"q"]\arrow[d,"p"'] &|[alias=f]| \D \arrow[d,"j"] \\ 
				|[alias=t]|\C \arrow[r, "i"'] & P 
				\arrow[from=f,to=t, Rightarrow]
			\end{tikzcd}
		\]
		is exact, 
		\item there is a canonical map $E_{c,d} \rightarrow \map_P(d,c)$
		which exhibits $\map_P(d,c)$ as the classifying space of the fiber $E_{c,d}$ if $(p,q)$ is regular.
	\end{enumerate}
\end{lemma}
\begin{proof}
	Let us first show (1). To show exactness \cite[Definition 6.24]{EquipI}, 
	we first produce a (essentially unique) factorization 
	\[
		\begin{tikzcd}
			E \arrow[r,equal] \arrow[d,"p"'] & E \arrow[d, "q"] \\
			\C\arrow[d,"i"'] \arrow[r, "q_\circledast p^\circledast"] & \D\arrow[d,"j"] \\
			P \arrow[r,equal] & P
		\end{tikzcd}
	\]
	of the lax square in $\CCAT_\infty$, where the top 2-cell is cocartesian. Then we must show that the bottom 2-cell is cartesian.
	From the viewpoint of correspondences, this factorization is supplied by the composite
	$$
	[1]\times E \rightarrow P \rightarrow [1] \times P
	$$
	in $\Cat_\infty/[1]$. This follows because the first arrow classifies a cocartesian arrow of the source-target projection $(\{1\}, \{0\})^* : \Cat_\infty/[1] \rightarrow \Cat_\infty^{\times 2}$ of $\CCAT_\infty$, as one can verify using the universal property of the pushout. 
	We thus have to show that the arrow $P\rightarrow [1] \times P$ is cartesian with respect to $(\{1\}, \{0\})^*$. This entails verifying that the commutative square 
	\[
		\begin{tikzcd}
		\map_{\Cat_\infty/[1]}(X, P) \arrow[r]\arrow[d] & \map_{\Cat_\infty}(X, P) \arrow[d] \\
		\map_{\Cat_\infty}(X_1, \C) \times \map_{\Cat_\infty}(X_0, \D) \arrow[r, "{i_* \times j_*}"] & \map_{\Cat_\infty}(X_1, P) \times \map_{\Cat_\infty}(X_0, P)	
		\end{tikzcd}
	\]
	is a pullback square for every correspondence $X \rightarrow [1]$. Note that the bottom arrow 
	may be identified with the arrow 
	$$
	\map_{\Cat_\infty/[1]}(X_1, P) \times \map_{\Cat_\infty/[1]}(X_0, P) \rightarrow \map_{\Cat_\infty}(X_1, P) \times \map_{\Cat_\infty}(X_0, P)	
	$$
	induced by the forgetful functor. 
	
	In fact, we may show that the above diagram is a pullback square for any correspondence 
	$P$ from $\C$ to $\D$. To verify this, it will be sufficient to check for $\sigma : [n] \rightarrow [1]$ in $\Delta/[1]$, since there is fully faithful inclusion $\Cat_\infty/[1] \rightarrow \PSh(\Delta)/[1] \simeq \PSh(\Delta/[1])$, see the proof of \cite[Proposition 5.19]{EquipI} for instance.
	In this case, the above square fits as the top square in a commutative cube
	\[
		\begin{tikzcd}[column sep = 0pt, row sep = small]
			\map_{/[1]}([n], P) \arrow[dr]\arrow[rr]\arrow[dd] && \map([n], P)\arrow[dr]\arrow[dd] \\
			& |[alias=f]|\map_{/[1]}([n]_1 \sqcup [n]_0, P)\arrow[rr, crossing over] && \map([n]_1 \sqcup [n]_0, P)\arrow[dd]\\ 
			\ast \arrow[dr, equal]\arrow[rr ,"\{\sigma\}"] && \map([n], [1])\arrow[dr] \\
			& |[alias=t]| \ast \arrow[rr] &&  \map([n]_1 \sqcup [n]_0, [1]).
			\arrow[from=f,to=t, crossing over]
		\end{tikzcd}
	\]
	Here $[n]_1$ and $[n]_0$ denote the fibers of $\sigma$ above $1$ and $0$ respectively. 
	The front and back squares are pullback squares. The bottom square is clearly a pullback square as well, hence the result follows.

	Let us show (2). Suppose that $(p,q)$ is regular.
	Consider the following lax commutative diagram
	\[
		\begin{tikzcd}
			E_{c,d} \arrow[r]\arrow[d]& E \times_\D \{d\} \arrow[d] \arrow[r] & \{d\}\arrow[d]\\ 
			\{c\} \times_\C E \arrow[d] \arrow[r]  &E \arrow[r,"q"]\arrow[d,"p"'] &|[alias=f]|  \D \arrow[d,"j"] \\ 
			\{c\} \arrow[r]  & |[alias=t]| \C \arrow[r, "i"'] & P 
			\arrow[from=f,to=t, Rightarrow]	
		\end{tikzcd}
	\]
	of $\infty$-categories. This produces a mate transformation 
	$$
	E_{c,d}[E_{c,d}^{-1}] \times (-)  \rightarrow \{c\}^*i^*j_!\{d\}_!(-)
	$$ 
	between endofunctors on $\S = \fun(\{d\}, \S)$. Evaluating at the point $\ast$, we obtain a canonical map 
	$E_{c,d}[E_{c,d}^{-1}] \rightarrow \map_P(d,c)$. On account of \cite[Corollary 6.27]{EquipI}, the mate is an equivalence if the outer lax square in the above diagram
	is exact. But the right bottom square is exact by assertion (1). To show that the 
	bottom pullback square is exact, it suffices to verify that $p$ is proper 
	(see \cite[Example 6.28]{EquipI}). But this follows from 
	the assumption that $p$ is cocartesian and \cite[Proposition 4.1.2.15]{HTT}.\footnote{We will use the proper/smooth terminology of Cisinski \cite{Cisinski}. Unfortunately, this is precisely the opposite of the conventions used by Lurie \cite{HTT}.} 
	Similarly, we can show that the outer (pullback) square of the two top squares is exact. 
	This follows from the fact that the composite $\{c\} \times_\C E \rightarrow \D$ is cartesian by \cite[Proposition 2.3.3]{HHLN}, hence smooth. 
	From this, the result follows.
\end{proof}

\begin{proof}[Proof of \ref{prop.int-ccat-fib}]
		We may pick a geometric embedding $i : \E \rightarrow \PSh(T)$ for some small $\infty$-category $T$. 
		This induces a reflective lax inclusion $$i_* : \CCAT_\infty(\E) \rightarrow [T^\op, \CCAT_\infty].$$
		Now in light of \ref{prop.geom-pres-tabs} and \ref{prop.refl-incl-formal-cat-theory}, it is necessary and sufficient to show that 
		$[T^\op, \CCAT_\infty]$ is fibrational. In turn, \ref{prop.tabs-cotensor} and \ref{prop.cart-cotensors} 
		imply that we may reduce to showing that $\CCAT_\infty$ is fibrational.
		
		In this case, condition (1) of \ref{thm.span-rep} has been demonstrated 
		by Lurie in \cite[Proposition B.3.17]{HA} and by Ayala-Francis in \cite[Lemma 4.1]{AyalaFrancis}, where it is shown that the canonical map 
		$$
		E_1 \cup_{[1] \times E_1} [1] \times \fun_{/[1]}([1], E) \cup_{[1] \times E_0} E_0 \rightarrow E 
		$$ 
		is an equivalence for every correspondence $E \rightarrow [1]$. 
		Condition (2) is precisely the assertion that the canonical map 
		$$
		E_2 \cup_{[1] \times E_2} [1] \times \fun_{/[2]}([2], E) \cup_{[1] \times E_0} E_0 \rightarrow d_1^*E 
		$$
		is an equivalence in $\Cat_\infty/[1]$ for every Conduch\'e fibration $E \rightarrow [2]$. By the above, 
		this is true precisely when the functor 
		$$
		E_2 \cup_{[1] \times E_2} [1] \times \fun_{/[2]}([2], E) \cup_{[1] \times E_0} E_0 \rightarrow E_2 \cup_{[1] \times E_2} [1] \times \fun_{/[1]}([1], d_1^*E) \cup_{[1] \times E_0} E_0
		$$
		is an equivalence. It is clearly essentially surjective, 
		so it suffices to show that it is fully faithful. This follows directly from the computation 
		(2) of 
		\ref{lem.cocomma} and characterization (6) of \cite[Lemma 5.16]{AyalaFrancisRozenblyum}.
\end{proof}

The span representations can be used to obtain the following double $\infty$-categorical version of a theorem of Carboni-Kasangian-Street \cite[Theorem 4]{CarboniKasangianStreet} about recognizing span bicategories.

\begin{corollary}\label{cor.char-span-dbl-cats}
	Let $\P$ be a double $\infty$-category. Then the following are equivalent:
	\begin{enumerate}
		\item $\P$ is equivalent to $\SSPAN(\C)$, where $\C$ is some $\infty$-category with pullbacks,
		\item $\P$ is a fibrational $\infty$-equipment and  every span in $\Vert(\P)^{(1)}$ is a two-sided discrete fibration (w.r.t.\ $\P$),
		\item the span representation for $\P$ is strict and an equivalence.
	\end{enumerate}
\end{corollary}
\begin{proof}
	It is clear that (3) implies (1) and it is straightforward to verify that every double $\infty$-category of spans meets condition (2). It remains to show that (2) implies (3). 
	Let $$\rho : \P \rightarrow \SSPAN(\Vert(\P)^{(1)})$$ be the span representation for $\P$. Then $\rho_0$ is the identity and $\rho_1$ is fully faithful in light of \ref{prop.hor-refl-char-equipments}.
	It is essentially surjective by the assumption that every span in $\P$ is a two-sided discrete fibration. Thus $\rho_1$ is an equivalence. Now 
	\ref{prop.hor-refl-char-equipments} implies that the lax naturality square associated to $\alpha = d_1 : [1] \rightarrow [2]$ strictly commutes. 
	In turn, one can then inductively show that the lax naturality square associated with every face map $\alpha$ in $\Delta$ commutes. Similarly, 
	one shows that span representation is also normal, since for every $x \in \P$, the span 
	$x \rightarrow x \times x$ is two-sided discrete. Thus $\rho$ is strict.
\end{proof}

\subsection{The existence of span representations}
To prove \ref{thm.span-rep}, we will formulate a technical and intermediate step that takes care of the construction of the span representation. We will make use of a certain universal property for mapping into span double categories 
exhibited by Haugseng  \cite{HaugsengSpans2}.

\begin{lemma}\label{lem.span-rep}
	Suppose that $\P$ is a double $\infty$-category with the property that $\P_0$ admits pullbacks 
	and the structure map $\P_0 \rightarrow \P_1$ admits a right adjoint. Then there exists a lax functor 
	$$\rho: \P \rightarrow \SSPAN(\Vert(\P)^{(1)})$$ 
	so that: 
	\begin{enumerate}
		\item The functor $\rho_0$ can be identified with the identity functor.
		\item If we have a 2-cell \[
			\begin{tikzcd}
				e \arrow[r,equal]\arrow[d, "p"'] & e \arrow[d,"q"] \\ 
				x \arrow[r, "F"] & y
			\end{tikzcd}
		\] in $\P$ that is tabulating, then $\rho$ carries $F$ 
		to the span $(p,q) : e \rightarrow x\times y$.
		\item If we have tabulating 2-cells 
		\[
			\begin{tikzcd}
				e \arrow[r,equal]\arrow[d] & e \arrow[d] \\ 
				x \arrow[r, "F"] & y, 
			\end{tikzcd}
			\begin{tikzcd}
				e' \arrow[r,equal]\arrow[d] & e' \arrow[d] \\ 
				y \arrow[r, "G"] & z
			\end{tikzcd}
			\text{ and }
			\begin{tikzcd}
				e'' \arrow[r,equal]\arrow[d] & e'' \arrow[d] \\ 
				x \arrow[r, "GF"] & z
			\end{tikzcd}
		\] 
		in $\P$, then the lax comparison 2-cell 
		$e \times_y e' \rightarrow  e''$ in $\Vert(\P)^{(1)}/(x \times z)$ provided by $\rho$ is obtained by applying 
		the universal property of tabulating 2-cells to
		the composite 2-cell $\id_{e \times_y e'} \rightarrow G\circ F$ described in condition (3) of \ref{def.fib-equipment}.
	\end{enumerate}
\end{lemma}
\begin{proof}
	We will provide a construction of the functor $\rho$ from the fibrational perspective of \ref{rem.lax-fun-fib}.
	Consider the canonical inclusion $l : \Vert(\P)^{(1)}_v = (\P_0)_v \rightarrow \P$ of the underlying vertical $\infty$-category. Passing to unstraightenings, this gives rise 
	to a map 
	$$
	l : \P_0 \times \Delta^\op \rightarrow \mathrm{Un}(\P)
	$$
	of cocartesian fibrations over $\Delta^\op$. On account of \cite[Proposition 7.3.2.6]{HA}, this is a left adjoint in $\CAT_\infty/\Delta^\op$ if and only if 
	for every $n$, the degeneracy functor $l_n : \P_0 \rightarrow \P_n$ induced by the terminal map $[n] \rightarrow [0]$ is a left adjoint. We claim that this is the case. 
	Note that $l_0$ is trivially a left adjoint (it is the identity) and that $l_1$ is a left adjoint by assumption.

	Let $i : \mathbb{G}^\op \rightarrow \Delta^\op$ denote the inclusion of the full subcategory spanned by $[0]$ and $[1]$. By the same reasoning, there exists a right adjoint 
	$i^*\mathrm{Un}(\P) \rightarrow \P_0 \times \mathbb{G}^\op$ for $i^*l$ in $\CAT_\infty/\mathbb{G}^\op$.
	Let us write $s : i^*\mathrm{Un}(\P) \rightarrow \P_0$ for this adjoint followed by the projection $\P_0 \times \mathbb{G}^\op \rightarrow \P_0$. 
	We then obtain a natural equivalence
	$$
	\map_{\mathrm{Un}(\P)}(x,y) = \map_{\mathrm{Un}(\P)}(l_0(x),y) \simeq \map_{\P_0}(x, s(y)) \times \map_{\mathbb{G}^\op}([0], [n]) =  \map_{\P_0}(x, s(y))
	$$ 
	for $x \in \P_0$ and $y \in \P_n$ with $n=0,1$. We will use this to show that $l_n$ admits a left adjoint.

	Suppose that $y\in \P_n$. Then we have to show that the presheaf 
	$$
	\map_{\P_n}(l_n(-), y) : \P_0^\op \rightarrow \S
	$$
	is representable. We may construct a natural transformation $(-) \rightarrow l_n(-)$ in $\mathrm{Un}(\P)$ that, component-wise, 
	consists of 
	cocartesian lifts over the arrow $[0] \rightarrow [n]$ in $\Delta^\op$. Then, for $x \in \P_0$, we obtain a natural commutative diagram
	\[
		\begin{tikzcd}
			\map_{\P_n}(l_n(x), y) \arrow[r]\arrow[d] &\map_{\mathrm{Un}(\P)}(l_n(x), y) \arrow[r]\arrow[d] & \map_{\mathrm{Un}(\P)}(x, y) \arrow[d] \\
			\{\id_{[n]}\} \arrow[r] & \map_{\Delta^\op}([n], [n]) \arrow[r] & \map_{\Delta^\op}([0], [n])
		\end{tikzcd}
	\] 
	in which every square is a pullback square.
	Since $\P$ is a double $\infty$-category, we have a further pullback square  
	\[
		\begin{tikzcd}
			\map_{\mathrm{Un}(\P)}(x, y) \arrow[r]\arrow[d] & \lim_{ i : [k] \rightarrow [n] \in (\mathbb{G}/[n])^\op}\map_{\mathrm{Un}(\P)}(x, i^*y) \arrow[d] \\
			\map_{\Delta^\op}([0], [n]) \arrow[r] &\lim_{ i : [k] \rightarrow [n] \in (\mathbb{G}/[n])^\op}\map_{\Delta^\op}([0], [k]),
		\end{tikzcd}
	\] 
	see item (iii) of \cite[Definition 2.6]{HaugsengSpans2}.
	Thus we deduce that 
	\begin{align*}
	\map_{\P_n}(l_n(-), y) &\simeq \lim_{ i : [k] \rightarrow [n] \in (\mathbb{G}/[n])^\op}\map_{\mathrm{Un}(\P)}(-, i^*y) \\&\simeq  \lim_{ i : [k] \rightarrow [n] \in (\mathbb{G}/[n])^\op}\map_{\P_0}(-, s(i^*y)).
	\end{align*}
	Since $\P_0$ admits all pullbacks, the latter presheaf is representable as desired.

	Thus we conclude that $l$ admits a fibered right adjoint over $\Delta^\op$. Let us write 
	$$
	s' : \mathrm{Un}(\P) \rightarrow \P_0
	$$
	for this right adjoint followed by the projection $\P_0 \times \Delta^\op \rightarrow \P_0$. By the reasoning of above, 
	we see that $s'$ may be computed as 
	$$
	s'(y) \simeq \lim_{ i : [k] \rightarrow [n] \in (\mathbb{G}/[n])^\op}s(i^*y)
	$$
	for $y \in \P_n$.
	Hence, the universal property of spans proven by Haugseng in \cite[Corollary 3.10]{HaugsengSpans2} together with its succeeding remark, 
	asserts that $s'$ has a uniquely associated map 
	$$
	\rho : \mathrm{Un}(\P) \rightarrow \mathrm{Un}(\SSPAN(\P_0))
	$$
	over $\Delta^\op$ that preserves cococartesian morphisms over inerts. Tracing through the proof of \cite[Corollary 3.10]{HaugsengSpans2}, we deduce the following:
	\begin{enumerate}[label=(\roman*)]
		\item The functor $\rho_0$ is the identity.
		\item The functor $\rho_1$ carries a horizontal arrow $F:x \rightarrow y$ to the associated span $s(F) \rightarrow x \times y$.
		\item If $F : x \rightarrow y$ and $G : y \rightarrow z$ are  
		horizontal arrows, then the image of the cocartesian lift $(G,F) \rightarrow G\circ F$ of $d_1 : [2] \rightarrow [1]$ in $\Delta^\op$ factors in 
		$ \mathrm{Un}(\SSPAN(\P_0))$ as a composite 
		$$
		\rho(G,F) \rightarrow d_1^*(\rho(G,F)) \rightarrow \rho(G \circ F)
		$$
		where the first arrow is cocartesian. The second arrow lives in the fiber $\SSPAN(\P_0)_1$ and is precisely computed as in (3).
	\end{enumerate}
	Thus $\rho$ straightens to the desired lax functor in light of the discussion in  \cite[Remark 5.24]{FunDblCats}.
\end{proof}

\begin{proof}[Proof of \ref{thm.span-rep}]
	Since $\P$ is tabular, \ref{lem.span-rep} produces the candidate lax functor
	$$
	\rho : \P \rightarrow \SSPAN(\Vert(\P)^{(1)}).
	$$
	which has the property that $\rho_0$ recovers the identity. 
	
	Let us now assume that $\P$ is an $\infty$-equipment. It remains to show that $\rho$ is horizontally locally reflective if and only if conditions (1) and (2) are met.
	Let $x, y$ be objects of $\P$. Then we have to show that the functor
	$$\rho_{x,y}:\Hor(\P)(x,y) \rightarrow \Vert(\P)^{(1)}/(x\times y)$$
	has a right adjoint. I.e., we have to show that for every span $(p,q) : e \rightarrow x\times y$ the copresheaf 
	$$
	\map_{\P_0/(x\times y)}(e,\rho_{x,y}(-)) : \Hor(\P)(x,y) \rightarrow \S
	$$ 
	is corepresentable. If $F : x \rightarrow y$ is a horizontal arrow in $\P$, then we have a natural pullback square 
	\[
		\begin{tikzcd}
		\map_{\P_0/(x\times y)}(e,\rho_{x,y}F) \arrow[r]\arrow[d] & \map_{\P_0}(e, \rho_{x,y}F) \arrow[d] \\
		\{(p,q)\} \arrow[r] & \map_{\P_0}(e,x) \times \map_{\P_0}(e,y).
		\end{tikzcd}
	\]
	This can be rewritten as the natural pullback square
	\[
		\begin{tikzcd}
		\map_{\P_0/(x\times y)}(e,\rho_{x,y}F) \arrow[r]\arrow[d] & \map_{\P_1}(\id_e, F) \arrow[d] \\
		\{(p,q)\} \arrow[r] & \map_{\P_0}(e,x) \times \map_{\P_0}(e,y)
		\end{tikzcd}
	\]
	using the universal property of tabulating 2-cells.
	Since $\P$ admits all cocartesian 2-cells, we obtain that
	$$
	\map_{\P_0/(x\times y)}(e,\rho_{x,y}F) \simeq \map_{\Hor(\P)(x,y)}(\cocart{q}{}{p}, F)
	$$
	naturally in $F$. Thus $\rho_{x,y}$ always has a right adjoint $L_{x,y}$ that carries a span $(p,q) : e\rightarrow x\times y$ to the cocartesian lift $\cocart{q}{}{p}$. Now, $\rho_{x,y}$ is fully faithful if and only if the counit 
	$L_{x,y}\rho_{x,y}F \rightarrow F$ is an equivalence for every horizontal arrow $F : x \rightarrow y$. But this is precisely the condition (1).
	Moreover, using the description given in \ref{lem.span-rep} of the lax compatibility 2-cells of $\rho$,  it can be readily verified 
	that condition (2) is met if and only if $\rho$ meets condition (2) of \ref{def.hor-loc-refl}.
\end{proof}

\subsection{Basic theory of fibrations} 
We will now fix a fibrational $\infty$-equipment $\P$ and demonstrate that such an equipment supports a good theory of fibrations.

By \ref{thm.span-rep}, there is a horizontally locally reflective span representation $$\rho : \P \rightarrow \SSPAN(\Vert(\P)^{(1)})$$ 
so that we locally obtain an adjunction 
$$
L_{x,y} : \Vert(\P)^{(1)}/(x\times y) \rightleftarrows \Hor(\P)(x,y) : \rho_{x,y}
$$
for objects $x$ and $y$. The functor $\rho_{x,y}$ is fully faithful in this case. The local objects of this adjunction are precisely 
the two-sided discrete fibrations. The left adjoint $L_{x,y}$ carries a span $(p,q) : e \rightarrow x \times y$ 
to the cocartesian lift $\cocart{q}{}{p}$. 

\begin{convention}
	We say that a two-sided discrete fibration $(p,q) : e \rightarrow x\times y$ \textit{classifies} a profunctor $F : x \rightarrow y$
	if $\cocart{q}{}{p} \simeq F,$ i.e.\ when there exists a (necessarily cocartesian) tabulating 2-cell
	\[
		\begin{tikzcd}
			e \arrow[r,equal]\arrow[d, "p"'] & e \arrow[d, "q"] \\ 
			x \arrow[r, "F"] & y.
		\end{tikzcd}
	\] 
\end{convention}

\begin{proposition}\label{prop.tsdfib-closed-pb}
	Suppose that $e \rightarrow x\times y$ is a two-sided discrete fibration that classifies a proarrow $F : x\rightarrow y$. 
	Given arrows $f : x'\rightarrow x$ and $g : y' \rightarrow y$, we can form the pullback
	\[
		\begin{tikzcd}
			e' \arrow[r] \arrow[d] & e \arrow[d] \\ 
			x'\times y' \arrow[r, "f\times g"] & x \times y,
		\end{tikzcd}
	\]
	and the left vertical arrow is again a two-sided discrete fibration that classifies the profunctor $\cart{g}{F}{f} : x' \rightarrow y'$. 
\end{proposition}
\begin{proof}
	The span representation $\rho$ carries the cartesian 2-cell
	\[
	\begin{tikzcd}
		x' \arrow[d,"f"'] \arrow[r, "\cart{g}{F}{f}"] & y' \arrow[d,"g"] \\
		x \arrow[r,"F"] & y
	\end{tikzcd}
	\]
	to a cartesian 2-cell in $\SSPAN(\Vert(\P)^{(1)})$ on account of \ref{cor.lax-fun-eq-pres-cart-cells}. The result follows directly from this observation. 
\end{proof}

We have some suggestively named important examples of two-sided discrete fibrations:

\begin{definition}\label{def.commas}
	Motivated by \ref{rem.arrh-vert}, we will we write $$(\ev_1, \ev_0) : a^{[1]} \rightarrow a^{\times 2}$$ for the two-sided discrete fibration that classifies the identity proarrow $\id_a$, with $a \in \P$. Moreover, 
	for arrows $f : x \rightarrow a$ and $g : y \rightarrow a$,  we define the \textit{comma fibration} $$g \downarrow f \rightarrow x\times y$$ to be the two-sided discrete fibration that fits in the pullback square 
	\[
		\begin{tikzcd}
			g\downarrow f \arrow[r]\arrow[d] & a^{[1]}  \arrow[d, "{(\ev_1, \ev_0)}"] \\ 
			x \times y \arrow[r, "f \times g"] & a \times a,
		\end{tikzcd}
	\]
	i.e.\ it is the one that classifies the proarrow $\cart{g}{}{f}$.
\end{definition}

We can easily obtain a fibrational variant of the Yoneda lemma, cf.\ \cite[Corollary 16]{StreetFib} and \cite[Theorem 5.7.1]{RiehlVerity}:

\begin{proposition}\label{prop.fib-yoneda}
	Suppose that $f : x \rightarrow y$ is an arrow in $\P$. Then the canonical map of spans  
	\[
	\begin{tikzcd}
		x \arrow[rr]\arrow[dr, "{(\id, f)}"'] && \id_y \downarrow f \arrow[dl] \\
		& x\times y
	\end{tikzcd}
	\]
	obtained by factoring the 2-cell
	\[
		\begin{tikzcd}
			x \arrow[r, equal] \arrow[d,equal] & x \arrow[d, "f"] \\
			x \arrow[r, "{f}_\circledast"] & y
		\end{tikzcd}
	\]
	through the tabulating 2-cell of ${f}_\circledast$,
	induces an equivalence on mapping spaces
	$$
	\map_{\Vert(\P)^{(1)}/(x \times y)}(\id_y \downarrow f, e) \xrightarrow{\simeq} \map_{\Vert(\P)^{(1)}/(x \times y)}(x, e)
	$$
	for every two-sided discrete fibration $e \rightarrow x \times y$.
\end{proposition}
\begin{proof}
	This canonical map of spans is precisely the component of the unit
	of the adjunction $(L_{x,y}, \rho_{x,y})$ at the span $(\id, f) : x \rightarrow x \times y$,
	which is hence local with respect to the two-sided discrete fibrations.
\end{proof}

\begin{definition}\label{def.left-right-fib}
	Suppose that $\P$ is pointed, and let $x \in \P$. A two-sided discrete fibration from 
	$x$ to $\ast$ is called a \textit{left fibration}. Dually, a two-sided discrete fibration 
	from $\ast$ to $x$ is called a \textit{right fibration}.
\end{definition}

\begin{proposition}
	Suppose that $\P$ is cartesian closed and let $i \in \P$.
	 If $e \rightarrow x \times y$ is a two-sided discrete fibration 
	classifying $F$, then $[i, e] \rightarrow [i,x] \times [i,y]$ is the two-sided discrete fibration that 
	classifies $[i,F]$. In particular, we deduce that 
	$$
	[i,x^{[1]}] \simeq [i,x]^{[1]}.
	$$
\end{proposition}
\begin{proof}
	This follows directly from the fact that $[i,-]$ is a normal lax right adjoint, and hence preserves 
	tabulating 2-cells.
\end{proof}

Using the comma fibrations, we can make sense of (co)slice objects in suitable equipments:

\begin{definition}\label{def.slice-fibs}
	Suppose that we are in the case that $\P$ is pointed and cartesian closed. Associated to an arrow $f : i \rightarrow x$ in $\P$,
	we can consider the following familiar commas:
	\begin{enumerate}
		\item the slice object
		 $$x/f := \Delta \downarrow (\ast \xrightarrow{\{f\}} [i,x]),$$
		\item the coslice object
		$$f/x := (\ast \xrightarrow{\{f\}} [i,x]) \downarrow \Delta,$$
	\end{enumerate}
	where $\Delta : x \rightarrow [i,x]$ denotes the diagonal arrow. Both slice and coslice objects are canonically fibered over $x$ as a right and left fibration respectively.
\end{definition}

We will now demonstrate that particular weighted colimits (see \cite[Subsection 6.2]{EquipI}) can be computed as conical colimits (see \cite[Definition 6.30]{EquipI}). This recovers a 
known fact that was proven in the case that $\P = \CCAT_\infty$ by Haugseng \cite[Theorem 1.2]{HaugsengEnds}, cf.\ the discussion in \cite[Example 6.14]{EquipI}. 

\begin{proposition}\label{prop.fib-cones-slice}
	Suppose that $\P$ is pointed and cartesian closed. Let $f : i \rightarrow x$ be an arrow in $\P$ and 
	$W : \ast \rightarrow i$ be a weight corresponding to a right fibration
	$
	p : w \rightarrow i.
	$ Then the proarrow $W$-cones under $f$ exists 
	and it is classified by the left fibration $fp/x \rightarrow x$.
\end{proposition}
\begin{proof}
	In this case, the functor $$\rho W \circ (-): \Vert(\P)^{(1)}/x \rightarrow \Vert(\P)^{(1)}/(x\times i)$$
	admits a right adjoint ${}^{\rho W}{(-)}$ since $\Vert(\P)$ is cartesian closed. This is explained in \cite[Example 7.5]{FunDblCats}.
	The right adjoint carries a span $e \rightarrow x \times i$ to the left arrow in the pullback square
	\[
		\begin{tikzcd}
			e' \arrow[r] \arrow[d] & {[w, e]} \arrow[d]   \\
			x \arrow[r, "{\Delta \times \{p\}}"] & {[w, x]} \times {[w,i]} 
		\end{tikzcd}
	\]
	in $\Vert(\P)^{(1)}$ where $\Delta : x \rightarrow [w,x]$ denotes the diagonal. 
	In particular, it follows that if $e \rightarrow x\times i$ is a two-sided discrete fibration classifying a proarrow $F$, 
	then $e' \rightarrow x$ is a left fibration classifying $\cart{\{p\}}{[w, F]}{\Delta}$. It follows 
	from \ref{prop.hor-refl-strong} that the functor
	$$W \circ (-) : \Hor(\P)(x, \ast) \rightarrow \Hor(\P)(x,i)$$ has a right adjoint as well, 
	which is restricted from ${}^{\rho W}{(-)}$. 
	The proarrow of $W$-cones under $f$ thus exists and is given by $\cart{\{p\}}{[w, f^\circledast]}{\Delta} \simeq \cart{\{p\}}{[w, f]^\circledast}{\Delta} \simeq \cart{\{fp\}}{}{\Delta}$. This is precisely 
	the proarrow classified by $fp/x \rightarrow x$.
\end{proof}

\begin{corollary}\label{cor.fib-weighted-vs-conical}
	Suppose that $\P$ is pointed and cartesian closed. Let $f : i \rightarrow x$ be an arrow in $\P$ and 
	$W : \ast \rightarrow i$ be a weight corresponding to a right fibration
	$
	p : w \rightarrow i.
	$ Then the proarrow of $W$-cones under $f$ coincides with the proarrow of conical cones under  the composite $fp$, and both exist.
\end{corollary} 

\begin{corollary}\label{cor.fib-pke}
	Suppose that $\P$ is pointed and cartesian closed.
	Let $f : i \rightarrow x$ and $w : i \rightarrow j$ be arrows in $\P$. Suppose that $C$ is the proarrow of $w^\circledast$-cones under 
	$f$. Consider an element $\gamma :\ast \rightarrow j$ of $j$, then $\gamma^\circledast C$ is the proarrow of conical cones under the composite
	$$
	i/\gamma \xrightarrow{p} i \xrightarrow{f} x,
	$$
	where $p$ denotes canonical right fibration 	$$i/\gamma := w \downarrow \gamma  \rightarrow i,
	$$ or, equivalently, the one that sits in a pullback square 
	\[
		\begin{tikzcd}
			i/\gamma \arrow[r]\arrow[d,"p"'] & j/\gamma \arrow[d] \\
			i \arrow[r, "w"] & j.
		\end{tikzcd}
	\]
	In particular, if the left Kan extension $g$ of $f$ along $w$ exists, then it is given by
	$$
	g(\gamma) = \colim (i/\gamma \xrightarrow{p} i \xrightarrow{f} x).
	$$
\end{corollary}
\begin{proof}
	Note that ${\gamma^\circledast}{C}$ is the proarrow of cones under $f$ weighted by the proarrow $\cart{w}{}{\gamma}$. This 
	can be deduced from the defining property of the proarrow of cones $C$ and \ref{rem.lax-adj-hor}.  
	Now the weight  $\cart{w}{}{\gamma}$ is classified by the right fibration $p$ so that
	\ref{cor.fib-weighted-vs-conical} asserts that $\gamma^\circledast C$ is the proarrow of conical cones under the composite $fp$.
\end{proof}

\subsection{Comprehensive factorizations}\label{ssection.compr-fact}

Recall that every functor between $\infty$-categories can be factored  
into an initial functor followed by a left fibration.\footnote{See \cite[Subsection 4.1.1]{HTT} for instance: in the model of quasi-categories, such a factorization can be obtained by factoring a functor into a left anodyne followed by a left fibration.} %
For ordinary categories, 
this is called the \textit{comprehensive factorization} of a functor \cite{StreetWaltersFact}. 
Since the classes of initial functors 
and left fibrations are orthogonal to each other, these are part of an
 orthogonal factorization system (see \cite[Definition 5.2.8.8]{HTT} for a definition). 

Now, for fibrational and pointed $\infty$-equipments, we also have a notion 
 of left fibrations and initial arrows (see \cite[Definition 6.34]{EquipI}). Hence, 
 we can pose the question:\ under what conditions 
 does such a comprehensive factorization system exist in these more general equipments? We first establish the following:

\begin{proposition}\label{prop.ortho-initial-lfib}
	Initial maps are left orthogonal to left fibrations
	in a pointed and fibrational $\infty$-equipment.   
\end{proposition}
\begin{proof}
	Let $\P$ be a pointed and fibrational $\infty$-equipment.
	Suppose that $f : i \rightarrow j$ is an initial map, and that $p : e \rightarrow x$ is a left fibration 
	in a pointed and fibrational $\infty$-equipment $\P$. 
	Then we have to show that the commutative square 
	\[
		\begin{tikzcd}
			\map_{\Vert(\P)^{(1)}}(j, e) \arrow[r]\arrow[d] & \map_{\Vert(\P)^{(1)}}(i, e) \arrow[d] \\
			\map_{\Vert(\P)^{(1)}}(j, x) \arrow[r] & \map_{\Vert(\P)^{(1)}}(i, x)
		\end{tikzcd}
	\]
	of mapping spaces is a pullback square.
	To this end, it suffices to check that for any map $q : j \rightarrow x$, the induced map 
	$$
	\map_{\Vert(\P)^{(1)}}(j,e) \times_{\map_{\Vert(\P)^{(1)}}(j,x)} \{q\} \rightarrow \map_{\Vert(\P)^{(1)}}(i,e) \times_{\map_{\Vert(\P)^{(1)}}(i,x)} \{qf\}
	$$
	is an equivalence.
	But this is precisely the map 
	$$
	\map_{\Vert(\P)^{(1)}/x}(j,e) \rightarrow \map_{\Vert(\P)^{(1)}/x}(i,e)
	$$
	given by restriction along $f$ that is now viewed as a map fibered over $x$. 
	Thus we have to show that $f$ is local with respect to all left fibrations over $x$. 
	In other words, we have to check that the canonical 2-cell
	$$
	\cocart{(tf)}{}{(qf)} \rightarrow \cocart{t}{}{q}
	$$
	is an equivalence in $\Hor(\P)(x, \ast)$, where $t$ denotes the unique map $j \rightarrow \ast$. 
	This 2-cell is obtained by whiskering the canonical 2-cell $\cocart{(tf)}{}{f} \rightarrow {t}_\circledast$ with 
	$q^\circledast$, which is already an equivalence by the assumption that $f$ is initial.
\end{proof}

\begin{definition}
	We say that a pointed and fibrational $\infty$-equipment $\P$ has \textit{comprehensive factorizations} 
	if the left class of initial arrows and right class of left fibrations 
	form an orthogonal factorization system on $\Vert(\P)^{(1)}$.
\end{definition}

We can give a simpler description of those equipments that admit such factorizations:

\begin{proposition}\label{prop.compr-equipments}
	Suppose that $\P$ is a pointed and fibrational $\infty$-equipment, then the following assertions are equivalent:
	\begin{enumerate}
		\item $\P$ has comprehensive factorizations,
		\item the class of left fibrations is closed under compositions and retracts.
	\end{enumerate}
\end{proposition}

\begin{corollary}\label{cor.int-ccat-compr}
	The $\infty$-equipment $\CCAT_\infty(\E)$ admits comprehensive factorizations 
	for every $\infty$-topos $\E$.
\end{corollary}
\begin{proof}
	Reasoning similarly as in the proof of \ref{prop.int-ccat-fib}, we may use 
	the characterization of \ref{prop.compr-equipments} to reduce to the case that $\E = \S$, where it is already known to be true.
\end{proof}

\begin{lemma}\label{lem.compr-equipments}
	Suppose that $p : e \rightarrow x$ is a left fibration in a pointed 
	and fibrational $\infty$-equipment $\P$ with the property that for any left fibration $q : e' \rightarrow e$, the composite 
	$pq$ is again a left fibration. Then the following holds true:
	\begin{enumerate}
		\item there exists a commutative square
		\[
			\begin{tikzcd}
				\Hor(\P)(e, \ast) \arrow[r, "(-){p}^\circledast"]\arrow[d] & \Hor(\P)(x, \ast) \arrow[d] \\
				\Vert(\P)^{(1)}/e \arrow[r, "p_!"] & \Vert(\P)^{(1)}/x,
			\end{tikzcd}
		\]
		\item the functor 
		${(-)}{p}^\circledast : \Hor(\P)(e, \ast) \rightarrow \Hor(\P)(x,\ast)$ 
		is conservative,
		\item for any commutative triangle 
		\[
			\begin{tikzcd}
				a \arrow[dr] \arrow[rr, "f"] && e\arrow[dl, "p"] \\ 
				& x
			\end{tikzcd}
		\]
		of arrows in $\P$, the map $f$ is initial if and only if $f$, viewed as a map in $\Vert(\P)^{(1)}/x$, is an equivalence local to all the left fibrations over $x$.
	\end{enumerate}
\end{lemma}
\begin{proof}
	Throughout this proof, let us write $\rho : \P \rightarrow \SSPAN(\Vert(\P)^{(1)})$ for the span representation of $\P$.
	Note that $p_!$ is conservative, so that (2) follows from (1). In turn, (3) follows from (2). To wit, the map 
	$f$ is initial if and only if the canonical 2-cell $\cocart{(tf)}{}{f} \rightarrow t_\circledast$ is an equivalence, where 
	$t : e \rightarrow \ast$ denotes the terminal arrow. In light of (3), this is true if and only if the induced 
	2-cell $\cocart{(tf)}{}{(pf)} \rightarrow \cocart{t}{}{p}$ is an equivalence. But this is precisely the condition for 
	$f$ to be a $\rho$-local 2-cell in $\Vert(\P)/x$. 
	
	Let us prove (1).
	A priori, the functor $\rho$ only provides  a \textit{lax} commutative square
	\[
		\begin{tikzcd}
			\Hor(\P)(e, \ast) \arrow[r, "{(-)}p^\circledast"]\arrow[d] & |[alias=f]|\Hor(\P)(x, \ast) \arrow[d] \\
			|[alias=t]|\Vert(\P)^{(1)}/e \arrow[r, "(p \downarrow \id_x)^*"'] & \Vert(\P)^{(1)}/x,
			\arrow[from=f, to=t, Rightarrow]
		\end{tikzcd}
	\]
	where $p \downarrow \id_x = \rho(p^\circledast)$. This means that for any proarrow $F : e \rightarrow \ast$ 
	we have a natural $\rho$-local 2-cell
	$$
	\alpha_F : \rho F \circ (p \downarrow \id_x) \rightarrow \rho ({F}{p}^\circledast)
	$$
	in $\Vert(\P)^{(1)}/x$. Recall that we have a canonical map 
	$
	(p, \id_e) \rightarrow (p \downarrow \id_x),
	$
	see also \ref{prop.fib-yoneda}. Precomposing with $\rho F$ yields a comparison map 
	$$
	\beta_F : p_!\rho F = \rho F \circ (p,\id_e)\rightarrow \rho F \circ (p \downarrow \id_x)
	$$
	that is natural in $F$. We claim that this map is a $\rho$-local 2-cell. Note that this 
	would finish the proof: then the composite $\beta_F\alpha_F : p_!\rho F \rightarrow \rho (Fp^\circledast)$ 
	will be an equivalence since both domain (by assumption) and codomain are left fibrations.

	Let us verify the remaining claim. One readily checks that $\beta_F$ is precisely 
	the pushforward of the canonical $\rho $-local 2-cell 
	$$
	(p \circ \rho F, \id_{e'}) \rightarrow (p \circ \rho F) \downarrow \id_x
	$$
	over $x \times e'$ along the projection $q : x \times e' \rightarrow x$. Since 
	the right adjoint $q^*$ 
	preserves two-sided discrete fibrations, the left adjoint $q_!$ preserves $\rho $-local equivalences 
	and hence $\beta_F$ is a $\rho $-local equivalence. 
\end{proof}

\begin{proof}[Proof of \ref{prop.compr-equipments}]
	It is clear that (2) implies (1), see also \cite[Corollary 5.2.8.13]{HTT}. Conversely, suppose that (2) holds. It is readily 
	verified that initial arrows are closed under retracts in any pointed $\infty$-equipment. Hence, it remains to check 
	that every map $g : a \rightarrow x$
	can be factored into an initial arrow followed by a left fibration. Since $\P$ is fibrational, we can find a commutative diagram 
	\[
			\begin{tikzcd}
				a \arrow[dr,"g"'] \arrow[rr, "f"] && e\arrow[dl, "p"] \\ 
				& x
			\end{tikzcd}
		\]
	of arrows in $\P$ so that $p$ is a left fibration and $f$ is an equivalence local to all left fibrations over $x$. Hence $f$ is initial by \ref{lem.compr-equipments} and 
	we obtain the desired factorization $g \simeq pf$.
\end{proof}

These are direct consequences of \cite[Proposition 5.2.8.11]{HTT} and \cite[Proposition 5.8.6]{HTT}:

\begin{corollary}\label{cor.fib-cancellation}
	Suppose that $\P$ is a pointed and fibrational $\infty$-equipment with comprehensive factorizations, then:
	\begin{enumerate}
		\item the class of left fibrations satisfy a cancellation property: if $p$ is a left fibration, then $pf$ is a left fibration if and only if $f$ is a left fibration,
		\item the class of initial arrows is closed under cobase change.
	\end{enumerate}
\end{corollary}

\begin{proposition}\label{prop.compr-conjoint-slice}
	Suppose that $\P$ is a fibrational and cartesian closed $\infty$-equipment with comprehensive factorizations. Let $\xi : \ast \rightarrow x$ 
	be an element of an object $x$ of $\P$. 
	For a left fibration $p : e\rightarrow x$, the following assertions are equivalent:
	\begin{enumerate}
		\item $p$ classifies the conjoint ${\xi}^\circledast : x \rightarrow \ast$,
		\item there exists an equivalence $e \simeq \xi/x$ above $x$,
		\item $\xi$ factors as 
		$$
		\ast \rightarrow e \rightarrow x
		$$
		so that the first map is initial. 
	\end{enumerate}
\end{proposition}
\begin{proof}
	It is immediate that (1) and (2) are equivalent. Let us show that (2) and (3) are equivalent. As reasoned in \ref{prop.fib-yoneda}, 
	the map 
	$$
	\begin{tikzcd}
		\ast \arrow[rr]\arrow[dr,"\xi"'] && \xi/x \arrow[dl] \\
		& x 
	\end{tikzcd}
	$$
	is local with respect to left fibrations over $x$. Hence, (2) is equivalent to providing a triangle  
	$$
	\begin{tikzcd}
		\ast \arrow[rr]\arrow[dr,"\xi"'] && e \arrow[dl, "p"] \\
		& x 
	\end{tikzcd}
	$$
	so that the top map is local with respect to left fibrations over $x$. In light of \ref{lem.compr-equipments}, 
	this map is local if and only if it is initial.
\end{proof}

\begin{corollary}\label{cor.colim-initial-slice}
Suppose that $\P$ is a fibrational and cartesian closed $\infty$-equipment with comprehensive factorizations. Let 
 $\xi : \ast \rightarrow x$ be an object of $x \in \P$. Then the following assertions are equivalent:
\begin{enumerate}
	\item $\xi$ is the colimit of $f$,
	\item there exists an equivalence $f/x \simeq \xi/x$  between  left fibrations above $x$,
	\item $\xi$ factors as
	$$
	\ast \rightarrow f/x \rightarrow x 
	$$
	so that the first arrow is initial.
\end{enumerate}
\end{corollary}
\begin{proof}
	This follows from \ref{prop.compr-conjoint-slice} and \ref{prop.fib-cones-slice}.
\end{proof}

\section{Internal category theory via \texorpdfstring{$\infty$}{∞}-equipments}\label{section.int-cat-theory}

In this final section, we will fix an $\infty$-topos $\E$ and investigate the formal category theory 
associated with the equipment 
$$
\CCAT_\infty(\E)
$$
that was constructed in \cite{EquipI}. Note that this equipment has the following properties:
\begin{itemize}
	\item first of all, it is pointed,
	\item it is cartesian closed (\ref{ex.int-ccat-cart-closed}),
	\item it is horizontally closed (\ref{cor.int-ccat-hor-closed}),
	\item it is fibrational (\ref{prop.int-ccat-fib}),
	\item it has comprehensive factorizations (\ref{cor.int-ccat-compr}).
\end{itemize}

One may find a concise overview of the aspects of internal category theory that we will cover here in the introduction (\ref{ssection.intro-int-ccat}).
We stress again that most of these aspects have also been studied previously by Martini and Wolf in \cite{Martini} and \cite{MartiniWolf}, but using a completely different approach.

\begin{notation}
	For a category $\C$ internal to $\E$, we will write $\fun_\E(\C, -)$ for the normal lax right adjoint to $\C \times (-)$. In particular,
for another category $\D$ internal to $\E$, $\fun_\E(\C,\D)$ is the internal hom object in the $\infty$-category $\Cat_\infty(\E)$. We will suppress the notation of $\E$ if it is clear from the context.
\end{notation}

\begin{notation}
	For an object $E \in \E$, we will write $\pi_E : \E \rightarrow \E/E$ for the \'etale geometric morphism. We have associated adjunctions
	\[ 
		\begin{tikzcd}[column sep = large]
			\CCAT_\infty(\E/E) \arrow[rr, bend left = 50pt, "{\pi_{E,!}}"name=t]\arrow[rr, bend right = 50pt, "{\pi_{E,*}}"name=b] && \arrow[ll, "\pi_E^* = E \times(-)"'name=m] \CCAT_\infty(\E), 
			\arrow[from=t,to=m, "\bot", phantom]
			\arrow[from=m,to=b, "\bot", phantom]
		\end{tikzcd}
	\]
	where the top two functors are strict, and the bottom functor is normal and lax (see \ref{ex.dbl-adj-geom-morph}).
\end{notation}

We have the following push-pull formula on the level of equipments:

\begin{proposition}\label{prop.int-ccat-push-pull}
	For $E \in \E$, there exists a commutative diagram of functors
	\[
		\begin{tikzcd}
			\CCAT_\infty(\E) \times \CCAT_\infty(\E/E) \arrow[r, "\id \times \pi_{E,!}"]\arrow[d, "\pi_{E,*} \times \id"'] & \CCAT_\infty(\E) \times \CCAT_\infty(\E) \arrow[d, "-\times -"] \\
			\CCAT_\infty(\E/E) \times \CCAT_\infty(\E/E) \arrow[r,"\pi_{E,!}(-\times-)"] & \CCAT_\infty(\E)
		\end{tikzcd}
	\]
	In particular, for a category $\C$ internal to $\E$, we have an equivalence 
	$$
	E \times \fun_{\E}(\C, -) \simeq \fun_{\E/E}(E\times \C, E\times(-))
	$$
	between lax functors $\CCAT_\infty(\E) \rightarrow \CCAT_\infty(\E/E)$.
\end{proposition}
\begin{proof}
	The first assertion follows from the fact that there exists similar commutative diagram of $\infty$-toposes before applying $\CCAT_\infty(-)$. This is a special case of the push-pull formula for \'etale morphisms \cite[Remark 6.3.5.12]{HTT}. 
	The second assertion follows from unicity of right adjoints in the $(\infty,2)$-category $\mathrm{DBLCAT}^\lax_\infty$.
\end{proof}

\subsection{Generalized objects}\label{ssection.int-cat-gen-obj} As we have already observed in \cite[Example 6.38]{EquipI}, the $\infty$-equipment of internal categories 
is virtually never strongly pointed. Whenever an equipment was strongly pointed, we had an efficient way of testing whether proarrows are companions or conjoints.
We commence our study of internal category theory by giving a generalization of this principle.

To this end, we note that we have an adjunction 
$$
\mathrm{incl} : \E \rightleftarrows \Cat_\infty(\E) : (-)^{\simeq}
$$
induced by the adjunction of categories
$$
\Delta \rightleftarrows \{[0]\}.
$$
Note that this adjunction is natural in the $\infty$-topos $\E$.
The right adjoint $(-)^{\simeq}$ assigns to each internal category its \textit{core}. The left adjoint 
is fully faithful and the objects in the essential image of this inclusion are called \textit{internal groupoids}. We will leave this inclusion 
implicit from now on, and view every object of $\E$ as an internal groupoid via this functor.

\begin{remark}\label{rem.int-groupoids}
	It can be verified that a category $\C$ internal to $\E$ is an internal groupoid if and only if it 
	is a groupoid representably (cf.\ \cite[Definition 1.7]{StreetFib2}): i.e.\ for every category $X$ internal to $\E$, the mapping $\infty$-category 
	$$
	\CAT_\infty(\E)(X,\C)
	$$
	is a space. 
\end{remark}

\begin{definition}
	Let $\C$ be a category internal to $\E$. A \textit{generalized object} 
	of $\C$ is a functor $x : E \rightarrow \C$ so that $E \in \E$, i.e.\ $E$ is an internal groupoid.
	If $E = \ast$, then $x$ is simply called an \textit{object} of $\C$. 
	
	We say that a set $\mathcal{U}$ of generalized objects of $\C$ 
	 form a \textit{groupoidal cover} of $\C$ if the following map 
	$$
	\textstyle \coprod_{E \rightarrow \C \in \mathcal{U}} E \rightarrow \C^\simeq
	$$
	in $\E$, that is induced by the adjunct morphisms, is an effective epimorphism.
\end{definition}

\begin{remark}\label{rem.gen-objects}
	By adjunction, a generalized object $x : E = \pi_{E,!}(\ast) \rightarrow \C$ 
	of an internal category $\C$ corresponds to an object of the category $E \times \C$ internal to $\E/E$. We will (with abuse of notation) denote this latter adjunct object again by $$x : E \rightarrow E \times \C.$$
\end{remark}

\begin{remark}
	Let $f : \E \rightarrow \F$ be a geometric morphism. Then 
	the inverse image $f^* : \F \rightarrow \E$ preserves effective epimorphisms.
	Consequently, if $\mathcal{U}$ is a groupoidal cover for 
	a category $\C$ internal to $\F$, then $f^*\mathcal{U} := \{f^*x \mid x  \in \mathcal{U}\}$ is again a groupoidal cover for $f^*\C$.
\end{remark}

We can now formulate the following recognition theorem for companions and conjoints in the $\infty$-equipment 
of categories internal to $\E$:

\begin{theorem}\label{thm.comp-conj-topos}
	Let  $\mathcal{U}$ be a groupoidal cover for a category $\C$ internal to $\E$. Suppose 
	that $F : \C \rightarrow \D$ is a proarrow between internal categories. Then the following assertions are equivalent 
	\begin{enumerate}
		\item $F$ is a companion,
		\item for every generalized object $x:E \rightarrow \C$ in $\mathcal{U}$, the restricted proarrow 
		$$
		{(E \times F)}{x_\circledast} : E \rightarrow E \times \D
		$$
		in $\CCAT_\infty(\E)/E \simeq \CCAT_\infty(\E/E)$ is a companion.
	\end{enumerate}
	\noindent Dually, 
	the following assertions are equivalent for a proarrow $G : \D \rightarrow \C$ between internal categories:
	\begin{enumerate}
		\item $G$ is a conjoint,
		\item for every generalized object $x:E \rightarrow \C$ in $\mathcal{U}$, the restricted proarrow 
		$$
		x^\circledast{(E \times G)} : E \times \D \rightarrow E
		$$
		in $\CCAT_\infty(\E)/E \simeq \CCAT_\infty(\E/E)$ is a conjoint.
	\end{enumerate}
\end{theorem}

To prove the above theorem, we will show that it holds for presheaf $\infty$-toposes via double categorical methods using 
the theory developed in \cite{FunDblCats}. 
We then bootstrap this result to general $\E$. The bootstrapping 
is achieved through three subsequent observations which we will present as lemmas:

\begin{lemma}\label{lem.comp-conj-topos-1}
	Suppose that $f : E \rightarrow X$ is an effective epimorphism in $\E$. Then the pullback functor 
	$$
	f^* : \CCAT_\infty(\E/X) \rightarrow \CCAT_\infty(\E/E)
	$$
	reflects companions and conjoints.
\end{lemma}
\begin{proof}
	Without loss of generality, we may assume that $X = \ast$, the terminal object of $\E$. By assumption, we have that $\ast$ is given by the 
	\v{C}ech nerve of $f$, 
	$$
	\ast \simeq \colim_{[n]\in \Delta^\op} E^{\times (n+1)}.
	$$ 
	Since $\E$ is an $\infty$-topos, the functor $\E/(-) : \E^\op \rightarrow \widehat{\Cat}_\infty$ 
	preserves small limits. Thus we obtain an equivalence
	$$
	\E \simeq \lim_{[n] \in \Delta} \E/E^{\times (n+1)}.
	$$
	This gives rise to the following limit description 
	$$
	\CCAT_\infty(\E) \simeq \lim_{[n] \in \Delta} \CCAT_\infty(\E/E^{\times (n+1)})
	$$
	on the level of $\infty$-equipments.
	In light of \cite[Hypothesis 3.10]{EquipI} (cf.\ \cite[Proposition 4.27]{FunDblCats}), this entails that a proarrow $F$ in $\CCAT_\infty(\E)$ is a companion (resp.\ conjoint) 
	if for any $n \geq 0$, the proarrow $E^{\times (n+1)} \times F$ is a companion (resp.\ conjoint) in $\CCAT_\infty(\E/E^{\times (n+1)})$. This, in turn, can be reduced to $n=1$, since the terminal map 
	$E^{\times (n+1)} \rightarrow \ast$ factors over $E \rightarrow \ast$ via one of the projections.
\end{proof}

\begin{lemma}\label{lem.comp-conj-topos-2}
	Suppose that $F : \C \rightarrow \D$ is a proarrow between internal categories. If 
	there exists a groupoidal cover $\mathcal{U}$ of $\C$ 
	so that condition (2) of \ref{thm.comp-conj-topos} is met, then this condition holds with 
	respect to any groupoidal cover $\mathcal{V}$ of $\C$.
\end{lemma}
\begin{proof}
	We just handle the case for companions since the dual case for conjoints is similar.
	Without loss of generality, we may assume that $\mathcal{U}$ and $\mathcal{V}$ both consist 
	of just one generalized object $x : E \rightarrow \C$ and $y : E' \rightarrow \C$ respectively. 
	We can now define an internal groupoid $E'' \in \E$ by the pullback square 
	\[
		\begin{tikzcd}
			E'' \arrow[r,"p"]\arrow[d,"q"'] & E \arrow[d,"x"] \\
			E' \arrow[r, "y"] & \C.
		\end{tikzcd}
	\]
	Since effective epimorphisms are closed under base change, 
	$p$ and $q$ are effective epimorphisms. Now, one can readily compute that 
	$$
	p^*({(E \times F)}x_\circledast) \simeq {(E'' \times F)}(xp)_\circledast \simeq q^*((E' \times F)y_\circledast)
	$$
	Thus the result follows from \ref{lem.comp-conj-topos-1}.
\end{proof}

\begin{lemma}\label{lem.comp-conj-topos-3}
	Suppose that $i : \E \rightarrow \F$ is a geometric embedding between $\infty$-toposes. 
	If \ref{thm.comp-conj-topos} holds for $\F$, then 
	it holds for $\E$ as well.
\end{lemma}
\begin{proof}
	We just handle the case for companions since the dual case for conjoints is similar.
	Suppose that $F : \C \rightarrow \D$ is a proarrow between categories internal to $\E$ so that 
	condition (2) of \ref{thm.comp-conj-topos} is met. In light of \ref{lem.comp-conj-topos-2}, this condition 
	does not depend on the chosen cover $\mathcal{U}$. In particular, we may assume that $\mathcal{U} = \{x : \C^{\simeq} \rightarrow \C\}$. 
	We must check that $F$ is a companion. Since $F \simeq i^*i_*F$, it will be necessary and sufficient to check that 
	$i_*F$ is a companion in $\CCAT_\infty(\F)$. 

	 Note that $i$ induces 
	a geometric morphism on slices that fits in a commutative diagram of 
	geometric morphisms 
	\[
		\begin{tikzcd}
			\E \arrow[r, "i"]\arrow[d] & \F \arrow[d] \\
			\E/\C^{\simeq} \arrow[r, "i/\C^{\simeq}"] & \F/i_*\C^{\simeq}.
		\end{tikzcd}
	\] 
	If we apply the normal lax functor $(i/\C^{\simeq})_*$ to the companion proarrow
	${(\C^{\simeq} \times F)}x_\circledast$, 
	then \ref{prop.lax-pres-cart-cells} asserts that we recover
	$$
	{(i_*\C^{\simeq} \times i_*F)}{(i_*x)_\circledast} : i_*\C^\simeq \rightarrow i_*\C^{\simeq} \times i_*\D.
	$$
	Thus this proarrow is again a companion by \ref{cor.normal-lax-fun-eq-pres-compconj}. Since $i_*x$ is precisely the canonical cover $(i_*\C)^\simeq \rightarrow i_*\C$, the assumption on $\F$ implies 
	that $i_*F$ is a companion.
\end{proof}

\begin{proof}[Proof of \ref{thm.comp-conj-topos}]
	On account of \ref{lem.comp-conj-topos-3}, we may reduce to the case that $\E = \PSh(T)$ for some small $\infty$-category $T$. 
	Let us show the only non-trivial assertion that (2) implies (1) for the case of companions. 
	Note that we can identify 
	$\CCAT_\infty(\E)$ with the cotensor $[T^\op, \CCAT_\infty]$.  In light of \cite[Theorem 4.6]{FunDblCats}, we have to check 
	that for each arrow $\phi : s \rightarrow t$ of $T$, the 2-cell 
	\[
		\begin{tikzcd}
			\C(t)\arrow[r, "{F_t}"]\arrow[d, "\phi^*"'] & \D(t) \arrow[d, "\phi^*"] \\
			\C(s) \arrow[r, "{F_s}"] & \D(s)
		\end{tikzcd}
	\]
	in $\CCAT_\infty$ is \textit{companionable}. That entails checking the following conditions:
	\begin{enumerate}[label=(\roman*)]
		\item the proarrow $F_t$ is the companion of a functor $f_t : \C(t) \rightarrow \D(t)$ for every $t \in T$, 
		\item and moreover, the induced natural transformation $\phi^*f_t \rightarrow f_s\phi^*$  by the pasting
		\[
		\begin{tikzcd}
			\C(t) \arrow[d,equal] \arrow[r,equal] & \C(t)\arrow[d,"f_t"] \\
			\C(t)\arrow[r, "{F_t}"]\arrow[d, "\phi^*"'] & \D(t) \arrow[d, "\phi^*"] \\
			\C(s) \arrow[r, "{F_s}"]\arrow[d,"f_s"'] & \D(s)\arrow[d,equal] \\
			\D(s) \arrow[r,equal] & \D(s)
		\end{tikzcd}
		\] is an equivalence. Here the top and bottom 2-cell are the companionship unit and counit respectively.
	\end{enumerate}

	By the Yoneda lemma, any object $x \in \C(t)$ corresponds to a map 
	$$
	 y_t := \map_T(-, t) \rightarrow \C,
	$$
	which we will again denote by $x$. This gives rise to a groupoidal cover 
	$$
	\{x : y_t \rightarrow \C \mid t \in T, x \in \C(t)\}
	$$
	of $\C$.
	By what we have proven in \ref{lem.comp-conj-topos-2}, we may assume that this is the cover $\mathcal{U}$. Let $x : y_t \rightarrow \C \in \mathcal{U}$. Note that the 
	\'etale morphism $\pi_{y_t} : \PSh(T) \rightarrow \PSh(T/t)$ is induced by the projection $\pi_t : (T/t)^\op \rightarrow T^\op$. Consequently, under the identifications of \ref{ex.cotensors-ccat},
	the functor $\pi_{y_t}^* = y_t \times (-)$ is given by the functor
	$$\pi_t^* : [T^\op, \CCAT_\infty] \rightarrow [(T/t)^\op, \CCAT_\infty].$$
	So, by assumption ${\pi_t^*F}{x_\circledast}$ is a companion in $[(T/t)^\op, \CCAT_\infty]$, where we now view $x$ as a functor $\ast \rightarrow \pi_t^*\C$ indexed by $T/t$ (cf.\ \ref{rem.gen-objects}). 
	Consequently, for any arrow $\phi : s \rightarrow t$, the composite 2-cell 
	\[
		\begin{tikzcd}
			\ast \arrow[d, equal]\arrow[r, "x_\circledast"]& \C(t)\arrow[r, "F_t"]\arrow[d, "\phi^*"] & \D(t) \arrow[d, "\phi^*"] \\
			\ast \arrow[r, "(\phi^*x)_\circledast"] & \C(s)  \arrow[r, "F_s"] & \D(s)
		\end{tikzcd}
	\]
	is companionable in $\CCAT_\infty$.

	We can now conclude the proof as follows. In particular, the above implies that $(F_t)x_\circledast$ is a companion for every $t\in T$ and $x \in \C(t)$, so that (i) holds
	by the Yoneda lemma, see \cite[Example 6.38]{EquipI}. Moreover, from the companionability of the above composite 2-cells, it can be deduced that the component 
	$$
	\phi^*f_t(x) \rightarrow f_s\phi^*(x)
	$$ of the natural transformation in (ii) at $x \in \C(t)$ 
	is an equivalence. 
\end{proof}

\subsection{Internal fibrations and colimits} Internal $\infty$-categories 
support a notion of two-sided discrete, left and right fibrations on account of \ref{section.fib-equipments}. In this instance, this a representable notion:

\begin{proposition}\label{prop.int-ccat-tsdfib-rep}
	Let $(p,q) : A \rightarrow \C \times \D$ be a span of categories internal to $\E$. Then the following assertions are equivalent:
	\begin{enumerate}
		\item the span $(p,q)$ is a two-sided discrete fibration of categories internal to $\E$,
		\item for every category $X$ internal to $\E$, the span 
		$$
		(p,q)_* : \CAT_\infty(\E)(X,A) \rightarrow \CAT_\infty(X,\C) \times \CAT_\infty(X,\D)
		$$
		is a two-sided discrete fibration of $\infty$-categories,
		\item condition (2) holds for all groupoids $X \in \E$.
	\end{enumerate}
\end{proposition}
\begin{proof}
	It directly follows from \ref{prop.tsdfib-rep} that (1) implies (2). It is also clear that (2) implies (3).  Thus it remains to show that (3) implies (1). To this end, 
	we first observe the following. Suppose that $i : \E \rightarrow \F$ is a geometric embedding.
	Then for every $Y \in \F$, we have a commutative diagram 
	\[
		\begin{tikzcd}
			\CAT_\infty(\E)(i^*Y, A)\arrow[d, "\simeq"']  \arrow[r, "{(p,q)_*}"] & \CAT_\infty(\E)(i^*Y, \C) \times \CAT_\infty(\E)(i^*Y, \D)\arrow[d, "\simeq"] \\
			\CAT_\infty(\F)(Y, i_*A) \arrow[r, "{(i_*p, i_*q)_*}"] & \CAT_\infty(\F)(Y, i_*\C) \times \CAT_\infty(\F)(Y, i_*\D).
		\end{tikzcd}
	\]
	Thus if (3) holds for $(p,q)$, then (3) holds for $(i_*p,i_*q)$ as well. Taking $Y = i_*X$ with $X \in \E$, we see that the converse implication is true as well. 
	Moreover, since $i_*$ reflects tabulations (see \ref{prop.tabs-closure}), (1) holds for $(p,q)$ if and only if it holds for $(i_*p, i_*q)$.

	Now, it follows from the above that it is necessary and sufficient to show that (3) implies (1) in the prototypical case that $\E = \PSh(T)$ for some small $\infty$-category $T$. Then 
	$\CCAT_\infty(\E) \simeq [T^\op, \CCAT_\infty]$. We have to check that the cocartesian 2-cell $\id_A \rightarrow \cocart{q}{}{p}$ is tabulating 
	and this can be checked point-wise on account of \ref{prop.tabs-closure}. Consequently, we deduce that $(p,q)$ is a two-sided discrete fibration if and only if 
	$(p(t), q(t)) : A(t) \rightarrow \C(t) \times \D(t)$ is a two-sided discrete fibration. But this can be identified with the post-composition functor $(p,q)_*$ at the internal groupoid $X := \map_T(-,t)$.
\end{proof}

We also have the following generalization of \cite[Proposition 4.4.11]{Cisinski}, see also \cite[Proposition 4.4.7]{Martini}:

\begin{proposition}\label{prop.int-ccat-lfib-proper}
	Left fibrations are proper and right fibrations are smooth in $\CCAT_\infty(\E)$. 
\end{proposition}
\begin{proof}
	We will just handle the case of left fibrations; the other case is similar.
	Let $f : X \rightarrow \C$ be a left fibration in $\CCAT_\infty(\E)$. Since left fibrations are closed under pullbacks, it is necessary and sufficient 
	to check that every pullback square
	\[
		\begin{tikzcd}
			A \arrow[d] \arrow[r] & X \arrow[d] \\ 
			B \arrow[r] & \C
		\end{tikzcd}
	\]
	of categories internal to $\E$ is exact. We may choose a presentation $i : \E \rightarrow \PSh(T)$ so that \ref{prop.refl-incl-formal-cat-theory} asserts that it suffices to check that 
	the pullback square
	\[
		\begin{tikzcd}
			i_*A \arrow[d] \arrow[r] & i_*X \arrow[d] \\ 
			i_*B \arrow[r] & i_*\C
		\end{tikzcd}
	\]
	is exact in $[T^\op, \CCAT_\infty]$. But this may be checked level-wise at objects $t\in T$ on account of \ref{prop.cart-cotensors}, see also the argument in 
	\ref{ex.proper-point-wise}. Thus 
	the desired result follows
	from the fact that left fibrations of $\infty$-categories are proper, see \cite[Proposition 4.4.11]{Cisinski} or \cite[Proposition 4.1.2.15]{HTT}. 
\end{proof}

Note that for an internal category, the comma two-sided discrete 
fibration $\C^{[1]} \rightarrow \C^{\times 2}$ of \ref{def.commas} is computed as 
$$
(\ev_1, \ev_0) : \fun([1], \C) \rightarrow \C^{\times 2}.
$$
Consequently, the coslice $f/\C$ associated with a functor $f : I \rightarrow \C$ 
sits in a pullback square 
\[
	\begin{tikzcd}
		f/\C \arrow[d]\arrow[r] & \fun([1] \times I, \C) \arrow[d, "{(\ev_1, \ev_0)}"] \\
		\C \arrow[r] & \fun(I, \C) \times \fun(I,\C),
	\end{tikzcd}
\]
and the left arrow is the associated left fibration. We can now interpret \ref{cor.colim-initial-slice}:

\begin{corollary}\label{cor.int-ccat-weighted-rfib}
	Let $f : I \rightarrow \C$ be an internal functor. Suppose that $W : \ast \rightarrow I$ is a weight classified by a 
	right fibration $p : E \rightarrow I$. Let $x$ be an object of $\C$. Then the following assertions are equivalent: 
	\begin{enumerate}
			\item $x$ is the $W$-weighted colimit of $f$,
			\item there exists an equivalence $fp/\C \simeq x/\C$ over $\C$,
			\item there exists an initial object of $fp/\C$ that lies over $x$. 
	\end{enumerate}
\end{corollary}

In case that $W$ corresponds to the coconical weight \cite[Definition 6.30]{EquipI} classified by $p = \id_I$, this was used as the definition of (conical)
colimits by  Martini and Wolf \cite{MartiniWolf}. Consequently, we obtain:

\begin{corollary}\label{cor.int-ccat-colim-mw}
	The notion of colimits in $\CCAT_\infty(\E)$ coincides with the notion developed by Martini and Wolf.
\end{corollary}

We may also characterize initial and final functors using an internal version of Quillen's theorem A (cf.\ \cite[Proposition 6.39]{EquipI}). This result was obtained by Martini as well \cite[Corollary 4.4.8]{Martini}, using a different method.
To formulate this, we need the fact that the inclusion
$$
\mathrm{incl} : \E \rightarrow \Cat_\infty(\E)
$$
admits a left adjoint as well that is natural in the $\infty$-topos $\E$. The left adjoint is described by 
$$
(-)[(-)^{-1}] : \Cat_\infty(\E) \rightarrow \E : \C \mapsto \colim_{[n] \in \Delta^\op} \C_n,
$$
and was discussed in the text following Corollary 3.2.12 of \cite{Martini}.
The image $\C[\C^{-1}]$ of $\C \in \Cat_\infty(\E)$ is usually called the \textit{classifying groupoid} of $\C$.
	
This adjunction can also be described as follows. 
Note that by \ref{rem.int-groupoids} and \ref{prop.int-ccat-tsdfib-rep}, the reflective inclusion
$$
\rho_{\ast, \ast} : \Hor(\CCAT_\infty(\E))(\ast, \ast) \rightarrow \Cat_\infty(\E)
$$
provided by the span representation
is an equivalence onto the internal groupoids. Hence, we obtain a factorization 
\[
	\begin{tikzcd}
		\Hor(\CCAT_\infty(\E))(\ast, \ast) \arrow[r, "\rho_{\ast,\ast}"]\arrow[d, dotted, "\simeq"'] & \Cat_\infty(\E). \\
		\E \arrow[ur, "\mathrm{incl}"']
	\end{tikzcd}
\]
Recall that $\rho_{\ast, \ast}$ has a left adjoint $L_{\ast, \ast}$
with a concrete description (see \ref{thm.span-rep}) that allows us to conclude that the terminal 
map $\C[\C^{-1}] \rightarrow \ast$ is precisely the two-sided discrete fibration classifying $\cocart{t}{}{t} = L_{\ast,\ast}(t)$ 
where
$t : \C \rightarrow \ast$ denotes the terminal arrow. 

\begin{proposition}[Quillen's theorem A internally]\label{prop.int-ccat-quilen-A}
	Let $J$ be an internal category with groupoidal cover $\mathcal{U}$.
	Then a functor $f : I \rightarrow J$ between internal categories is final if and only if for every generalized 
	object $x : E \rightarrow J$ in $\mathcal{U}$, we have
	$$
	x/(E \times I)[(x/(E \times I))^{-1}] \simeq  E
	$$
	in $\E/E$.
\end{proposition}
\begin{proof}
	By definition (see \cite[Definition 6.34]{EquipI}), $f$ is final if and only if the proarrow $\cocart{f}{}{t} : \ast \rightarrow J$ is a conjoint where 
	$t : I \rightarrow \ast$ denotes the terminal arrow.
	On account of \ref{thm.comp-conj-topos}, this is equivalent to asking that the proarrow
	$$
	G := x^\circledast{(E \times \cocart{f}{}{t})} \simeq x^\circledast(E\times f)_\circledast(E\times t)^\circledast: E \rightarrow E 
	$$
	is a conjoint in $\CCAT_\infty(\E/E)$ or every $x : E \rightarrow J$ in $\mathcal{U}$. 
	Note that $\cart{x}{}{(E\times f)}$ is classified by the left fibration $x/(E\times I) \rightarrow  E\times I$ internal to $\E/E$. 
	In the proof of \ref{lem.compr-equipments}, we have shown that there is a map from the terminal arrow
	$$
	\tau : x/(E\times I) \rightarrow E \times I \xrightarrow{E\times t} E
	$$
	to the left fibration that classifies $G$, which is local with respect to all left fibrations.
	Hence, we deduce from the discussion preceding the proposition that $G$ is classified by the classifying 
	groupoid of $x/(E \times I) \in \Cat_\infty(\E/E)$. It is a conjoint precisely if it is the conjoint of the unique functor $E \rightarrow E$, which is classified by $E$.
\end{proof}

\subsection{Internal Kan extensions} As a concluding application of the developed theory of equipments, we will study point-wise Kan extensions \cite[Section 6.3]{EquipI}
internal to $\CCAT_\infty(\E)$.  Using our methods, we are able to give a quick proof of the following result for the existence and computation of left Kan extensions internal to $\E$, which was 
earlier established by Shah \cite{Shah} in the case that $\E$ is a presheaf topos:

\begin{proposition}\label{prop.internal-pke}
	Let $f : I \rightarrow \C$ and $w : I \rightarrow J$ be internal functors.
	Suppose that $\mathcal{U}$ is a groupoidal cover for $J$.
	Then the following assertions are equivalent:
	\begin{enumerate}
		\item the left Kan extension of $f$ along $w$ exists,
		\item for each generalized element $x : E \rightarrow J$ in $\mathcal{U}$, 
		the colimit of the functor 
		$$
		(E\times I)/x := (E \times J)/x \times_{E \times J} (E\times I) \rightarrow E \times I \xrightarrow{E \times f} E\times \C
		$$
		internal to $\E/E$ exists.
	\end{enumerate}
	In this case, the left Kan extension $g$ is given on 
	a generalized object $x : E \rightarrow J$ by 
	the colimit
	$$
	(E \times g)(x) = \colim ((E \times I)/x \rightarrow E \times I \rightarrow E\times \C)
	$$
	internal to $\E/E$.
\end{proposition}
\begin{proof}
	Since $\CCAT_\infty(\E)$ is horizontally closed (see \ref{cor.int-ccat-hor-closed}), there exists a proarrow $C : J \rightarrow \C$  of $w^\circledast$-cones under $f$. 
	Let $x : E \rightarrow J$ be a generalized object. Then the functor $E \times (-) = \pi_{E}^\ast$ is proper, see \ref{prop.int-ccat-lfib-proper}. 
	Consequently, $E\times C$ is a proarrow of $(E\times w)^\circledast$-cones under $E \times f$. Hence, it follows from 
	\ref{cor.fib-pke} that $x^\circledast{(E\times C)}$ is the proarrow of conical cones under the composite $(E \times I)/x \rightarrow E \times I \rightarrow E\times \C$.
	The desired conclusions can now be deduced from \ref{thm.comp-conj-topos}.
\end{proof}

The above formula for the left Kan extension also appears in \cite[Remark 6.3.6]{MartiniWolf}. Our next goal 
is to study the internal functors of Kan extensions using the analysis of \ \ref{ssection.pke-arrows}. We will see later that the following equivalent conditions are sufficient for the existence of these:

\begin{lemma}\label{lem.internal-pke-functor}
	Let $\C$ be a category and $w : I \rightarrow J$ be a functor internal to $\E$. 
	Then the following assertions are equivalent:
	\begin{enumerate}
		\item There exists a groupoidal cover $\mathcal{U}$ of $\fun(I,\C)$ so that for every $f : E \rightarrow \fun(I,\C) \in \mathcal{U}$, 
		the left Kan extension of $f$, viewed as a functor $E \times I \rightarrow E\times \C$, along $E \times w$ internal to $\E/E$ exists.
		\item There exists a groupoidal cover $\mathcal{U}$ of $\fun(I,\C)$ so that for every $f : E \rightarrow \fun(I,\C) \in \mathcal{U}$,
		there exists a groupoidal cover $\mathcal{V}_f$ of the category $E\times J$ internal to $\E/E$ with the property that
		the colimit 
		$$
		(E' \times E \times I)/x \rightarrow E' \times E \times I \xrightarrow{E' \times f} E' \times E \times \C 
		$$
		internal to $\E/(E' \times E)$ exists for any $x : E' \rightarrow E\times J$ in $\mathcal{V}_f$.
		\item Assertion (1) holds for any choice of groupoidal cover $\mathcal{U}$ of $\fun(I,\C)$.
		\item Assertion (2) holds for any choice of groupoidal cover $\mathcal{U}$ of $\fun(I,\C)$ and family of groupoidal covers $\{\mathcal{V}_f\}_{f\in \mathcal{U}}$.
	\end{enumerate}
\end{lemma}
\begin{proof}
	Note that the assertions (1) and (2), and the assertions (3) and (4) are equivalent on account of \ref{prop.internal-pke}. Clearly, (3) implies (1) and 
	hence it remains to show that (1) implies (3). So, assume that (1) holds. Let $\mathcal{U}'$ be another 
	cover of $\fun(I,\C)$. We may assume that $\mathcal{U}$ and $\mathcal{U}'$ consist of one generalized element $f : E \rightarrow \fun(I,\C)$ 
	and $g : E' \rightarrow \fun(I,\C)$ respectively.  Consider the internal groupoid $E'' \in \E$ with the defining feature that it sits in a pullback square
	\[
		\begin{tikzcd}
			E'' \arrow[r, "p"]\arrow[d,"q"'] & E \arrow[d,"f"] \\
			E' \arrow[r, "g"] & \fun(I,\C).
		\end{tikzcd}
	\]
	Let $C$ be the proarrow of $(E' \times w)^\circledast$-cones under $g : E' \times I \rightarrow E' \times \C$ in $\CCAT_\infty(\E/E')$. As 
	$q$ is an effective epimorphism, $C$ will be a conjoint if 
	$q^*C$ is a conjoint in $\CCAT_\infty(\E/E'')$ (see \ref{lem.comp-conj-topos-1}). But $q$ is proper (see \ref{prop.int-ccat-lfib-proper})
	and this means that $q^*C$ is the proarrow 
	of $(E'' \times w)^\circledast$-cones under $q^*g \simeq p^*f$ on account of \ref{prop.proper-pres-cones}. By the same argument,  $q^*C$ is the image of the proarrow of $(E\times w)^\circledast$-cones 
	under $f : E \times I \rightarrow E \times \C$ in $\CCAT_\infty(\E/E)$, which was a conjoint by assumption.
\end{proof}

Next, we will now show the existence of functors of Kan extensions under these conditions. This was 
shown before by Shah \cite[Theorem 10.5]{Shah} in case that $\E$ is a presheaf topos, and by Martini and Wolf \cite[Theorem 6.3.5]{MartiniWolf} in the general case.

\begin{proposition}\label{prop.internal-pke-functor}
	Let $\C$ be a category and $w : I \rightarrow J$ be a functor internal to $\E$.  
	If the equivalent conditions of \ref{lem.internal-pke-functor} are met,
	then the internal functor
	$$
	w^* : \fun(J, \C) \rightarrow \fun(I, \C)
	$$
	admits a left adjoint $$w_! : \fun(I, \C) \rightarrow \fun(J, \C).$$

	The value of the left adjoint on generalized elements of $\fun(I,\C)$ can be described as follows. For every groupoid $E \in \E$, 
	the functor $$E \times w_! : \fun_{\E/E}(E \times I, E\times \C) \rightarrow \fun_{\E/E}(E \times J, E \times \C)$$ internal to $\E/E$ carries 
	a functor $f : E \times I \rightarrow E \times \C$ to its left Kan extension along $E \times w$ internal to $\E/E$. 
\end{proposition}
\begin{proof}
	The functor $w^*$ admits a left adjoint if the proarrow
	$$
	w^*_\circledast : \fun(J,\C) \rightarrow \fun(I,\C)
	$$
	is a conjoint in light of \cite[Corollary 6.6]{EquipI}. In light of \ref{thm.comp-conj-topos}, it suffices to check that for any $f : E \rightarrow \fun_\E(I,\C)$, the proarrow 
	$$
	f^\circledast(E \times w)^*_\circledast : \fun_{\E/E}(E \times J, E\times \C) \rightarrow E
	$$
	is a conjoint in $\Cat_\infty(\E/E)$. Here we used that the arrow $E\times w^*$ can be identified 
	with $(E\times w)^\ast$ under the push-pull identifications of \ref{prop.int-ccat-push-pull}. The desired result follows now from \ref{prop.computation-proarrow-lke}.
\end{proof}

\begin{remark}
	In the context of \ref{lem.internal-pke-functor}, note that assertion (2) is met whenever there exists a groupoidal cover $\mathcal{U}$ of $\C$ with the following property:
	for every $x : E \rightarrow J \in \mathcal{U}$, $E \times \C$ admits those colimits internal to $\E/E$ that are shaped by the internal category $(E \times I)/x$. This condition does not depend on $\mathcal{U}$. This is essentially 
	the assumption that the authors of \cite{Shah} and \cite{MartiniWolf} put on $\C$ in their respective versions of \ref{prop.internal-pke}.
\end{remark}

Recall that we have a notion of fully faithful functors internal to $\E$, by interpreting the abstract notion defined in \cite[Definition 6.1]{EquipI} in the equipment $\CCAT_\infty(\E)$. 
This gives the following expected result:

\begin{proposition}
	In the situation of \ref{prop.internal-pke-functor}, the resulting functor $w_!$ will be fully faithful if $w$ is fully faithful.
\end{proposition}

To this end, we will need some analysis of the internal fully faithful functors. In particular, we will show that it coincides Martini's notion \cite[Definition 3.8.1]{Martini}.

\begin{proposition}\label{prop.internal-ff}
	Let $f : \C \rightarrow \D$ be a functor internal to $\E$. Then the following are equivalent:
	\begin{enumerate}
		\item the functor $f$ is fully faithful,
		\item the functor $f$ is representably fully faithful, i.e.\ 
	for every category $X$ internal to $\E$, 
	the functor 
	$$
	f_* : \CAT_\infty(\E)(X, \C) \rightarrow \CAT_\infty(\E)(X, \D)
	$$
	between $\infty$-categories is fully faithful,
		\item condition (2) holds for all groupoids $X \in \E$.
	\end{enumerate}
\end{proposition}
\begin{proof}
	Note that (1) implies (2) on account of \cite[Proposition 6.4]{EquipI}. It is clear that (2) implies (3). 
	To show that (3) implies (1), we note that reflective inclusions of double $\infty$-categories
	reflect cartesian cells, and hence fully faithfulness. Thus a similar argument as \ref{prop.int-ccat-tsdfib-rep} applies 
	so that may reduce to the 
	prototypical case that $\E = \PSh(T)$ for some small $\infty$-category $T$. Then 
	$\CCAT_\infty(\E) \simeq [T^\op, \CCAT_\infty]$. Since cartesian cells 
	are detected point-wise in this cotensor (see \ref{prop.cart-cotensors}), $f$ is fully faithful if and only if $f(t) : \C(t) \rightarrow \D(t)$ is fully faithful. But this can be identified with the post-composition functor $f_*$ at the internal groupoid $X = \map_T(-,t)$, 
	which is fully faithful by assumption.
\end{proof}

\begin{corollary}
	The notion of internal fully faithful functors coincides with that of Martini \cite{Martini}.
\end{corollary}
\begin{proof}
	This follows because characterization (3) of \ref{prop.internal-ff} is precisely characterization (3) of \cite[Proposition 3.8.7]{Martini}.
\end{proof}

\begin{proof}
	On account of \ref{prop.internal-ff}, $w_!$ is fully faithful if and only if for every $E \in \E$, the induced functor 
	$$
	(w_!)_* : \CAT_\infty(\E)(E, \fun(I, \C)) \rightarrow \CAT_\infty(\E)(E, \fun(J, \C))
	$$ 
	is fully faithful. On account of \ref{prop.internal-restr-vertical}, this is left adjoint to the (external) restriction functor $(E\times w)^* :  \CAT_\infty(\E)(E \times J, \C) \rightarrow \CAT_\infty(\E)(E \times I, \C)$, 
	which, by adjunction, is given by
	$
  	(E \times w)^* : \CAT_\infty(\E/E)(E \times J, E\times \C) \rightarrow \CAT_\infty(\E/E)(E \times I, E\times \C).
	$
	The result now follows from Corollary 6.21 and Proposition 6.22 of \cite{EquipI}, because the functor $E \times (-)$ preserves fully faithfulness, hence 
	the functor $E\times w$ is again fully faithful internal to $\E/E$.
\end{proof}

\nocite{*}
\bibliographystyle{amsalpha}
\bibliography{EquipmentsII}

\end{document}